\documentclass[11pt,reqno,nocut]{amsart}

\usepackage[utf8]{inputenc}
\usepackage[T1]{fontenc}
\usepackage{amsmath,amssymb,amsthm,mathtools}

\usepackage{mathrsfs,bm,color}
\usepackage{fancyhdr,lastpage}
\usepackage[shortlabels]{enumitem}
\usepackage{calc}
\usepackage{bbm}
\usepackage{graphicx}
\usepackage{upgreek}
\usepackage{lmodern}
\usepackage[normalem]{ulem}
\usepackage{verbatim}
\usepackage{bbm}
\usepackage{stmaryrd}
\usepackage{amsmath}
\usepackage{amssymb}
\usepackage{dsfont}
\usepackage{amsthm}
\usepackage{hyperref}
\usepackage{array}
\hypersetup{
	colorlinks,
	linkcolor={red!50!black},
	citecolor={red!80!black},
	urlcolor={blue!80!black}
}

\usepackage{xcolor} 
\renewcommand{\limsup}{\varlimsup}
\renewcommand{\liminf}{\varliminf}
\usepackage{graphicx}

\usepackage{amsfonts}%
\usepackage{dsfont}

\allowdisplaybreaks

\usepackage{fullpage}

\numberwithin{equation}{section}

\theoremstyle{plain}
\newtheorem{theo}{Theorem}[section]
\newtheorem{lem}{Lemma}[section]
\newtheorem{cor}{Corollary}[section]
\newtheorem{prop}{Proposition}[section]

\theoremstyle{definition}
\newtheorem{exam}{Example}[section]
\newtheorem{defn}{Definition}[section]
\newtheorem{rem}{Remark}[section]

\newtheorem{ass}{Assumption}[section]

\newcommand{\rdmmat}[1]{\ensuremath{\bm{#1}}}

\newcommand{\vect}[1]{\ensuremath{\boldsymbol{\mathbf{#1}}}}

\newcommand\given[1][]{\:#1\vert\:}

\newcommand{\mmm}{\mathrel{}\mid\mathrel{}}

\newcommand{\bs}{\bm{\sigma}}
\newcommand{\iii}{i_1,\dots,i_p}
\newcommand{\E}{\mathbb{E}}	
\newcommand{\cA}{\mathcal{A}}
\newcommand{\pert}{\mathrm{pert}}

\newcommand{\cL}{\mathcal{L}}

\newcommand{\cM}{\mathcal{M}}

\newcommand{\cI}{\mathcal{I}}
\newcommand{\cC}{\mathcal{C}}

\newcommand{\cF}{\mathcal{F}}

\newcommand{\N}{\mathbb{N}}

\newcommand{\R}{\mathbb{R}}
\newcommand{\Z}{\mathbb{Z}}

\newcommand{\eps}{\epsilon}

\DeclareMathOperator{\pP}{\mathbb{P}}

\DeclareMathOperator{\pQ}{\mathbb{Q}}
\DeclareMathOperator{\1}{\mathbbm{1}}

\newcommand{\bx}{\bm{x}}
\newcommand{\by}{\bm{y}}

\newcommand{\vx}{\vec{x}}
\newcommand{\hbx}{\hat{\bm{x}} }

\newcommand{\MLE}{\mathrm{MLE}}
\newcommand{\mis}{\mathrm{mis}}
\newcommand{\PMLE}{\mathrm{PMLE}}
\newcommand{\cs}{\mathbf{CS}} 
\newcommand{\abs}[1]{\left\lvert#1\right\rvert}
\newcommand{\argmax}{\operatorname*{argmax}}

\newcommand\numberthis{\addtocounter{equation}{1}\tag{\theequation}}
\newcommand{\barbeta}{\bar\beta}

\newcommand{\Var}{\operatorname{Var}}

\renewcommand{\epsilon}{\varepsilon}

\newcommand{\norm}[1]{\left\lVert#1\right\rVert}

\newcommand{\bb}{\bar{\beta}}
\renewcommand{\vx}{\vect{x}}

\begin{document}

\title{Pseudo-Maximum Likelihood Theory for High-Dimensional Rank One Inference}

    \author{Curtis Grant}
    \address[Curtis Grant]{Northwestern University}
    \email{curtisgrant2026@u.northwestern.edu}
    \author{Aukosh Jagannath}
    \address[Aukosh Jagannath]{University of Waterloo}
    \email{a.jagannath@uwaterloo.ca}
    \author{Justin Ko}
    \address[Justin Ko]{University of Waterloo}
    \email{justin.ko@uwaterloo.ca}

	\date{\today}
    
	\begin{abstract}
    We develop a pseudo-likelihood theory for rank one matrix estimation problems in the high dimensional limit. We prove a variational principle for the limiting pseudo-maximum likelihood which also characterizes the performance of the corresponding pseudo-maximum likelihood estimator. We show that this variational principle is universal and depends only on four parameters determined by the corresponding null model. Through this universality, we introduce a notion of equivalence for estimation problems of this type and, in particular, show that a broad class of estimation tasks, including community detection, sparse submatrix detection, and non-linear spiked matrix models, are equivalent to spiked matrix models.
    As an application, we obtain a complete description of the performance of the least-squares (or ``best rank one'') estimator for any rank one matrix estimation problem. 
	\end{abstract}
    \maketitle

\section{Introduction}

Suppose that we are given data in the form of a real, symmetric $N\times N$ matrix, $Y\in \R^{N \times N}$, whose entries are conditionally independent given an unknown vector $\vect{x}_0 \in \Omega^N\subseteq\R^N$ and where each entry of $Y$ has conditional law 
\[
	Y_{ij} \sim \pP_Y\Big(\cdot \vert  \frac{\lambda}{\sqrt{N}} x^0_i x^0_j\Big) \quad \text{for } i\leq j \, , 
\]
for some $\lambda>0$, and $\Omega\subseteq \R$ is compact.\footnote{The scaling assumption here in $N$ matches the regime where non-trivial high-dimensional effects, such as the BBP phase transition, occur. It guarantees that both the operator norm of $Y$ and $\frac{1}{\sqrt{N}} (\vect{x^0}) (\vect{x^0})^\intercal$ are of the same order.} 
Our goal is to infer $\vect{x}_0$. 

High-dimensional rank one estimation tasks  with structure form one of the central classes of problems in high-dimensional statistics. This data model captures a broad range of problems that have received a tremendous amount of attention in recent years, such as sparse PCA \cite{zou2006sparse}, $\Z_2$ synchronization \cite{javanmard2016phase}, submatrix localization \cite{bhamidi_submatrix,hajek2017information}, matrix factorization \cite{johnstone2009consistency}, community detection \cite{Blockmodel}, biclustering \cite{biclustering}, and non-linear spiked matrix models \cite{pennington_nonlinear_2017} among many others.

From a statistical perspective, a substantial literature on these problems has emerged over the past decade, particularly from the perspective of hypothesis testing and Bayesian inference. The fundamental limits of hypothesis testing have been explored in \cite{ahmed_hypothesis}. The fundamental limits of Bayesian inference, specifically computing the mutual information of and characterizing the performance for the (matrix) minimum mean-squared error estimator, has been explored in  \cite{lelarge2017fundamental,barbiertensor,Reeves_tensor, MourratXia-tensor, MourratXiaChen-tensor,dominguez2024mutual}. More generally, the setting of ``mismatched'' Bayesian inference was developed in \cite{Camillimismatch,barbier2021performance,barbier2022price,farzadmismatch,barbier_mismatch_nonsymAMP,nonbayes}.  From an algorithmic perspective, various algorithms (along with performance guarantees) for specific problems and estimators---including the MMSE---have been introduced in recent years using the frameworks of approximate message passing \cite{montanari_AMP,rangan2012iterative,deshpande2014information,lesieur2015mmse}, spectral methods \cite{PerryPCA,mergny_ko_colt}, semi-definite programs \cite{SINGER_angular, singer_semidefinite, semidefinite_relaxations}, low-degree methods \cite{montanari2022equivalence}, and the sum-of-squares hierarchy \cite{hopkins2015tensor,hopkins2016fast}. 

A natural question is to understand the statistical performance of more general optimization based procedures, such as maximum likelihood estimation (MLE), maximum a posteriori (MAP) estimation, or best low rank approximations.  The literature for these methods, however, is far more sparse. To our knowledge, to date, there have been a sharp understanding only the case of the MLE in sparse PCA \cite{jagannath_tensorPCA} as well as variational inference  for $\Z_2$-synchronization \cite{fan2021tap,celentano2023local}. 

We seek here to close this gap. To this end, observe that many popular optimization based estimators for such problems, such as those mentioned above, can be interpreted as pseudo-likelihood methods \cite{PMLE}. In this paper, we provide a unified analysis of the performance of pseudo-likelihood methods. 

We develop a pseudo maximum likelihood theory for rank one inference tasks in the high-dimensional regime for when the latent vector, $\bx^0$, is structured. We provide exact variational formulas for the asymptotic pseudo-likelihood and, as a direct consequence, obtain exact variational characterizations for the performance of the corresponding estimators. See Section~\ref{sec:varcharac_main}.

We find that these problem exhibit ``universal'' behaviour in that these variational characterizations depend only on four scalar quantities, which we call the \emph{information parameters}. These parameters encode certain Fisher-type information of the pseudolikelihood with respect to a ``null'' model and are reminiscent of the score parameters appearing in the classical regime \cite{fisherinformation}. 

Surprisingly, we find that if one of these information parameters, which we call the \emph{score parameter}, is not zero, then it entirely dictates the effectiveness of our inference method, and the effect is typically catastrophic. We refer to such models as \emph{ill-scored} models. We present here a data-driven approach to systematically correct for this effect and obtain a corresponding variational characterization for the performance of this \emph{score-corrected} method. See Section~\ref{section:scored-corrected}.

Since a given inference task is entirely characterized by its information parameters, our analysis yields two general notions of equivalence of inference tasks, called strong and coarse equivalence. For example, we give a precise sense in which the problem of maximum likelihood estimation for certain spiked matrix models and the stochastic block model are equivalent. See Section~\ref{sec:coarse_equiv_main}.

We then illustrate our results with a broad range of examples.
First, we present a complete analysis of the performance of the popular ``best rank 1 approximation'' procedure \cite{Eckart1936lowrank}.   We also provide a method to correct for some of these issues, by introducing the score-corrected least squares procedure. Surprisingly, however, we find that in natural problems, such as a sparse Rademacher matrices, the best rank one approximation and its score-corrected version are necessarily completely uninformative. Indeed, we provide a sufficient condition for the failure of such methods.
Finally in Section~\ref{sec:exam}, we illustrate how our approach can be used to analyze a broad range of problems and methods. Specifically, we study popular inference methods for spiked matrix models, $\Z_2$-synchronization, the stochastic block model, sparse Rademacher matrices, non-linear transformations of spiked matrix models, sparse PCA, and Poisson-Bernoulli matrices. 

Let us pause here to discuss the technical tools involved in our work and how they compare to the above mentioned literature. Since the latent vector is structured, standard tools of high-dimensional statistics, such as concentration of measure or random matrix theory,  are unable to yield a sharp understanding of these problems. To circumvent this, the  recent progress in the past decade has used deep connections to statistical physics, specifically to the theory of spin glasses. In particular, the central insight is that hypothesis testing and Bayesian inference of matrix models are deeply connected to the Sherrington-Kirkpatrick model \cite{SK,parisi1979infinite,talagrand2006parisi,PUltra} (and its relatives) in a special regime called the ``Nishimori Line'' as a consequence of Bayes theorem \cite{nishimori_line}. 

With this in mind, it is natural that optimization-based procedures have been less understood: on the ``Nishimori line'' the corresponding spin glass model is in the so-called ``replica symmetric phase''. While deeply challenging, this regime is comparatively simpler to understand as the corresponding variational problems reduce to optimizing functions of one real variable \cite{lelarge2017fundamental,barbier_replica_matrix}. To understand more general optimization methods, such as maximum likelihood estimation, one must use the recently developed tool, called the \emph{method of annealing} to understand the ``zero-temperature'' asymptotics of spin glasses \cite{ground_state_chen,ground_state_jagannath}. Here the corresponding model enters the so-called ``replica symmetry breaking'' phase and can exhibit deep and nontrivial structure \cite{AT_line,Toninelli_2002_RSB,PUltra}.

The key technical insight is that one can view pseudolikelihood inference as a ``zero temperature'' asymptotic of mis-matched Bayesian inference. We can then combine the recent analysis of such problems developed by one of us and co-authors in \cite{nonbayes}  with the $\Gamma$-convergence based ``method of annealing'' approach developed by one of us and co-authors in \cite{JagSen}. The combination of these works is non-trivial and several new techniques are utilized. To deal with ill-scored models we generalize the universality result of \cite{nonbayes} and remove the simplifying assumption \cite[Hypothesis~2.3]{nonbayes} (see Appendix~\ref{sec:univ}) and prove the analogous universality statement for pseudo maximum likelihood estimation (see Appendix~\ref{sec:devvarI}). Ill-scored models introduce an additional mean parameter that has to be controlled, so we use the techniques developed in  \cite{nonbayes} and prove a generalized variational formula for ill-scored models. We also prove new regularity results for the variational formulas with respect to general reference measures, extending the results in \cite{specgap}, which were previously only done for the uniform measure on $\{ \pm 1 \}$ (see Appendix~\ref{AP:discrete-convergent}). Lastly, we extend the work of \cite{JagSen} to allow for random initial conditions (see Appendix~\ref{sec:devvarII}).

	\section{Variational Characterization for Pseudo-Likelihood Estimation}\label{sec:varcharac_main}
	\subsection{Data model and assumptions}
	Suppose that we are given data in the form of a real, symmetric $N\times N$ matrix $Y\in \R^{N \times N}$, whose upper entries are conditionally independent given an unknown vector $\vect{x}^{0,N} \in \Omega_0^N \subseteq \R^N$ with law 
    \begin{equation}\label{eq:conditional_data_dist}
        Y_{ij} \sim \pP_Y(\cdot \given[]  \frac{1}{\sqrt{N}} x^{0,N}_i x^{0,N}_j) \quad \text{for } i\leq j \, , 
    \end{equation}
	for some $\lambda>0$. Here $\Omega_0\subseteq \R$ is a compact set. Our goal is to infer $\vect{x}^{0,N}$.

	We will assume throughout that the laws of $Y_{ij}$ are jointly absolutely continuous with respect to either Lebesgue measure on $\R^{N\times N}$ or a product of counting measures on $\Omega_0$, so that the conditional densities are well-defined. In the following, we denote to the underlying Lebesgue or counting measure by $dy$. (The meaning of notation will be clear from context.) We denote the log-likelihood of a single coordinate as $g_0(y,w)$, i.e.,
	\[
	g_0(y,w) = \log \pP_Y(Y=y \mid w )\, , 
	\]
	and the log-likelihood of $Y$ given $x \in \mathbb{R}^N$ is 
	\[
	\mathcal{L}^{g_0}_N(Y,x) =   \sum_{i \leq j } g_0 (Y_{ij}, \frac{ x_i x_j}{\sqrt{N}} ) \, .
	\]
	We will further assume that there is a null model whose likelihood we denote by $g_0(y,0)$, and we denote the corresponding null measure by $\pP_0$. The null model corresponds to the case $\vect{x}^{0,N}=0$. Under the assumptions above a maximum likelihood estimator is defined as: 
	\[
	\hat{\mathbf{x}}_{\MLE} =  \arg \max_{x \in \Omega_0^N} \mathcal{L}^{g_0}_N(Y,x) \, .
	\]
	Note that, at this level of generality, this estimator may not be uniquely defined.
 
	We are also interested in understanding the pseudo-likelihood. Here we allow for misspecification of both the likelihood function and the support of the unknown vector $\Omega_0$. In this case we denote the pseudo-likelihood by $g$ and the parameter set by $\Omega$. Throughout we shall denote our pseudo maximum likelihood estimator by $\hat{ \mathbf{x}}_{\PMLE}$, and it is given by: 
	\begin{align} \label{eq:def-of-mismatch-mle}
		\hat{\mathbf{x}}_{\PMLE} := \arg \max_{x \in \Omega^N}  \sum_{i \leq j} g (Y_{ij}, \frac{ x_i x_j}{\sqrt{N}} )  \, ,
	\end{align}
	which again may not be uniquely defined.
We measure the performance of the estimator by its cosine similarity with the unknown vector, that is: 
	\begin{align} \label{eq:def-cosine-similarity}
		\lim_{N \to \infty} \cs( \hat{\mathbf{x}}, \mathbf{x}^{0,N} ) \qquad \text{where} \qquad \cs (x,y) := \frac{ x \cdot y }{\norm{x} \norm{y} } \,,
	\end{align}
    and its squared norm.

	In order to develop a high-dimensional theory, we need certain basic assumptions on the data distribution, and the unknown vector.   Note that as $\Omega_0$ is compact, for any sequence $\vx^{0,N}$, the sequence of empirical measures
    $$\mu_{\vx^{0,N}}=\frac{1}{N}\sum_i \delta_{x_i^{0,N}},$$
    is always tight.  
    \begin{defn}
        We say that a sequence $\vx^{0,N}\in\Omega_0^N$ is \emph{tame} if $\mu_{\vx^{0,N}}\to \mathbb Q$ weakly for some probability measure $\mathbb Q$.
    \end{defn}

\begin{rem}\label{rem:unboundedsignal}
Instead of assuming $\Omega_0$ is compact, one may assume that $\mathbf{x}^{0,N}$ is a vector of i.i.d subgaussian random variables, with no assumption on them being bounded. All theorem statements hold with this assumption instead (see Appendix~\ref{sec:subgauss_signal}). 
    \end{rem}
We work under the assumption that $\vx^{0,N}$ is tame.
  This assumption is common in the high-dimensional statistics literature (see, e.g., \cite{montanari_AMP,zhou_nonrotationallyinvar,PerryPCA}).
	Next we need some basic regularity assumptions on the (pseudo)-likelihood. To this end, we need the following function class
	\begin{defn}
		Let $\cF_0(dy)$ denote the set of pairs of measurable functions, $(f_1(y,w)$, $f_2(y,w))$, with common domain $\R \times U\subseteq \R^2$, where $U$ is an open neighbourhood of $0$, that are three times continuously differentiable in $w$ for every $y$ and satisfy the following four conditions:
		\begin{equation}\label{eq:regularityg0}
        \begin{aligned}
			\int_{\Omega_0} \exp(f_i(y,0))dy, \qquad &\int_{\Omega_0} [ \lvert\partial_w f_i(y ,0) \rvert^4 ]\exp(f_1(y,0))dy,\\
            \|\partial_w^{2}f_i(\cdot ,0)\|_\infty,~ \qquad&\|\partial_w^{3}f_i(\cdot,\cdot)\|_\infty <\infty .
        \end{aligned}
		\end{equation}
    for each $i=1,2$.
	\end{defn}
    
	\noindent ($\cF_0$ in principle depends on the choice of $U$, whose choice is problem dependent. We suppress this dependence for the sake of exposition.) 
	We will assume throughout the following that the pair of the likelihood, $g_0$, and pseudo-likelihood $g$ are in this class, $(g_0,g)\in \cF_0(dy)$. Since the pair $(g_0,g)$ completely specify the underlying inference problem, we refer to this pair as an \emph{inference task}.

	\subsection{Information parameters}
	One of our central results is that the performance of the (pseudo)-likelihood estimator in problems of this class is \emph{entirely} determined by the \emph{information parameters} of the pair $(g_0,g)$, which are defined as follows
	
	\begin{defn}
		The \emph{information parameters} of an inference task $(g_0,g)$ are 
		\begin{align}
			{\beta_1(g_0, g)} &=  \E_{\pP_0} {\big[ (\partial_w g(Y,0) - \E_{\pP_0} [ \partial_w g(Y,0) ] )^2 \big]}\label{eq:fisherscore1}\\
			\beta_2(g_0,g) &=  \E_{\pP_0} \big[ \partial_w g(Y,0) \partial_w g_0(Y,0) \big] \label{eq:fisherscore2}\\
			{\beta_3(g_0,g)} &= {-} \E_{\pP_0} \big[  \partial_{w}^{2} g(Y,0) \big] \label{eq:fisherscore3}\\
	\beta_4 &= \E_{\pP_0} \big[ \partial_w g(Y,0) \big]\,. \label{eq:fisherscore4}
		\end{align}
	\end{defn}
	
	{The information parameters measure the effect of the misspecification on the null model. Let us briefly discuss the statistical interpretation of $\bar\beta=(\beta_1,\beta_2,\beta_3,\beta_4)$. The meaning of $\beta_4$ is discussed in the next subsection, we begin first with the interpretation of $\beta_1,\beta_2,\beta_3.$
    
    First observe that if we denote the null Fisher information by
	\begin{equation}\label{eq:fisherscore}
		{\beta_*} = \E_{\pP_0} \big[ (\partial_w g_0(Y,0))^2 \big],
	\end{equation}
	then when $g=g_0$,
	the information parameters satisfy the \emph{Rao relation} (or Bartlett identities)
	\[
\beta_*=\beta_1=\beta_2=\beta_3,
	\] 
    by standard properties of score functions \cite[Chapter~2.3]{schervish_theoryofstatistics}.
    Observe that $\beta_1$ is the variance of the pseudo-score under the null model. Observe that $\beta_2$ is, in some sense, a measure of the alignment of the underlying noise with the model of the noise under the psuedo-likelihood In particular, the ratio $\beta_2/\sqrt{\beta_1\cdot\beta_*}$ is the cosine similarity of the pseudo-score with the true score. Finally note that $\beta_2+\beta_3$ is, in some sense, measuring the failure of the integration-by-parts formula relating the two forms of the Fisher information matrix.}

 \subsection{Well-scored v.s. ill-scored pseudolikelihoods}

A classical fact is that the expected score of $g_0$ is $0$. As we are allowing for the case that $g\neq g_0$, however, this identity may no longer hold. As we shall see below, the failure of this identity has substantial repercussions for inference. To this end, it helps to introduce the following criterion.

	\begin{defn}[well-scored pseudo-likelihood]\label{def:regular}
		We say that a pseudo-likelihood function $g(y,w)$ is \emph{well-scored}
		if its score satisfies
		\begin{align} \label{eq:consistent-estimator-def}
			\beta_4 = \E_{\pP_0} \big[ \partial_w g(Y,0) \big]  = 0.
		\end{align}
		Otherwise we call it \emph{ill-scored}.
	\end{defn}
	
 The case of well-scored models represents an ideal case for pseudo-maximum likelihood theory. On the contrary, in the ill-scored setting, the pseudo-maximum likelihood is heavily influenced by the sign of the score parameter, $\beta_4$, and can lead to complete failure of pseudo maximum likelihood estimation (see Section~\ref{sec:varformulairregular}).

\subsection{ Variational
characterization of performance for well-scored PMLEs (and MLEs)}

We are now in the position to state our main results. We begin by discussing the case of well-scored models. Our first main result is a characterization of the asymptotic pseudo-maximum likelihood and a corresponding characterization of the asymptotic performance of pseudo-maximum likelihood estimators. 

In what follows we shall use the notation $M_N(x)$ and $S_N(x)$ for the overlap with the true signal, and the self overlap. Namely we set:
\begin{equation}\label{eq:SMnotation}
    M_N(\vx) = \frac{1}{N}\vx\cdot\vx^{0,N}\qquad \text{ and }\qquad S_N(\vx) = \frac{1}{N}\norm{\vx}^2 \,.
\end{equation}
Observe that $\cs(\vx,\vx^{0,N})=M_N(\vx)/\sqrt{S_N(\vx)S_N(\vx^{0,N}})$, and under the assumption that $\vx^{0,N}$ is tame, that $S_N(\vx^{0,N})\to \E_\mathbb Q [x_0^2]$ almost surely. 

Our first main result is an explicit formula for the maximum pseudo-likelihood when restricted to a specific value of overlap and self overlap. More precisely, for $\epsilon>0$, let us define:
\begin{equation}\label{eq:defnOmegaepsilon}
\Omega^N_\epsilon(S,M) :=\{x \in \Omega^N: \abs{M_N(x)-M}\leq \epsilon, \abs{S_N(x)-S}\leq \epsilon\}
\end{equation}
We then have the following. (We follow the convention that the maximum of a function over an empty set is $-\infty$.)
\begin{theo} \label{thm:main-pointwise}
Suppose that $\vx^{0,N}$ is tame and that the pair $(g_0,g)$ is well-scored with information parameters $\barbeta$. For every $(S,M)\in \R^2$, we have that 
    \[
    \lim_{\epsilon\to 0}\lim_{N\to\infty} \frac{1}{N} \max_{x \in \Omega_{\epsilon}^N(S,M)} \left[  \cL_N^{g}(Y,x) - \sum_{i \leq j} g(Y_{ij},0) \right]  = \psi_\beta(S,M) \qquad a.s.
    \]
Here $\psi_{\barbeta}$ is an explicit, Holder continuous function given by \eqref{eq:def-of-psi} below, depending only on $\barbeta,\Omega,$ and $\pQ$.
\end{theo}
This theorem states that there is an exact formula for the constrained pseudo-maximum likelihood, which allows us to probe the pseudolikelihood landscape geometry. (For the necessity of the centring term, see Remark~\ref{rem:centre} below.) The precise definition of $\psi_{\barbeta}$ is technically demanding: it involves an infinite-dimensional, strictly convex variational problem which itself involves a certain Hamilton-Jacobi-Bellman equation. As such, we have deferred its precise definition to Section \ref{sec:definition-of-parisi-functional} so that we can more gently introduce this formula for readers less familiar with spin glass methods. 

As an immediate consequence of this, we have the following result which is a variational formula for the limiting maximum pseudo-likelihood and a variational characterization for the performance of the corresponding pseudo MLEs. In the following, we let $\cC$ denote the effective domain of $\psi_{\barbeta}$. (A more concrete description is given in \eqref{defC} below.) We note here that $\cC\subseteq\R^2$ is compact and convex and depends only on $\pQ$ and $\Omega$.
\begin{theo}\label{thm:main_regular}
		Suppose that $\vx^{0,N}$ is tame and that the pair $(g_0,g)$ is  well-scored with information parameters $\bar\beta$. Then 
		\begin{equation}\label{eq:main-limit}
			\frac{1}{N}\max_{x \in \Omega^N} \left[ \mathcal{L}^g_N(Y,x)  - \sum_{i \leq j} g(Y_{ij},0) \right] \to \sup_{(S,M) \in \mathcal{C}} \psi_{\bar{\beta}}(S,M)   \qquad  a.s.
		\end{equation}
        Furthermore, if $\mathcal{C}_{\bar \beta}$ denotes the collection of maximizers of $\psi_{\bar \beta}$, then        
        for any sequence of choices of $\hat\vx_{\PMLE}^N$, the corresponding sequence of overlaps
\[
(S_N(\hat\vx_{\PMLE}),M_N(\hat\vx_{\PMLE})) \, , 
\]
is tight, with limit points contained in $\cC_{\bar \beta}$. 
	\end{theo}

\begin{rem}\label{rem:centre}
		The centring term $ \sum_{i \leq j} g(Y,0)$ does not depend on $x$, so it will not affect the pseudo-maximum likelihood estimator. However, this normalization term needs to be subtracted off for the maximum pseudo-likelihood to have a well-defined limit. For example, take the case of spiked Gaussian matrices, $Y=W+\lambda vv^T$, where $W_{ij}$ is a Wigner random matrix with entries that are i.i.d. standard Gaussian. Here for the Gaussian (pseudo)likelihood one has:
		\[
		\frac{1}{N} \mathcal{L}^{g_0}_N(Y,x)  = \frac{1}{N} \sum_{i\leq j} - \frac{1}{2} W_{ij}^2 + \frac{x^0_i x_j^0 x_i x_j}{N}  - \frac{x_i^2 x_j^2}{2N} \, ,
		\]
		and this quantity diverges as $\frac{1}{N} \sum_{ij} g_0(Y,0) = \frac{1}{N}\sum_{ij} \frac{1}{2} W_{ij}^2 = \Theta(N) $ with high probability.
	\end{rem} 

    To conclude the section,  observe that in the above, we do not guarantee the convergence of cosine similarity. This is because, in some settings, $\cC_\beta$ may contain several points. This is due to the existence of many near optimizers of the pseudo maximum likelihood. It is natural to ask under which regimes one has true convergence. A sufficient condition is if $\cC_\beta$ consists of at most two points. 

    \begin{ass}\label{hyp:existintro}
		Suppose that $\bar \beta$ is such that $\cC_{\bar \beta}$ consists of at most two points. Furthermore, the coordinate associated with the $m$ parameter is unique up to a sign.
	\end{ass}

\begin{cor}
    Suppose that $\bar \beta$ satisfies Assumption~\ref{hyp:existintro}, then for any sequence of pseudo-likelihood estimators, $(S_N(\hat \vx_{\PMLE}^N),|M_N(\hat \vx_{\PMLE}^N)|)\to(s,m)$ almost surely.  In particular, the absolute cosine similarity converges almost surely to $m/\sqrt{s \E_\mathbb Q x_0^2}$
\end{cor}

	\subsection{Behaviour of Ill-scored pseudolikelihoods} ~\label{sec:ill-scored-behavior} 
	Before turning to our variational characterization in the case of ill-scored models, let us pause here for a discussion of the key issue in this setting. 
	Suppose that 
	\[
	\beta_4 = \E_{\pP_0} [ \partial_w g(Y,0) ] = C \neq 0 \, .
	\]
	When $\beta_4 \neq 0$,  the leading order behaviour of the pseudo-likelihood is dominated by the empirical mean of the parameter. Roughly speaking, in this regime one has the expansion
 \[
 \max_{x \in \Omega_N} \cL_N^g(Y,x) = N^{3/2}C \max_{x \in \Omega_N}\bigg( \frac{1}{N }\sum_{i = 1}^N x_i\bigg)^2 +  o(N^{3/2}).
 \]
 Note that the leading order term here does \emph{not} depend on the unknown parameter, $x_0$, and is an order of magnitude larger than in the well-scored setting (c.f. Example~\ref{ex:illscored-spiked-matrix} and Lemma~\ref{lem:lemUniv}). 

 This has catastrophic consequences on inference which are best illustrated by way of example.
	\begin{exam}\label{ex:illscored-spiked-matrix}
		Consider the data matrix generated from a spiked Gaussian matrix with non-zero mean in the null model,
		\[
		Y_{ij} \sim \mathcal{N}( C + \frac{x_i^0x_j^0}{\sqrt{N}}, 1 ).
		\]
		To infer $\vect{x}_0$, we take an irregular pseudo likelihood from a centred Gaussian likelihood,
		\[
		g(y,w) = -\frac{1}{2} (y - w)^2,
		\]
		which represents a large misspecification of an order $1$ parameter of the data model. 
		It follows that
		\[
		\beta_4 = \E_{\pP_0} [ \partial_w g(Y,0) ] = \E_{\pP_0} [Y] = C.
		\]
		When $C > 0$, then the data distribution $Y$ has a large positive eigenvalue of order $N$ while the positive eigenvalue of the spike we want to infer $ \frac{x_i^0x_j^0}{\sqrt{N}}$ is of order $\sqrt{N}$. Conversely, if $C < 0$, then the data distribution $Y$ has a large negative eigenvalue of order $N$ while the positive eigenvalue of the spike we want to infer $ \frac{x_i^0x_j^0}{\sqrt{N}}$ is of order $\sqrt{N}$. In either case, there is large but spurious shift in the likelihood that obscures the parameter we want to infer. 
	\end{exam}

In the following sections, we begin by first stating the variational formula in the case of ill-scored pseudolikelihoods. Importantly, however, the statistician does not, a priori, have access to the underlingly null distribution $\pP_0$. As such it is important to understand whether or not it is possible to determine if one is in the ill-scored scenario and, in particular, if it is possible to systematically correct for this effect. 
    We present such an approach in the subsequent section.
	
	\subsection{Variational Formula for Ill-scored pseudolikelihood}\label{sec:varformulairregular}
	
	We now state an extension of Theorem~\ref{thm:main_regular} to these ill-scored models. Due to the importance of the sign of $\beta_4$, the results will be separated into cases. We begin with the case of $\beta_4>0$.

\begin{theo}[Positive $\beta_4$]\label{th:posscore}
		Suppose that $\vx^0$ is tame and $\beta_4 > 0$. Let $\tilde x_+ = \sup \Omega$ and $\tilde x_- = \inf \Omega$. 
         Let $$x_+ = \begin{cases}
			\tilde x_+ & \text{ if } | \tilde x_+| \geq |\tilde x_-|\\
			\tilde x_- & \text{ if } | \tilde x_+| < |\tilde x_-|.
		\end{cases}
		$$ We have $ \hat{\mathbf{x}}_{\PMLE} = x_+ \vect{1}$ is a constant vector, so the set of limit points of $S_N(\hat x_{PMLE})$ and $M_N(\hat x_{\PMLE})$ is unique and given by $\cC_{\bar \beta} = \{ ( x_+^2, x_+ \E_{\pQ}(x^0) ) \}$. In particular,
		\begin{align*}
			\cs ( \hat{\mathbf{x}}_{\PMLE}, \mathbf{x}_0 )  \to \frac{ \E_{\pQ} (x^0) }{( \E_{\mathbb{Q}} (x^0)^2)^{1/2} } \qquad \text{and}\qquad
			\frac{1}{N}\norm{\hat{x}_{PMLE}}^2\to x_+^2 \quad a.s.
		\end{align*}
\end{theo}
	
\noindent Evidently if $\bar{x}_0 = 0$  then $\cs ( \hat{\mathbf{x}}_{\PMLE}, \mathbf{x}_0 ) = 0$ and the estimator is useless.

	The case when $\beta_4 < 0$ is more delicate since there is a large, spurious signal induced by the misspecification which is, in some sense, in the \emph{opposite} direction of the vector we want to infer.

  In the following, we denote the effective domain of $\psi_{\barbeta,-}$ by $\cC_-$. (See \eqref{def:C-} for an alternative characterization.) In the ill-scored setting, we will also denote the set of corresponding maximizers by $\cC_{\barbeta}$.

	\begin{theo}[Negative $\beta_4$]\label{th:negscore}
		Suppose that $\vx^{0,N}$ is tame and $\beta_4 < 0$. Let 
		$x_-$
		denote the point in the convex hull of the parameter space, $\mathrm{conv}(\Omega)$, that closest to the origin.
		The maximum pseudo-likelihood satisfies
		\begin{equation}
			\max_{x \in \Omega^N} \frac{1}{N} \left[ \mathcal{L}^g_N(Y,x) - \sum_{i \leq j}g(Y_{ij},0)  - \sqrt{N} (x_-)^2 \beta_4 \right]  \to \sup_{(S,M) \in \mathcal{C}} \psi_{\bar{\beta},-}(S,M, x_-)   \, 
		\end{equation}
		almost surely. Furthermore, for any sequence of choices of $\hat\vx_{\PMLE}^N$, the corresponding sequence $(S_N(\hat\vx_{\PMLE}),M_N(\hat\vx_{\PMLE}))$ is tight with limit points contained in $\cC_{\bar \beta}$. Here $\psi_{\barbeta,-}$ is an explicit, Holder continuous function given by \eqref{eq:psiminus} below, depending only on $\barbeta,\Omega$, $\pQ$. 
	\end{theo}
        \noindent Similarly to the case with positive score, if $\bar{x}_0 = c\vect{1}$ and $x_- = 0$ then $\cs ( \hat{\mathbf{x}}_{\PMLE}, \mathbf{x}_0 ) = 0$ and the estimator is useless.

	\subsection{The score-corrected pseudolikelihood}
	~\label{section:scored-corrected} 
	As seen in the previous section, ill-scored pseudolikelihoods have behaviour dictated by the sign of $\beta_4$ which introduces a very large uninformative ``spike'' in the models, which can lead to a complete failure of the inferential procedure in certain scenarios. To this end, we propose a correction to the pseudo-likelihood estimator that resolves this issue.

	A natural way to deal with the high order term when $\beta_4 \neq 0$ is to introduce an additional term to the pseudo-likelihood to centre the corresponding score by using an estimate of the score parameter, $\beta_4$. 
    A priori, the statistician does not have access to $\pP_0$. That said, for any pseudo-likelihood, $g$, we can consider the naive estimator,
	\[
	\hat \beta_4 = \frac{1}{N^2} \sum_{i , j= 1}^{N} \partial_w g(Y_{ij}, 0).
	\]
    If we let $\vect{\bar{x}_0}=\frac{1}{N}\sum x^0_i$, then
	by the law of large numbers, this quantity will concentrate around its expected value
	\begin{align*}
		\hat \beta_4 
        = \E_{\pP_0} [ \partial_w g(Y_{ij}, 0) ] +\frac{2}{N^2}\beta_2 \sum_{i\leq j} \frac{x^0_i x^0_j}{\sqrt{N}}  + O(N^{-1})
		= \beta_4 + \beta_2 \frac{\vect{\bar x_0}^2}{\sqrt{N}} + O(N^{-1}) \, ,
	\end{align*}
	where the lower order terms $\beta_2 \frac{\vect{\bar x_0}^2}{\sqrt{N}}$ come from the fact that we can estimate $\E_Y [ \partial_w g(Y_{ij}, 0) ] $ using the data distribution. (See Lemma~\ref{lem:approxcorrected} for a precise statement.) While nominally, the second term is a lower order effect, this lower order term will have a nontrivial contribution when multiplied by $N^{3/2}$, i.e., the appropriate power of $N$ to counteract the expected score. Thus  $\E_{\pP_0} [ \partial_w g(Y_{ij}, 0) ] $ unfortunately remains inaccessible.
	
	To account for this, let us introduce a hyper-parameter $\alpha \in \R$ and define the corresponding score-corrected pseudo-likelihood by
    \begin{equation}\label{eq:correctedPMLE}
        \mathcal{L}^{g}_{N,\alpha}(Y,x) =   \sum_{i \leq j } g (Y_{ij}, \frac{\lambda x_i x_j}{\sqrt{N}} ) - N^{\frac{3}{2}} \hat \beta_4 \bar x^2 + N \alpha \bar x^2.
    \end{equation}
	The centering by $-  N^{\frac{3}{2}} \hat \beta_4 (\bar x)^2$ kills off the large effect induced by score parameter, while $\alpha$ is a ridge correction term to offset the lower order terms in the score approximation $\hat \beta_4$. If $\alpha = \beta_2 \E_{\pQ} [x_0]^2$ then pseudo maximum likelihood estimation on the score-corrected likelihood is equivalent to optimizing the pseudo-likelihood 
	\[
	\mathcal{L}^{g}_{N,\alpha}(Y,x) =   \sum_{i \leq j } g (Y_{ij}, \frac{\lambda x_i x_j}{\sqrt{N}}) - N^{\frac{3}{2}} \beta_4 \bar x^2.
	\]
	\begin{rem}
		One might also consider a slightly generalized version of the score corrected pseudo likelihood,
		\[
		\mathcal{L}^{g}_{N,\gamma}(Y,x) =   \sum_{i \leq j } g (Y_{ij}, \frac{\lambda x_i x_j}{\sqrt{N}} ) - N^{\frac{3}{2}} \gamma \bar x^2.
		\]
		If we take $\gamma = \beta_4$, then this will also remove the adverse effect caused by non-zero score. However, the scaling of the correction term is order $N^{3/2}$, so  that $\gamma$ must be calibrated to within $o(N^{\frac{1}{2}})$ of $\beta_4$ to avoid introducing lower order corrections. 
	\end{rem}
	
	We have the following variational formula for the score-corrected pseudo-maximum likelihood. Let
	\begin{align}
		\psi_{\mathbf{\beta},\alpha}(S,M,v) = \psi_{\mathbf{\beta},-}(S,M,v) - \frac{\beta_2 [\E_{\pQ} x_0 ]^2 v^2}{2} + \frac{\alpha v^2}{2},
	\end{align}
    where $\psi_{\mathbf{\beta},-}$ is defined below in \eqref{eq:psiminus}.
	Note that the information parameters $\beta_1, \beta_2, \beta_3$ are defined with respect to $g$ and not $g_c$. The effective domain of this function is also $\cC_-$. Let 
    \[
    \cC_{\beta,\alpha} = \argmax_{(S,M,v)\in \cC} \psi_{\beta,\alpha}(S,M,v).
    \]
	We then have the following.
	\begin{theo}\label{th:corrected}
		Suppose that $\vx^0$ is tame. The maximum score-corrected pseudo-likelihood satisfies
		\begin{equation}
			\max_{x \in \Omega^N}  ( \mathcal{L}^g_{N,\alpha}(Y,x) - \sum_{i \leq j}g(Y_{ij},0) ) \to \sup_{(S,M,v) } \psi_{\bar{\beta},\alpha}(S,M,v)   \, ,
		\end{equation}\label{eq:psi-beta-alpha-def}
		almost surely. Furthermore, for any sequence of choices of $\hat\vx_{\PMLE}^N$, the corresponding sequence $(S_N(\hat\vx_{\PMLE}),M_N(\hat\vx_{\PMLE}))$ is tight with limit points contained in $\cC_{\beta,\alpha}$. 
	\end{theo}
	
	\begin{rem}
		If $\alpha = \beta_2 \E_{\pQ}[x_0]^2$, then the variational formula is equivalent to a regular model with information parameters $\beta_1, \beta_2, \beta_3$.
	\end{rem}
{
\subsection{A variational representation for $\psi_{\barbeta}$} \label{sec:definition-of-parisi-functional}
We now return to the precise formulae for $\psi_{\barbeta}$ and $\psi_{\barbeta,-}$. 
The contents of this section can be skipped upon a first reading, as the precise definition of $\psi_{\barbeta }$ are not required to understand the results in sections \ref{sec:coarse_equiv_main} and \ref{sec:least-squares} at a formal level. For those unfamiliar with spin glass theory, the terms appearing in this section can be hard to parse. We will explain in Section~\ref{sec:interpretation} below the interpretation of each of these terms in terms of certain "universal statistical tasks".

	Let $\cM([0,S])$ denote the space of non-negative, finite measures on $[0,S]$ equipped with the weak-* topology, and let $\cA_S\subseteq \cM([0,S])$ be the subset of measures defined as follows
	\[
	\cA_S = \{ \gamma \in \cM([0,S]): \gamma = m(s)ds+c\delta_S, m(s)\geq 0 \text{ non-decreasing}, \ c\geq 0 \}.
	\]
	For each $\gamma=mds+c\delta_S\in\cA_S$, let $\Phi_{\gamma,\lambda,\mu}(t,y)$ denote the weak solution to the Hamilton-Jacobi-Bellman equation,
	\[
	\begin{cases}
		\partial_t \Phi_{\gamma,\lambda,\mu} = - \frac{\beta_1}{4} [ \partial_{y}^2 \Phi_{\gamma,\lambda,\mu} + m(t) ( \partial_y \Phi_{\gamma,\lambda,\mu} )^2 ] & (t,y)\in [0,S]\times \R \\
		\Phi_{\gamma,\lambda,\mu}(S,y) = \max_{x \in \Omega} ( yx + \lambda x x^0 + \big(\mu+\frac{\beta_1}{2} c\big)x^2  ) \, & y\in \R.
	\end{cases}
	\]
	For the notion of weak solution for partial differential equations (PDEs) of this type see, e.g., \cite{JagTob16} and for the existence, uniqueness and regularity of weak solutions to this PDE see \cite[Appendix~A]{JagSen}.
	
	Let us now define the functional $\psi_{\bar\beta}$,
	\begin{align} \label{eq:def-of-psi} 
		\psi_{\barbeta}(S,M) =  \inf_{\gamma,\lambda,\mu}  \E_{\mathbb{Q}}[\Phi_{\gamma,\lambda,\mu}(0,0)] - \frac{\beta_1}{2} \int_0^S td\gamma(t) - \mu S - \lambda M  + \frac{\beta_2 M^2}{2}  - \frac{\beta_3 S^2}{4}   .
	\end{align}
Observe that $\psi_\beta$ is well-defined on $[0,(\max \Omega)^2]\times[\min\Omega,\max\Omega]$ and upper semi-continuous  there, though it may take the value $-\infty$. The effective domain of $\psi_{\barbeta}$ is the set
	\begin{equation}\label{defC}
		\mathcal{C}=\cap_{\rho,\tau \in [-1,1]^{2}}\{(S,M):  \E_{\pQ}[\inf_{x \in \Omega}\{\rho x^{2}+ \tau x x^{0}\}] \le \rho S+\tau M \le  \E_{ \pQ} [\sup_{x \in \Omega}\{\rho  x^{2}+ \tau x x^{0}\}]\}.
	\end{equation}
	Observe that the set $\cC$ is convex and compact and depends implicitly on $\pQ$, the limiting empirical measure of the signal. We now turn to defining $\psi_{\barbeta,-}$.

	\begin{equation}
    \begin{aligned}
		\psi_{\bar \beta,-}&(S,M,v) =\\  &\inf_{\mu,\lambda,\rho,\gamma} \big( \E_{\mathbb{Q}}[\Phi_{\gamma,\lambda,\mu}(0,\rho)] - \frac{\beta_1}{2} \int_0^S td\gamma(t) - \mu S - \lambda M - \rho v + \frac{\beta_2 M^2}{2}  - \frac{\beta_3 S^2}{4}  \big). \label{eq:psiminus}
	\end{aligned}
    \end{equation}
        Notice that the $\psi_{\bb,-}$ defined here differs from \eqref{eq:def-of-psi} only by shifting the spatial coordinate of $\Phi$.  
    
        The effective domain of $\psi_{\bar \beta,-}$ is the set
	\begin{equation}\label{def:C-}
    \begin{aligned}
		\cC_- =\cap_{\rho,\tau,\eta \in [-1,1]^{3}}\{ & (S,M,v):  \E_{ \pQ}[\inf_{x \in \Omega}\{\rho x^{2}+ \tau x x^{0} + x \eta\}] \\ &\le \rho S+\tau M + \eta v\le  \E_{\pQ} [\sup_{x \in \Omega}\{\rho  x^{2}+ \tau x x^{0} + x \eta\}]\}.
            \end{aligned}
	\end{equation}
        Furthermore, in the context of illscored models, we let $\cC_{\barbeta}$ denote the set of maximizers of $\psi_{\beta}$ given in \eqref{eq:psiminus} over the set $\cC$ defined in \eqref{def:C-} subject to a constraint on the third coordinate, that is
        \begin{equation}\label{def:C-set}
             \cC_{\bar \beta} = \argmax_{(S,M,v)\in \cC_- : v = x_-} \psi_{\barbeta,-}(S,M,x_-).
        \end{equation}
       
    }
\section{Strong and coarse equivalence of inference tasks and a universal task} \label{sec:coarse_equiv_main}
It is natural to ask if two pseudo likelihoods lead to estimators that are, from a statistical perspective,  equivalent. For example, in the spiked matrix model, while the top eigenvector obtains a nontrivial cosine similarity with the ground truth, any other unit vector with the same cosine similarity has the same performance with respect to the underlying statistical task.
It turns out that our results lead to an even deeper notion of equivalence between pseudolikelihood estimation \emph{problems}. For example, there is a precise sense in which maximum likelihood estimation of the ``spike'' in spiked matrix models is ``equivalent'' to maximum likelihood estimation of the communities in stochastic block models!

Our first notion of equivalence is \emph{strong equivalence}.
\begin{defn}
    We say that two inference tasks are \emph{strongly equivalent} if they have the same information parameters.
\end{defn}
Evidently strong equivalence is an equivalence relation. Furthermore, there is a natural universal statistical task corresponding to given information parameters which is defined as follows.

Let $g^1$ denote a pseudo likelihood whose information parameters with respect to $g_0$ are given by  $\bar \beta = (\beta_1,\beta_2,\beta_3,\beta_4)$. We consider the corresponding inference task with likelihoods given by 
 \begin{align}
     g_{U,0}^{\bar \beta}(y,w) &= -\frac{1}{2\beta_1} (y-\beta_2 w - \beta_4)^2 - \frac{1}{2} \log(2 \pi \beta_1) \, , \\ 
     g_{U,1}^{\bar \beta}(y,w) &= - \frac{1}{2} (y-w)^2 - \frac{\beta_3-1}{2} w^2 \, ,
 \end{align}
 which corresponds to least squares estimation with a \emph{correction}. The universal statistical task corresponds to estimating the spike in the matrix 
 \[
Y=  G+ \frac{\beta_2}{\sqrt{N} } x^0 (x^0)^T \, ,
 \]
 where $G$ has i.i.d  $\mathcal{N}(\beta_4,\beta_1)$ entries, via the pseudo-likelihood $g_{U,1}^{\bar \beta}$.

\begin{theo} \label{thm:master-problem}
Any inference task $(g_0,g_1)$ with information parameters given by $\bar \beta$ is strongly equivalent to the inference task $(g_{U,0}^{\bar \beta},g_{U,1}^{\bar \beta} )$.  
\end{theo}

\begin{rem}
Theorem ~\ref{thm:master-problem} simplifies greatly in the well scored case with $\beta_3>0$. In this case $g_{U,1}^{\bar \beta}$, may instead be taken to be $-\frac{1}{2} (y-\sqrt{\beta_3} w)^2$, with an appropriate normalization in $g_{U,0}^{\bar \beta}$. 

\end{rem}

An important consequence of our work is that there is in fact a substantially weaker notion of equivalence that captures the underlying statistical task. Recalling the statistical interpretation of $\cC_\beta$ from Theorem ~\ref{thm:main_regular},~\ref{th:posscore},~\ref{th:negscore},~\ref{th:corrected} 
as the set of near optimal overlaps of estimators in the respective problems, we are led to the following natural notion.
    \begin{defn} 
    We say that two inference tasks $(g_0^1,g^1)$ and $(g_0^2,g^2)$ are \emph{coarsely equivalent} if $\cC_{\bar \beta^1} = \cC_{\bar \beta^2}$ where $\bar{\beta}^i = \bar{\beta}(g_0^i,g^i)$ for $i=1,2$ are their corresponding information parameters.
    \end{defn}
    Notice that if $g_1$ and $g_2$ have the same information parameters with respect to $g_0^1,g_0^2$ then they are coarsely equivalent. More generally, one has the following sufficient conditions for coarse equivalence of well-scored pseudolikelihoods.
	\begin{theo}\label{th:robust}
		Consider two well-scored inference tasks $(g_0^1,g^1)$ and $(g_0^2 ,g^2)$, with information parameters $\bar \beta(g^1)=\bar{\beta}(g_0^1,g^1)$ and $\bar \beta(g^2)=\bar{\beta}(g_0^2,g^2)$. Suppose that at least one of the following conditions are true
		\begin{itemize}
			\item The ratio of all the information parameters is constant
			\begin{equation}\label{eq:homo4}
				\frac{\beta_1(g^1)}{\beta_1(g^2)} = \frac{\beta_2(g^1)}{\beta_2(g^2)} =  \frac{\beta_3(g^1)}{\beta_3(g^2)},
			\end{equation}
			\item There exists a constant $C$ such that the parameter space $\Omega$ satisfies $\abs{\bx} = C$ for every $\bx \in \Omega$ and the first ratio of the two information parameters are equal
			\begin{equation}\label{eq:homo2}
				\frac{\beta_1(g^1)}{\beta_1(g^2)} = \frac{\beta_2(g^1)}{\beta_2(g^2)} ,
			\end{equation}
		\end{itemize}
		then $(g_0^1,g^1)$ and $(g_0^2,g^2)$ are coarsely equivalent.
	\end{theo}

In the case of ill-scored pseudolikelihoods one must further include a condition on the correction parameters used: 
    
	\begin{theo} \label{cor:model-equiv-corrected }
		Consider two ill-scored inference tasks $(g_0^1,g^1)$ and $(g_0^2,g^2)$ with information parameters $\bar \beta(g^1)=\bar{\beta}(g_0^1,g^1)$ and $\bar \beta(g^2)=\bar{\beta}(g_0,g^2)$, and let  $\alpha^1$ and $\alpha^2$ be the  correction parameters for $g^1$ and $g^2$ respectively. Suppose that at least one of the following conditions are true
		\begin{itemize}
			\item The ratio of all the information and correction parameters are constant
			\begin{equation}\label{eq:homo1}
				\frac{\beta_1(g^1)}{\beta_1(g^2)} = \frac{\beta_2(g^1)}{\beta_2(g^2)} =  \frac{\beta_3(g^1)}{\beta_3(g^2)} = \frac{\beta_4(g^1)}{\beta_4(g^2)} = \frac{\alpha^1}{\alpha^2},
			\end{equation}
			\item There exists a constant $C$ such that the parameter space $\Omega$ satisfies $\abs{\bx} = C$ for every $\bx \in \Omega$ and the first ratio of the the information and correction parameters are equal
			\begin{equation}\label{eq:homo3}
				\frac{\beta_1(g^1)}{\beta_1(g^2)} = \frac{\beta_2(g^1)}{\beta_2(g^2)} = \frac{\beta_4(g^1)}{\beta_4(g^2)} = \frac{\alpha^1}{\alpha^2} ,
			\end{equation}
			
		\end{itemize}
		then $(g_0^1,g_1)$ and $(g_0^2,g_2)$ with correction parameters $\alpha^1$ and $\alpha^2$ are coarsely equivalent.
	\end{theo}

\begin{rem}
    While Theorem ~\ref{thm:master-problem} guarantees a measure of the performance of the pseudo likelihood $g_1$ in terms of the performance of a least squares problem, it does not mean that the performance of $g_1$ is equivalent to the performance of least-squares for the initial matrix $Y$.  In fact, we show in Section ~\ref{sec:least-squares} that the least square estimator can be completely uninformative regardless of the SNR used. 
\end{rem}
{
\subsection{Interpretation of the variational formulas}\label{sec:interpretation}
We end this section with a brief discussion about the interpretation of the terms in the variational formula \eqref{eq:def-of-psi}. 
Variational formulas of this type have a long history in the statistical physics literature and are now called Parisi formulas, after the landmark work of Parisi \cite{parisi1979infinite}. The connection is as follows. For simplicity, let us discuss the well-scored scenario.

As we have seen above, the inference tasks we consider are coarsely equivalent to a universal statistical task of the form of a (constrained) spiked Wigner model.
Observe that the objective function for this problem is of the form  (modulo a constant shift and lower order terms)
\begin{equation}\label{eq:objective}
\cL_N(x)\approx H_N(x)+ \frac{N \beta_2}{2} M_N(x)^2 - \frac{N \beta_3}{4}S_N(x)^2  \, , 
\end{equation}
where
\[
H_N(x) = \frac{\sqrt{\beta_1}}{\sqrt{N}} \sum_{i \leq j} g_{ij} x_i x_j \, .
\]
The term
$H_N$ is a generalization of the Sherrington-Kirkpatrick model. For concreteness, let us call $H_N$ the ``Hamiltonian'' and the remaining terms the ``constraint functions''. The above can thus be viewed as the Lagrangian for the constrained optimization problem of computing the minimizer of the Hamiltonian, $H_N$, subject to constraints on the pair $(M_N,S_N)$.

We can now interpret the terms in \eqref{eq:def-of-psi} in turn: The first two terms--namely $\min_\gamma \E_{\pQ} \Phi-\frac{\beta_1}{2}\int td\gamma$---can be viewed as the limit of the value of the term $H_N$, and the remaining four terms are the constraint functions. 
}

	\section{Application to Gaussian Pseudolikelihoods (the best rank 1 approximation)} \label{sec:least-squares}
        
	A popular approach to tackling rank one estimation problems is to consider the \emph{best rank 1 approximation} (equiv., the least squares approach). That is, consider a vector, $x\in \Omega^N$, such that
	\[
	\vect{\hat x}_{LS}(\lambda) = \arg \min_{x \in \Omega^N} \frac{1}{2} \| \rdmmat{Y} - \frac{\lambda \vect{x} \vect{x}^\intercal}{\sqrt{N}} \|_F^2,
	\]
	where $\lambda>0$ is a scale hyper-parameter. Observe that this corresponds to pseudolikelihood estimation with a Gaussian likelihood $g(y,w)=- \frac{1}{2\sigma^2} (y-\lambda w)^2$ where $\sigma^2 = \E_0[Y^2]$. 
    Let
    \[
    \beta_{LS} = \frac{\E[Y\partial_w g_0(Y,0)]}{\sigma}.
    \]
    We then have the following.

	\begin{prop}\label{th:leastsquares}
		The pair $(g,g_0)$ has information parameters $(\lambda^2, \lambda\beta_{LS}, \lambda^2,\lambda \E_0 Y)$. In particular $g$ is well-scored if and only if $\E_0 Y=0$.
		
	\end{prop}

Let us pause to consider the case that $g$ is well-scored and $\beta_{LS}>0$. In this case Theorem~\ref{thm:main_regular} applies. In particular, the corresponding overlap and squared norm have limit points lying in $\cC_\beta$. If $\lambda= \sigma \sqrt{\beta_{LS}}$ then the information parameters satisfy the Rao relation. Note that by Cauchy-Schwarz, $\beta_{LS}\leq\sqrt{\beta_0}.$ If, furthermore, $\beta_{LS}=\sqrt{\beta_0}$ then information parameters are equal to those of the log-likelihood.

In practice, we do not necessarily know that the data distribution under the null model has zero mean, and, as shown in Example \ref{ex:illscored-spiked-matrix} above, a seemingly innocuous misspecification can lead to substantial effects on inference.
 In order to counteract these potential effects, we introduce a score-corrected best rank 1 approximation (as in Section \ref{section:scored-corrected}) by subtracting off the mean of the data distribution and adding a ridge term.

Let
\[
\rdmmat{\bar  Y} = \frac{1}{N^2} \sum_{i \leq j} Y_{ij} \, ,
\]
and consider the \emph{score-corrected least squares estimator}
\[
	\vect{\hat x}_{LS,\alpha}(\lambda,\alpha) = \arg \min_{x \in \Omega^N} \frac{1}{2} \| \rdmmat{Y} - \frac{\lambda \vect{x} \vect{x}^\intercal}{\sqrt{N}} \|_F^2 + N^{3/2} 	\rdmmat{\bar  Y} \bar x^2 - N \alpha \bar x^2,
\]
where $\lambda>0$ is a scale parameter. To offset the correction term, we set 
\[
\alpha = \beta_2 \E_{\pQ} [x_0]^2 = \lambda \beta_{LS} \E_{\pQ} [x_0]^2 \, ,
\]
and then Theorem~\ref{th:corrected} applies. 

It is interesting to note that the above gives the following important negative result.
\begin{prop} \label{prop:least-square-failure}
If $\beta_{LS}=0$ and $\vx^{0,N}$ is tame with
$\E_\mathbb{Q} x^0 =0$, then
\[
\cs(\hat \vx_{\PMLE},\vx)\to 0 \qquad a.s.
\]
\end{prop}
This occurs, for example, in the case of sparse Rademacher matrices. See Section ~\ref{ex:sparse-rademacher} and Table ~\ref{tab:example} below.

 { \subsection{Gaussian Pseudolikelihoods over $\R^N$ with sub-Gaussian Signals} \label{sec:LSentirespace} 

It is natural to wonder about how our results generalize to settings where the signal can have un-bounded entries. Our results straight-forwardly extend to the case where the signal distribution has i.i.d. sub-Gaussian entries. (See Remark~\ref{rem:unboundedsignal}.) However, the situation is more nuanced when it comes to the domain of the pseudolikelihood. Indeed, the (true) log-likelihood of our data model \eqref{eq:conditional_data_dist} is not generally well-defined on all of $\R^N$. This happens, e.g., in problems with discrete structure such as  community detection (see Example~\ref{example:community-detection}). As such, in order to extend our results to unbounded pseudolikelihood domains, one must take an approach that depends on the specific properties of a given task and specific pseudolikelihood. 

To illustrate how our results extend to this more general setting, we consider here the case of Gaussian pseudolikelihood estimation with iid sub-Gaussian signals. This is well-defined for all problems of the form \eqref{eq:conditional_data_dist}.
 
In this case, the relevant object to look at is
\[
	\hat{\mathbf{x}}_{\PMLE} =  \arg \max_{ \vect{x} \in \R^N} - \frac{1}{2} \norm{Y- \frac{\lambda \mathbf{x} \mathbf{x}^T}{\sqrt{N}} }_F^2 \, 
\]
for some $\lambda>0$. The following theorem gives a complete characterization of the performance of the least squares estimator in terms of the null distribution. We write $g_{\lambda}$ for the centred Gaussian log-likelihood with SNR $\lambda>0$.

\begin{theo} \label{thm:least-squares-all-space} 
    Suppose that $(g_0,g_{\lambda})$ is well scored with parameters $\barbeta$, and that $x^{0,N}$ has i.i.d. subgaussian entries. Then: 
    \begin{align*}
\lim_{N \to \infty} &\frac{1}{N} \max_{x \in \R^N} \left[ -\frac{1}{2} \norm{Y- \frac{\lambda \mathbf{x} \mathbf{x}^T}{\sqrt{N}} }_F^2  + \frac{1}{2} \norm{Y}_F^2\right]  
=
\begin{cases}
    \frac{1}{\beta_3 \beta_2^2 \E_{\pQ} [x_0^2]^2 } (\E_{\pQ} [x_0^2]^2 \beta_2^2 + \beta_1)^2 &\text{if} \ \frac{\beta_2 \E_{\pQ} [x_0^2]}{\sqrt{\beta_1}} > 1 \\ 
    \frac{4}{\beta_3} &\text{otherwise} 
\end{cases}
\, ,
    \end{align*}
    and furthermore the estimator $\mathbf{\hat x}_{\MLE} $ satisfies: 
    \[
    \cs(\hat \vx_{\PMLE},\vx)\ \to \begin{cases}
       \sqrt{ 1 - \frac{\beta_1}{\beta_2^2 \E_{\pQ} [x_0^2]^2 } }  &\text{if} \ \frac{\beta_2 \E_{\pQ} [x_0^2] }{\sqrt{\beta_1}} > 1 \\ 
    0 &\text{otherwise} 
    \end{cases}
    \]
\end{theo}

As a consequence of Proposition \ref{th:leastsquares}, the value of $\barbeta$ is explicit in terms of the null distribution, and Proposition \ref{prop:least-square-failure} still holds true when $\beta_2=\beta_{LS} =0 $, however we see that the regime where one can guarantee failure of recovery via least-squares is in general much larger.

Theorem~\ref{thm:least-squares-all-space} should be viewed as a generalization of the BBP transition \cite{baik2005phase} for principal component analysis. This has been previously understood in \cite{Benaych-Georges_spiked, peche_spike, renfrew_soshnikov_spiked}. 
We also take a moment to point out that the theory of PCA for spiked matrices is not universal in the case of heavy-tailed distributions (in fact the behaviour of the spectrum is wildly different, see \cite{arous2008spectrum}), and since this is a special case of maximum likelihood estimation, we do not expect universality to hold without at least some assumptions of the form in ~\eqref{eq:regularityg0}. }

	\section{Examples}\label{sec:exam}
In this section, we outline several explicit models which fall into our framework. We summarize some of the models we consider, and their corresponding likelihoods and information parameters in the table below.

    \begin{table}[h!] 
        \centering
        		\begin{tabular}{| m{6cm} | m{6cm} | m{0.7cm} | m{0.7cm} | m{0.7cm} | m{0.4cm} |} 
			\hline
			Model Type & Likelihood & $\beta_1$ & $\beta_2$ & $\beta_3$ & $\beta_4$  \\
			\hline 
			\text{Spiked  Wigner with} SNR $\lambda_0$   
    & $-\frac{1}{2} (y-\lambda w +C)^2 - \frac{1}{2} \log (2\pi) $ & $\lambda^2$ & $\lambda \lambda_0$ & $\lambda^2$ & $C$ \\ 
			\hline 
			\text{Community Detection} $(\frac{1}{2},\mu_0)$ & $y \log( \frac{1}{2} + \mu w) + (1-y) \log( \frac{1}{2} - \mu w)$ & $ 4\mu^2 $ & $ 4 \mu \mu_0 $ & $4\mu^2$ & 0 \\
			\hline 
			\text{Sparse Rademacher}  & $-\frac{1}{2} (y-\lambda w)^2 - \frac{1}{2} \log(2\pi)  $ & $\lambda p$ & $0$ & $\lambda^2$ & 0 \\ 
			\hline  
			\text{Signs of Spiked} Wigner Matrix with SNR $\lambda_0$ & \vspace{0.1cm}

 		$ \frac{(1-y)}{2} \log \frac{1}{\sqrt{2\pi}} \int_{-\infty}^{-\lambda w} e^{-\frac{x^2}{2}} dx $
        $+ \frac{(1+y)}{2} \log \frac{1}{\sqrt{2\pi}} \int_{-\lambda w}^{\infty} e^{-x^2/2} dx $

   & $\frac{2}{\pi} \lambda^2$ & $\frac{2}{\pi} \lambda^2$ & $\frac{2}{\pi} \lambda^2$  & 0 \\ 
   \hline
			\text{Well Specified} Sparse PCA  & \vspace{0.1cm}

 		$-\frac{1}{2}(y-\lambda w)^2 - \frac{1}{2} \log (2\pi)  $

   & $\lambda^2$ & $\lambda^2$ & $ \lambda^2$  & 0 \\ 
   \hline
   		\text{Poisson-Bernoulli}Matrices  & \vspace{0.1cm}

 		$-\log(2) + \abs{y} \log(\lambda+w) - \log(\abs{y}!) - \lambda -w $

   & $\frac{1}{\lambda}$ & $ \frac{1}{\lambda} $ & $ \frac{1}{\lambda}$  & 0 \\ 
   \hline
		\end{tabular}
        \vspace{0.2cm}
		\caption{This table lists the information parameters of several well studied inference problems that fall under our framework. These examples are described in detail in Section~\ref{sec:exam}.}
        \label{tab:example}
    \end{table}

	\subsection{Spiked Matrices and $\Z_2$ synchronization} 
Suppose that we want to recover an unknown vector $\vect{x}^0$  that has been corrupted with additive Gaussian noise $G_{ij} \sim N(0,1)$ at signal to noise ratio $\lambda_0$. That is,
	\[
	Y = G + \frac{\lambda_0}{\sqrt{N}} x_0 x_0^\intercal \, .
	\]
    In the case that $\bx^0$ has $\{ \pm 1 \}$ valued entries, this is known as the $\Z_2$ synchronization problem. This special case has been studied extensively (see~\cite{lesieur_constrained_2017,PerryPCA,miolane2018phase,celentano2023local,barbier2020information})  .
    
    { It is instructive to compare our variational characterization to the characterization of the MMSE in these problems due to \cite[Theorem 1]{lelarge2017fundamental}. Note that the latter provides a variational representation for the mutual information between $Y$ and $x_0x_0^\top$ in terms of a one-dimensional optimization problem on $\R_+$, from which the MMSE is obtained by differentiation. By contrast, ours involves an infinite dimensional minimization problem for the constrained pseudo-ML, and thus a max-min representation for the pseudo-ML.
    This is because in the (well-specified) Bayesian setting, the corresponding statistical physics model is in the so-called ``replica symmetric'' phase, whereas risk-based inference typically involves the ``replica symmetry breaking'' phase.
    
    }

    In this case,  the log likelihood of any coordinate is given by: 
	\[
	g_0(y,w) = -\frac{1}{2}(y-\lambda_0 w)^2 - \frac{1}{2} \log (2\pi) \, .
	\]
	Suppose  that we have misspecified the signal to noise ratio $\lambda \neq \lambda_0$, and we build statistical estimators from the following misspecified spiked matrix model
	\begin{equation}\label{eq:missmatrixmodel}
		Y = G + \frac{\lambda}{\sqrt{N}} x
        _0 x_0^\intercal \, , 
	\end{equation}
	in other words, we assume that the log likelihood is given by: 
	\[
	g_\lambda(y,w) = -\frac{(y - \lambda w)^2}{2} - \frac{1}{2}\log(2 \pi) .
	\]
	The information parameters for are given by
	\begin{align} \label{eq:spike-matrix-fisher-scores}
		\beta_0    = \lambda^2_0 \, ,  ~\beta_1(\lambda) = \lambda^2 \, , ~\beta_2(\lambda) = \lambda \lambda_0 \, , ~\beta_3(\lambda) =  \lambda^2
	\end{align}
	where we recall that $\beta_0$ is the true information parameter associated with the correctly specified model.

       \begin{figure}[h!]
        \centering
        \includegraphics[width=0.49\linewidth]{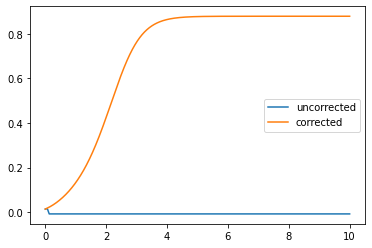}
        \includegraphics[width=0.49\linewidth]{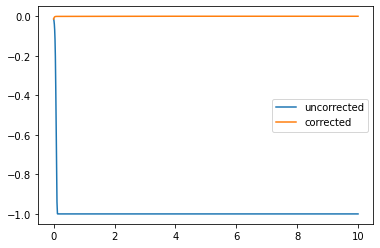}
        \caption{The cosine similarity in the spiked matrix problem with Rademacher latent variable and noise with mean $1$ solved using corrected and uncorrected least squares. A data matrix of size $2500 \times 2500$ and the uncorrected and corrected likelihoods were optimized using gradient descent. The left plot displays the cosine similarity against the true signal for the corrected and non-corrected estimators plotted against the number of steps. It is clear that correcting the likelihood gets rid of the effect from the score parameter, and the corresponding pseudo-maximum likelihood estimator achieves a non-zero cosine similarity. The right plot shows the correlation with the spurious all one's vector. The uncorrected model rapidly correlates with the spurious vector while the corrected model decorrelates with it.}
        \label{fig:corrected_rademacher}
    \end{figure}

	\subsection{Stochastic Block Model}
	\label{example:community-detection} We now consider a community detection problem with two groups. We work with the stochastic block model SBM$(n, \frac{1}{2} + \mu_0 N^{-1/2}, \frac{1}{2} - \mu_0 N^{-1/2} )$ on two communities. In this model we shall assume that our unknown signal $\bx^0$ lies in $\{ \pm 1 \}^N$, and serves as the index vector for the two communities. The corresponding data matrix is the adjacency matrix, and its entries have distribution given by: 
	\begin{align*}
    \pP (Y_{i,j} = 1 \given \frac{x_i x_j}{\sqrt{N}}) = \frac{1}{2} + \mu_0 \frac{x_i x_j}{\sqrt{N}}\ \quad \text{ and }  \quad
	\pP (Y_{i,j} = 0 \given \frac{x_i x_j}{\sqrt{N}} ) = \frac{1}{2} - \mu_0 \frac{x_i x_j}{\sqrt{N}} \, . 
	\end{align*}
	The parameter $\mu_0 > 0$ represents the difference between the probability of edges appearing within and outside of each group. Notice that when $x_i,x_j$ take the same sign, the probability is higher, and when $x_i x_j$ take different signs then the probability of connecting an edge is lower. The $\sqrt{N}$ scaling is such that the detection problem becomes non-trivial, and a phase transition on the weak recovery of the groups is observable (see \cite{lesieur_constrained_2017}). { A rigorous proof of the MMSE in this particular scaling regime can be found in \cite[Section~2.5]{koinhomo}, and some numerical analysis of the MMSE and algorithmic questions related to attaining the MMSE are discussed in detail in \cite[Section~5.15]{lesieur_constrained_2017}.}
    
    More generally the stochastic block model has been studied in a wide variety of regimes for the connection probabilities between communities (see ~\cite{Abbe} for a detailed overview of different regimes).  There is a large collection of literature concerned with showing when different notions of recovery of the communities is possible, as well as when there are efficient algorithms for recovery. See ~\cite{Abbe-Bandeira,hopkins2017power,miolane2018phase,hajek2018recovering,mayya2019mutual,dominguez2024mutual} and the references therein. 
    
    For the SBM with connection probabilities as above, the loglikelihood is given by
	\[
	g_0(Y,w) = Y \log ( \frac{1}{2} + \mu_0 w) + (1 - Y)  \log (\frac{1}{2} - \mu_0 w).
	\]
	Suppose that a signal to noise ratio $\mu_0 \neq \mu$ is chosen, that is, we choose the pseudo-likelihood 
	\[
	g(Y,w) = Y \log ( \frac{1}{2} + \mu w) + (1 - Y)  \log (\frac{1}{2} - \mu w) \, ,
	\]
	 then the information parameters are given by
	\[
	 \beta_1 = 4 \mu^2, \beta_2 = 4\mu\mu_0, \beta_3 = 4\mu^2, \beta_4 = 0 \, ,
	\]
    and the Rao relation is not satisfied. We note however that $\beta_4=0$ and so our choice of pseudo-likelihood is well scored. 

     One method to introduce an ill-scored pseudo-likelihood is to work with an incorrect assumption on the null-model. If we suppose that null model corresponds to the adjacency matrix of a $G(n,p)$ matrix with $p \neq 1/2$, that is the pseudo-likelihood is given by: 
     \[
 g_p(y,w) = y \log( p + \mu w ) + (1-y) \log (1-p - \mu w) \, ,
     \]
     and a direct computation yields:
     \[
\beta_4^{g_p} =  \frac{ \mu (1-2p)}{2p(1-p)} \, ,
     \]
     which is zero if and only if $p=1/2$.

	\subsection{Sparse Rademacher Matrices and Best Rank-1 approximation}  \label{ex:sparse-rademacher}
	
    We now consider a class of sparse submatrix detection problems \cite{sparse_submatrix}.	For this example, we suppose that our unknown vector lies in $\Omega^N$ where $\Omega$ is either an interval $[a,b]$ or a finite set.
    
	Consider the case where $Y$ is a sparse Rademacher matrix, i.e $Y$, conditionally on $w$, takes values in $\{- 1,0,1\}$ with probabilities given by:
	\begin{align} 
		\pP \Big( Y_{ij} = \pm 1 | \frac{x_i x_j}{\sqrt{N}} \Big ) = 
		\frac{p}{2}+ \lambda \frac{x_i x_j}{\sqrt{N}} \, , \quad \text{and} \quad \pP \Big(Y =0 | \  \frac{x_i x_j}{\sqrt{N}} \Big) = 1-  p -2\lambda \frac{x_i x_j}{\sqrt{N}}   \, ,
	\end{align}
    where throughout $p$ is a fixed number in $(0,1)$. 
	In this case the log-likelihood is given by: 
	\[
	g_0(Y,w) = (1-Y^2) \log(1-p-2\lambda w) + \frac{Y(Y-1)}{2} \log\big(\frac{p}{2}+\lambda w\big) - \frac{Y(1+Y)}{2} \log\big( \frac{p}{2} + \lambda w\big)  ,
	\] 
	and the corresponding score parameters are given by: 
	\[
	\beta_1= \beta_2=\beta_3  =  \frac{4 \lambda^2}{1-p} + \frac{\lambda^2}{p}  \, .
	\]
	Suppose now that we try to infer the unknown vector via the best rank 1 approximation, that is, we try to minimize
    \[
    \min_{x\in{\Omega}^N} \|Y - \lambda \frac{\vx \vx^T }{\sqrt{N}} \|_{F}^2 \, ,
    \]
then as discussed in Section ~\ref{sec:least-squares}, the corresponding estimator $\vx_{LS}$ corresponds to a pseudo likelihood estimator $\hat \vx_{\PMLE}$ with pseudo-likelihood given by:     
	\[
	g(Y,w) = -\frac{1}{2} (Y- \lambda w)^2 - \log(2\pi)  \, . 
	\]
    By Proposition ~\ref{th:leastsquares} the model is well scored, and furthermore, an explicit computation shows the Fisher parameters for the gaussian equivalent are given by: 
\[
	\beta_1 = \lambda p, \beta_2 = 0 , \beta_3 =  \lambda^2 \, ,
\]
and consequently by Proposition~\ref{prop:least-square-failure} the least-square estimator is completely uninformative provided that the limiting empirical measure of $\vx^{0,N}$  is balanced. 

 \begin{figure}[h!]
    \centering
    \includegraphics[width=0.5\linewidth]{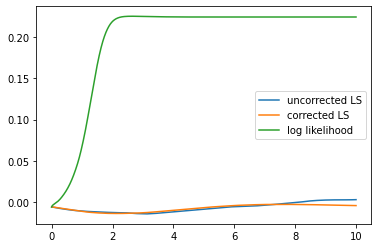}
    \caption{The cosine similarity in the sparse Rademacher problem for corrected and uncorrected least squares, and the log-likelihood. A data matrix of size $2500 \times 2500$ was used, and corrected and uncorrected least squares was performed on the data matrix and optimized using gradient descent. The plot displays the cosine similarity plotted against the number of steps. It is clear that the least squares estimator was uninformative and always achieved a cosine-similarity of zero. However, when gradient descent was performed on the log-likelihood, we see that maximum likelihood estimator achieved non-trivial performance. This demonstrates the failure of least squares in some problems that can be solvable by MLE.}
    \label{fig:sparse_rademacher}
\end{figure}
    
	\subsection{Non-Linear transformations of rank $1$ matrices } 

 Consider a data vector $x \in \{ \pm 1 \}^N$, and a spiked Wigner matrix $W$ given by: 
 \[ 
 W = G + \frac{\lambda}{\sqrt{N}} xx^T \, ,
 \]
 where $G$ is a symmetric matrix with i.i.d standard Gaussian entries.  From $W$ we consider the transformation taking each entry $W_{ij}$ and sending them to $Y_{ij}=F(W_{ij})$ for some function $F$. Non-linear transformations of random matrices have applications in to kernel methods \cite{romain_couillet_kernel_2016, ElKaroui10Kernel, ElKarouiSpikeKernel} and the spectra of one-layer neural networks \cite{pennington_nonlinear_2017, piccolo21, Louart18RMTtoNN}. The spectrum of $F(W_{ij})$ was thoroughly analyzed in \cite{guionnet2023spectral} and \cite{feldman2023spectral}. { A complete characterization of the MMSE and the information theoertic limit in these problems can be found in  \cite[Section~3.2]{mergny_ko_colt}. Similarly to the spiked matrix problem, the variational representations are also in terms of a one-dimensional optimization problem over $\R_+$.}
 
 From the matrix $Y$ we will study the behaviour of maximum likelihood estimation for certain choices of $F$. We remark that some choices of $F$ will lead to irregular likelihoods that do not fall into our framework. We provide an example in Section ~\ref{sec:squaring}.

 \subsubsection{Rounded Entries:}
 
 Suppose that $F(x) = \text{sgn}(x)$ (with the convention  $\text{sgn}(0)=1$). This is the censored spiked matrix model that was studied recently in \cite{kunisky2024low}. In this case the likelihood of the output matrix $Y$ is given by:
 \[
g(y,w) = \frac{(1-y)}{2} \log \frac{1}{\sqrt{2\pi}} \int_{-\infty}^{-\lambda w} e^{-\frac{x^2}{2}} dx + \frac{(1+y)}{2} \log \frac{1}{\sqrt{2\pi}} \int_{-\lambda w}^{\infty} e^{-x^2/2} dx \, .
 \]

 We may explicitly compute the $\beta$ values in this case:  
\[
\beta_1=\beta_2=\beta_3 = \frac{2}{\pi} \lambda^2 , \beta_4 =0 \, .
\]

\subsubsection{Squaring Entries:} \label{sec:squaring} We now provide an example which does not fall into the class $\mathcal{F}_0$.  Suppose we choose $F(x)= x^2$, then explicitly one computes the log-likelihood to be given by: 
\[ g(y,w) =
-\frac{1}{2} \log(2\pi) - \log(2) - \frac{1}{2} \log(y) + \log [ e^{-(\sqrt{y}-\lambda w)^2/2} + e^{-(\sqrt{y}+\lambda w)^2/2} ] \, .
\]
In particular, the second derivative of $g(y,w)$ at $w=0$ is given by: 
\[
\partial_w^2 g(y,0) = \lambda^2(y-1) \, ,
\]
and consequently the bound $\norm{ \partial_w^2 g(\cdot,0) } <\infty$ fails. 

{
\subsection{Sparse PCA.} 

Let $x_i$ be i.i.d Bernoulli$(\rho)$ random variables and consider the matrix 
\[
Y= G + \frac{\lambda}{\sqrt{N}} \mathbf{x}\mathbf{x}^T \, ,
\]
this problem has been studied under the name of Sparse PCA or the principal submatrix recovery problem. Well-specified maximum likelihood estimation over the boolean hypercube $\{0,1 \}^N$ was exactly analyzed in \cite{gamarnik2021overlap}, where the authors show that if $\lambda > (2+\epsilon) \sqrt{\frac{1}{\rho} \log \rho }$, then the estimator satisfies $ \langle \mathbf{x}^{0,N}, \hat{\mathbf{x}}_{\MLE} \rangle > \rho (1-c)N $ for some constant $c$. 

 It is shown in the same work that the problem admits the Overlap Gap Property (OGP), giving evidence of algorithmic hardness for  computation of the MLE.
}

{
\subsection{Poisson-Bernoulli Matrices}

Consider matrices $Y_{ij}$ whose entries are generated conditionally on $w$ as symmetric poisson random variables of parameter $\lambda + w$. The log-likelihood in this case is given by: 
\[
g(y,w) = -\log(2) + \abs{y} \log(\lambda+w) - \log(\abs{y}!) - \lambda -w \, \quad  y \in \mathbb{Z} \, ,
\]
and the score parameters for the well specified case are $(1/\lambda,1/\lambda,1/\lambda,0)$. Note that the second derivative in $w$ of the likelihood is unbounded in $y$, but one can check that the results of Theorem~\ref{thm:main_regular} still holds. 
}
{
\section{Extensions and Directions for Further Research}
We end by discussing some future avenues of research. One important direction concerns the related algorithmic question of computing the pseudo MLE. This is the content of our forthcoming work \cite{upcoming}. In particular, we note here that our variational formula for the maximum constrained pseudo-likelihood can be used to probe the geometry of the pseudo-likelihood landscape and prove, e.g., OGP style results.

Another natural question surrounds extending our results to related rank 1 tensor recovery problems, such as multi-rank Tensor PCA. We expect our techniques to apply \emph{mutatis mutandis} using Derrida's $p$-spin model as the ''Gaussian equivalent'' in place of the Sherrington-Kirkpatrick model. 

Finally, one might be interested in extending our results to the rank $k>1$ setting. This is substantially more complex and new ideas are required. This is because it naturally involves the analysis of so-called "vector spin" spin glass models. The related ``method of annealing'' for such problems has yet to be developed due to the complexity of the the corresponding variational problem. In place of a convex variational problem over the space of probability measures, one considers a variational problem over the space of weighted paths of positive semi-definite matrices. Basic facts about this variational problem (e.g., convexity/nonconvexity, first variation formulas) are still open problems in the related spin glass literature. 
}\\

\noindent \textbf{Acknowledgements.} C.G.\ acknowledges the partial support of the Natural Sciences and Engineering
    Research Council of Canada Post-Graduate Scholarship Doctoral award. 
    A.J.\ and J.K.\ acknowledge the support of the Natural Sciences and Engineering Research Council of Canada (NSERC), the Canada Research Chairs program, the Canadian Foundation for Innovation- John Evan's Leaders fund, and the Ontario Research Fund. Cette recherche a \'et\'e enterprise gr\^ace, en partie, au 
soutien financier du Conseil de Recherches en Sciences Naturelles et en G\'enie du Canada (CRSNG), du Programme des chaires de recherche du Canada, du Foundation canadienne pour l'innovation FLJE, et les Fonds pour la recherce en Ontario. [RGPIN-2020-04597, DGECR-2020-00199,CRC-2022-00142,CFI-JELF Project 43994]

\begin{appendix}

\section{Outline of Proofs}
	
In this section, we will summarize the overall strategy to prove the main results. The  proofs will be deferred to the relevant sections of the Appendix ~below. To simplify the notation in this section, we only consider the ill-scored scenario. The case for well-scored problems are simpler and the proof is essentially the same. The only difference is the constraint on the mean $\bar{x}$, which is not needed when well-scored. If it is clear from context which scenario we are in, we will sometimes exclude the $-$ in the subscript.  

\subsection{Universality}
	
We begin by showing that the limit of the (normalized) maximum pseudo-likelihood is equivalent to that obtained by a maximization of a Gaussian model parameterized by the information parameters. The Gaussian model is given by
\begin{align*}
H_N^{\bar \beta}(x) 
=  \frac{\sqrt{\beta_1}}{\sqrt{N}} \sum_{i \leq j} g_{ij} x_i x_j + \frac{N \beta_2}{2} M_N(x)^2 - \frac{N \beta_3}{4}S_N(x)^2 + \beta_4 N^{\frac{3}{2}}( \bar{x} )^2+ O(1) \,  \numberthis \label{eq:modelequivmain}
\end{align*}
where we recall that $M_N$ and $S_N$ denote the normalized inner product and norm defined in \eqref{eq:SMnotation} and $\bar{x} = \frac{1}{N} \sum_{i = 1}^N x_i$ is the sample mean.

We prove in Appendix~\ref{sec:univ} that the asymptotic maxmium pseudolikelihood is equal to the one given by the maximum of the gaussian equivalent on average. Given $S, M , v \subseteq \R$, define $\Omega_{\epsilon}(S,M,v)$ as the set of points in $\Omega^N$ within $\epsilon$ of $(S,M,v)$, that is:
\[
\Omega_{\epsilon}(S,M,v) : = \{ x \in \Omega^N : \abs{ M_N(x)-M} \leq \epsilon, \abs{S_N(x)-S}\leq \epsilon, \abs{ \overline{x} -v} \leq \epsilon \} \, ,\]
and let us define
 \begin{align}\label{eq:constrainedPML}
    \cL_N^{g,\epsilon}(S,M,v) &=  \E  \max_{\Omega_{\epsilon}(S,M,v)} \bigg( \sum_{i \leq j} g(Y_{ij}, \frac{x_i x_j}{\sqrt{N}}\Big) - \sum_{i \leq j} g(Y_{ij},0) \bigg) \, ,\\ 
\label{eq:constrained}
       \cL_N^{\bar \beta,\epsilon}(S,M,v) &=   \E  \max_{\Omega_{\epsilon}(S,M,v)} H_N^{\bar \beta}(x) \, , 
    \end{align}
	to denote the restricted pseudo likelihood and the gaussian likelihood respectively. We prove in Appendix~\ref{sec:univ} that the pseudo maximum likelihood and the maximum of the Gaussian equivalent are asymptotically equal in the following sense. 
	\begin{lem}\label{lem:lemUniv}
		If $g,g_0 \in \cF_0$, then for any $(S,M,v)\in \cC_c$
		\[
		\lim_{N \to \infty} \frac{1}{N} | \cL_N^{g,\epsilon}(S,M,v)  - \cL_N^{\bar \beta,\epsilon}(S,M,v)   | = 0 \, .
		\]
	\end{lem}
	
	This is proved by showing equivalence for a well-chosen log-likelihood ratio with respect to a uniform prior which can be viewed as a ``smooth approximation'' of the maximum likelihood. Let $\pP_X(x)$ be the uniform measure on $\Omega$. 

    Define the log-likelihood ratio associated with the pseudo likelihood
	\begin{equation} 
		F_N(g,\epsilon;S,M,v)
            := \frac{1}{N}\E_Y  \log \int_{\Omega_{\epsilon}(S,M,v)} e^{ \sum_{i \leq j} g(Y_{ij}, \frac{x_i x_j}{\sqrt{N}}) } \, d\pP_X^{\otimes N}(\vect{x}) - \sum_{i \leq j} g(Y_{ij},0)    \label{eq:FEg}
	\end{equation}
	where $\E_Y$ is with the expectation with respect to the conditional data distribution \eqref{eq:conditional_data_dist}.    On the other hand, we define the log-likelihood ratio of the Gaussian equivalent for $\bar \beta = (\beta_1, \beta_2, \beta_3, \beta_4) \in \R^4$ by
    \begin{equation}\label{eq:FEgrowingrank_beta}
        F_N(\bar\beta,\epsilon;S,M,v) := \frac{1}{N} \E_Y  \log \int_{\Omega_\epsilon(S,M,v)}  e^{ H_N^{\bar \beta}( \vect{x}) } \, d\pP_X^{\otimes N}(\vect{x})  \, ,
    \end{equation}
	where $	H_N^{\bar \beta}( \vect{x})$ is as in \eqref{eq:modelequivmain}. We define $F_N^{g,\epsilon}$ and $F_N^{\bar \beta, \epsilon}$ to be the equal to \eqref{eq:FEg} and \eqref{eq:FEgrowingrank_beta} without the constraints in the integral. These quantities approximate the pseudo maximum likelihood in the sense that for any $(S,M,v) \in \cC_c$
    \begin{align}\label{eq:goundstateboundgeneral}
    \lim_{L \to \infty}  \lim_{N \to \infty}| \frac{1}{L}F_N(Lg,\epsilon;S,M,v) - \frac{1}{N} \cL_{N}^{g,\epsilon}(S,M,v) | &= 0 \\ 
    \lim_{L \to \infty}  \lim_{N \to \infty}| \frac{1}{L}F_N(L \bar \beta,\epsilon;S,M,v) - \frac{1}{N} \cL_{N}^{\bar \beta,\epsilon}(S,M,v) | &= 0 \, .
    \end{align}
    An analogous statement holds for the unconstrained versions.   Universality for the pseudo maximum likelihood in Lemma~\ref{lem:lemUniv} follows from the following universality for the log-likelihood functions:  
	\begin{lem}\label{lem:univfinitetemp}
		If $g,g_0 \in \cF_0$, then for any $S,M,v \in \cC_c$
		\begin{align*}
		\lim_{L \to \infty} \lim_{N \to \infty} | \frac{1}{L} 	F_N(Lg,\epsilon;S,M,v) - \frac{1}{L} F_N(L\bar\beta,\epsilon;S,M,v)   | &= 0\\
        \lim_{L \to \infty} \lim_{N \to \infty} | \frac{1}{L} 	F_N^{Lg, \epsilon} - \frac{1}{L} F_N^{L \bar \beta, \epsilon}   | &= 0 \, .
        \end{align*}
	\end{lem} 
    \begin{proof}[Proof sketch]
        The proof follows the arguments in \cite[Section~3]{nonbayes}. The key difference is that the universality result is extended to ill-scored models. The technical details of this argument are provided in Appendix~\ref{sec:univ} for generic constraints, but we identify the key steps below.

        The key idea in this proof is that at the level of the log likelihood, we are able to use Taylor's theorem to expand around the likelihood in the exponent with respect to  $w_{ij}=\frac{x_i x_j}{\sqrt{N}}$,
        \[
        \frac{1}{L}F_N(Lg,\epsilon;S,M,v) = \frac{1}{NL}  \E_Y  \log \int_{\Omega_{\epsilon}}  e^{ \sum \partial_w L g( Y_{ij},0) w_{ij} + \frac{1}{2} \partial_w^{2} L g(  Y_{ij},0)  w_{ij}^2 }  d\pP_X^{\otimes N}  + o_N(1) \, ,
        \]
        where $\Omega_{\epsilon}=\Omega_{\epsilon}(S,M,v)$
        since the third derivative of $g$ is uniformly bounded for $g \in \cF_0$. 
        The second order coefficients of the Taylor series will concentrate in the high-dimensional limit while the first order term will be approximately Gaussian with a specific mean and variance given by,
        \[
        \E[ \partial_w g( Y_{ij},0) ] =  \beta_4 + \frac{x_i^0 x_j^0}{\sqrt{N}} \beta_2 + O(N^{-1})\quad \text{ and } \quad \Var[  \partial_w g( Y_{ij},0) ] = \beta_1 + O(N^{-1/2}).
        \]
        The decomposition then follows immediately from standard universality in disorder arguments for spin glasses, see, e.g., \cite{CarmonaHu}. We conclude that
        \[
        \lim_{L \to \infty} \lim_{N \to \infty} | \frac{1}{L} 	F_N(Lg,\epsilon;S,M,v) - \frac{1}{L} F_N(L\bar\beta,\epsilon;S,M,v)   | = 0.
        \]
        The proof for the unconstrained problem is identical.
    \end{proof}

    The proof of Lemma~\ref{lem:lemUniv} now follows from the triangle inequality.

    \begin{proof}[Proof of Lemma~\ref{lem:lemUniv}] By the triangle inequality, we have
    \begin{align*}
        \frac{1}{N} | \cL_N^{g,\epsilon}(S,M,v)  - \cL_N^{\bar \beta,\epsilon}(S,M,v)   | &\leq  |\frac{1}{N} \cL_{N}^{g,\epsilon}(S,M,v) - \frac{1}{L}F_N(Lg,\epsilon;S,M,v) |  
        \\&\quad + | \frac{1}{L} 	F_N(Lg,\epsilon;S,M,v) - \frac{1}{L} F_N(L\bar\beta,\epsilon;S,M,v)   |
        \\&\quad + | \frac{1}{L}F_N(L \bar \beta,\epsilon;S,M,v) - \frac{1}{N} \cL_{N}^{\bar \beta,\epsilon}(S,M,v) | \, .
    \end{align*}
    The first and third terms vanish in the limit by \eqref{eq:goundstateboundgeneral} and the second term vanishes by Lemma~\ref{lem:univfinitetemp}. 
    \end{proof}
	
	The main consequence of Lemma~\ref{lem:univfinitetemp} is that it suffices to compute the limit in the case of the Gaussian equivalent of $F_N(\bar\beta;S,M,v)$ instead of the pseudo maximum likelihood. The computation of this limit is the focus of the following two sections. 
	
	\subsection{Derivation of the Variational Formula I}
    In this section, we once again use the approximation of the likelihoods with the loglikelihood ratios and first compute the limit of the loglikelihood ratio. Our goal is to first define the variational formula for the loglikelihood ratios. 
    
    Notice that the term corresponding to $H_{N}^{\bar \beta}$ is of higher order, so this term must be corrected  in order to have a well-defined limit. To this end, we define
    \begin{equation}\label{eq:modelequiv2}
    H_{N}^{\bar \beta,\alpha}(\vect{x}) = H_{N}^{\bar \beta}(\vect{x}) - \beta_4 N^{\frac{3}{2}}( \bar{x} )^2 + \alpha N ( \bar{x} )^2 \, ,
    \end{equation}
    which is the gaussian equivalent for \eqref{eq:correctedPMLE}. We define
    \begin{equation}\label{eq:constrainedFE}
    F_{N,\alpha}(\bar\beta,\alpha,\epsilon;S,M,v) := \frac{1}{N} \E_Y  \log \int_{\Omega_\epsilon(S,M,v)}  e^{ H_{N}^{\bar \beta,\alpha}( \vect{x}) } \, d\pP_X^{\otimes N}(\vect{x})  \, ,  
    \end{equation}
    and let $ F_{N,\alpha}(\bar\beta,\alpha,\epsilon)$ denotes its unconstrained version.

    We now define the variational formula which will compute the limit. Let $\zeta$ be a probability measure, and let $\Phi_{\zeta,\lambda,\mu,\rho}(t,y)$ be the unique weak solution to the Parisi PDE
	\begin{equation}\label{eq:parisipdefinite}
		\begin{cases}
			\partial_t \Phi_{\zeta,\lambda,\mu,\rho} = - \frac{\beta_1}{4} ( \partial_{y}^2 \Phi_\zeta + \zeta( [0,t] ) ( \partial_y \Phi_\zeta )^2) & (t,y) \in (0,S) \times \R \\
			\Phi_{\zeta,\lambda,\mu,\rho}(S,y) = \log \int e^{yx + \lambda x x^0 + \mu x^2 + \rho x} \, d \pP_X (x)
		\end{cases}
        \, .
	\end{equation}
See \cite{JagTob16} for the notion of weak solutions for this PDE and the corresponding well-posedness. 
Define the corresponding Parisi functional by 
	\begin{equation}
	    \begin{aligned} \label{eq:variational-ridge} 
		\varphi_{\bar\beta,\alpha}(S,M,v)  =\inf_{\mu,\lambda,\rho, \zeta}  \E_{\pQ}[&\Phi_{\zeta,\lambda,\mu,\rho}(0,0)] - \frac{\beta_1}{2} \int_0^S t\zeta([0,t])  dt \\ &- \mu S - \lambda M - \rho v 
		+ \frac{\beta_2 M^2}{2}  - \frac{\beta_3 S^2}{4} + \frac{ \alpha v^2}{2}   .
	    \end{aligned}
	\end{equation}
 
	Furthermore, we see that $(S_N(\vx_{\PMLE}),M_N(\vx_{\PMLE}), \bar \vx_{\PMLE} )$ asymptotically live 
	in the closed subset $\mathcal C_c$  of $[0,C^{2}]\times [-C^{2},C^{2}] \times [-C,C]$  defined in \eqref{defC}. This is the domain of our functional. We will show in Appendix~\ref{sec:devvarI} that the limit of the loglikelihood is given by the Parisi functional.
	
	\begin{theo}\label{th:finitetempcorrected}
		For any $\beta_1, \beta_2, \beta_3$ and $\alpha$ and constraints $(S,M,v) \in \mathcal C$, we have 
        \begin{align*}
		\lim_{\epsilon \to 0} \lim_{N \to \infty} F_{N,\alpha}(\bar\beta,\alpha,\epsilon;S,M,v) &= \varphi_{\mathbf{\beta},\alpha}(S,M,v)\\
		\lim_{\epsilon \to 0} \lim_{N \to \infty} F_{N,\alpha}(\bar\beta,\alpha,\epsilon)&=\sup_{(S,M,v) \in \cC }\varphi_{\mathbf{\beta},\alpha}(S,M,v) .
		\end{align*}
	\end{theo}
	\begin{rem}
	For regular models, the constraint on $v$ can be completely removed and was proven in \cite{nonbayes}. In such cases, the optimization is over the functional $\varphi_{\bar \beta}$ which is defined on only two parameters $S$ and $M$. 
	\end{rem}
    \begin{proof}[Proof sketch of Theorem~\ref{th:finitetempcorrected}]
    The proof of this results use techniques first developed to study mean-field models of spin glasses. By introducing a small perturbation to the log-likelihood, we are able to characterize the limiting behaviour of independent samples from the perturbed posterior measure and explicitly compute the limit. The proof also borrows techniques from large deviations to remove the constraint on the overlaps. We sketch the key steps here. The full proofs are provided in the Supplement as indicated below.
    \hfill
    
    \textit{Upper Bound:} We first prove that
\begin{equation}\label{eq:finitetempupbound}
     F_{N,\alpha}(\bar\beta,\alpha,\epsilon;S,M,v) \leq \varphi_{\mathbf{\beta},\alpha}(S,M,v). 
    \end{equation}
    This argument follows from a Guerra-type replica symmetry breaking interpolation \cite{guerra2003broken}. In particular, it follows from the computations in Proposition~\ref{prop:upbd} that for any $\lambda,\mu,\rho$ and $\zeta$
    \begin{align*}
			F_{N,\alpha}(\bar\beta,\alpha,\epsilon;S,M,v)  &\leq 
			- \lambda S - \mu M - \rho v + \frac{1}{N} \sum_{i = 1}^N \Phi_{\lambda,\mu,\zeta}(0,0) - \frac{\beta_1^2}{2} \int_0^S t\zeta(t) \, dt 
			\\&+ \frac{\beta_2}{2} M^{2} - \frac{\beta_3}{4} S^{2} + \frac{\alpha}{2} v^2  +   L\varepsilon(|\mu|+|\lambda|+|\rho|)+o_{N,\epsilon}(1) \, .
		\end{align*}
    This upper bound holds for all $\lambda,\mu,\rho$ and $\zeta$, so we can take the infumum to arrive at \eqref{eq:finitetempupbound}. The sharpness of the upper bound after minimizing over $\lambda,\mu,\rho$ is a consequence of a modification of the Gartner--Ellis Theorem \cite[Theorem~2.3.6]{DZ} and is given in detail in Lemma~\ref{lem:sharpupbd}.
   \hfill
   
   \textit{Lower Bound:} We then prove the matching lower bound
    \[
    F_{N,\alpha}(\bar\beta,\alpha,\epsilon;S,M,v) \geq \varphi_{\mathbf{\beta},\alpha}(S,M,v).
    \]
    This proof uses the cavity method, in this case called the Aizenman-Sims-Star scheme \cite{AS2}, and a perturbation of the posterior that forces the limiting overlap to satisfy the Ghirlanda--Guerra identities \cite{GG} and ultrametricity \cite{PUltra}. The proof of the lower bound is included in Proposition~\ref{prop:lowerboundFE} for completeness. It is worth pointing out that on the set $\Omega_{\epsilon}(S,M,v)$, the $S_N(\vect{x})$ is approximately constant, so we may do a change of variables, and restrict ourselves onto the ball of radius $S$ at the cost of a small error term, which we can control. This implies that the usual proof of the Ghirlanda--Guerra identities holds directly in our setting. 
    \end{proof}

	Having computed the appropriate limit of $F_N(\bar\beta;S,M,v)$, we can apply Lemma~\ref{lem:univfinitetemp} to recover the limit of the loglikihood ratio. By \eqref{eq:goundstateboundgeneral} if one can compute the limit of this quantity as $L \to \infty$, then we can recover the limiting formula for the pseudo maximum likelihood. 

	\subsection{Derivation of the Variational Formula II}
	
	This variational formula holds for all $\beta_1, \beta_2, \beta_3$ and $\alpha$, so it also holds when these parameters are scaled by $L$ as in the smooth approximation. We will show in Appendix~\ref{sec:devvarII} that taking the limit as $L \to \infty$ of this variational formula will give the formula for the pseudo maximum likelihood, after an application of \eqref{eq:goundstateboundgeneral}, which will give us a variational formula for the limit of pseudo maximum likelihood. 
	
	\begin{lem}\label{lem:groundstate}
		For any $\bar \beta$, in the constrained pseudo maximum likelihood we have
        \begin{align*}
		\lim_{\epsilon \to 0} \lim_{N \to \infty} \frac{1}{N} \cL_{N,\alpha}^{g,\epsilon}(S,M,v) &= \lim_{\epsilon \to 0} \lim_{N \to \infty} \frac{1}{N} \cL_{N}^{\bar \beta,\epsilon}(S,M,v)  =   \psi_{\bar \beta,\alpha}(S,M,v) \, \\
		\lim_{\epsilon \to 0}  \lim_{N \to \infty} \frac{1}{N} \cL_{N}^{g,\epsilon} &= \lim_{\epsilon \to 0} \lim_{N \to \infty} \frac{1}{N} \cL_{N,\alpha}^{\bar \beta,\epsilon}  = \sup_{(S,M,v) \in \cC_c }  \psi_{\bar \beta,\alpha}(S,M,v) \, ,
		\end{align*}
		where $\psi_{\bar\beta,\alpha}$ is as in \eqref{eq:psi-beta-alpha-def}
	\end{lem}
This proof uses the $\Gamma$ limit of the solutions in \cite{JagSen} to the Parisi PDE \eqref{eq:parisipdefinite} to identify $\psi_{ \bar \beta}$ as the limit  of $ \frac{1}{L} \varphi_{\bar \beta}$. 
In the lemma above the constrained case is established in Appendix ~\ref{sec:devvarII}.
The limit formula in the unconstrained case is established in Appendix ~\ref{AP:discrete-convergent}.

Having understood the limiting variational formula, one can also show using \eqref{eq:goundstateboundgeneral}, that this limiting variational formula characterizes the constrained pseudo maximum likelihood.  
	
	We can conclude that the limit of the pseudo maximum likelihood is a variational optimization over the parameters $S,M,v$. We will show in the next section that the maximizers of the variational problem encode the limiting performance of the maximum likelihood estimators. 
	
	\subsection{Characterization of the Maximizers}
	In Appendix~\ref{sec:charac} we prove tightness of the overlaps as stated in Theorems ~\ref{thm:main_regular} and ~\ref{th:corrected}.
    The tightness will follow from concentration properties satisfied by the gaussian equivalent ~\eqref{eq:modelequivmain}, and the results proved in Appendix ~\ref{sec:devvarII}. 
    
    Next, under the further assumption that $\psi_{\bar \beta,\alpha}$ has a unique maximizer, we are able to prove the following characterization of performance: 
	\begin{lem}\label{lem:charmaxintro}
	For $g, g^0 \in \cF_0$ let $\bar \beta$ denote the corresponding information and score parameters and suppose that $\psi_{\bar \beta}$ has a unique (up to the sign of $m$) maximizer $(s_{\bar \beta}, \pm m_{\bar \beta},v_{\bar \beta} )$. Then
		\begin{align}
		|\cs ( \hat{\mathbf{x}}_{\PMLE}, \mathbf{x}_0 )|  &\to \frac{|m_{\bar{\beta}}|}{ (s_{\bar{\beta}} \E_{\mathbb{Q}} (x^0)^2)^{1/2} } \qquad \text{and}\qquad 
		\frac{1}{N}\norm{\hat{x}_{PMLE}}^2\to s_{\bb} \qquad  \, \text{a.s.} 
	\end{align}
	\end{lem}
    \begin{proof}[Proof sketch]
        The proof follows from the fact that the limit of the constrained pseudo maximum likelihood over the set $\Omega_\epsilon(s,m,v)$ is given by $\psi_{\bar \beta}(s,m,v)$,
        \begin{equation}\label{eq:chargroundstateconstrained}
     \lim_{\epsilon \to 0} \lim_{N \to \infty} \frac{1}{N} \E \max_{\Omega_\epsilon(s,m)} H_N^{\bar \beta}(\vect{x})  = \psi_{\bar \beta}(s,m) \, ,
        \end{equation}
        which follows from Lemma~\ref{lem:groundstate}. In the notation above, we have defined 
        \[
	\psi_{\bar \beta}(S,M) = \begin{cases}
		\sup_{v} \psi_{\bar \beta, \alpha}(S,M,v) & \text{ if } \alpha \neq 0\\
		\psi_{\bar \beta,0}(S,M) & \text{ if } \alpha = 0.
	\end{cases} 
	\]
    to handle the cases for well-scored and corrected models simultaneously. Next, by concentration \cite[Section~2.1]{adler_taylor_book} for every $\epsilon > 0$, 
		\[
		\lim_{N \to \infty} \frac{1}{N} \E \max_{x \in \Omega^N} H_N^{\bar \beta}(\vect{x}) = \lim_{N \to \infty} \frac{1}{N} \E \max_{s,m} \max_{\Omega_\epsilon(s,m)} H_N^{\bar \beta}(\vect{x}) = \lim_{N \to \infty} \frac{1}{N}\max_{s,m}  \E  \max_{\Omega_\epsilon(s,m)} H_N^{\bar \beta}(\vect{x}) \, .
		\]
    Since the maximizer $(s_{\bar \beta}, \pm m_{\bar \beta},v_{\bar \beta} )$ is unique up to a sign and $H_N^{\bar \beta}$ only depends on $S_N$ and $M_N$ through its squared value we have
    \[
		\lim_{N \to \infty} \frac{1}{N} \E \max_{x \in \Omega^N} H_N^{\bar \beta}(\vect{x}) = \lim_{N \to \infty} \frac{1}{N} \E \max_{s,m} \max_{\Omega_\epsilon(s,m)} H_N^{\bar \beta}(\vect{x}) = \lim_{N \to \infty} \frac{1}{N}\E  \max_{\Omega_\epsilon(s_{\bar \beta},m_{\bar \beta})} H_N^{\bar \beta}(\vect{x}) \, .
	\]
    Finally, we can apply universality in Lemma~\ref{lem:lemUniv} to conclude that the pseudo maximum likelihood is maximized on the set $\Omega_{\epsilon}(s_{\bar \beta},m_{\bar \beta})$, which implies that the PMLE satisfies 
    \[
		|M_N(\vect{\hat x}^g_{\PMLE})| = |M_N(\vect{\hat x}^{\bar \beta}_{\PMLE})| = |m_{\bar \beta}| \quad \text{ and } \quad  S_N(\vect{\hat x}^g_{\PMLE}) = S_N(\vect{\hat x}^{\bar \beta}_{\PMLE}) = s_{\bar \beta} \, ,
		\]
    leading to the characterization of the cosine similarity and norm. 
    \end{proof}
    If the maximizers of $\psi_{\bar \beta}$ are not unique, then the limit points of all near maximizers are attained on the set $\cC_{\bar \beta}$ or $\cC_{\bar \beta,\alpha}$. 
    \begin{lem} \label{lem:char_limitpoints}
    Let $\overline{\beta}$ be fixed, and suppose that $(S,M)$ are such that $-\infty < \psi(S,M) < \sup \psi$. Let $\text{GS}_N$ denote the (random) collection of maximizers of $H_N^{\overline{\beta} }$ in $\Omega^N$, then for $\epsilon>0$ sufficiently small, one has:
    \[
 \limsup_{N \to \infty} \frac{1}{N} \log \pP ( \text{GS}_N \cap \{ \vx \mmm |S_N(\vx)-S|\leq \epsilon, |M_N(\vx)- M|\leq\epsilon \} \neq \emptyset ) < 0  \, .
    \]
    Furthermore, one has that the collection of all limit points, taken over all sequences of near maximizers $\vx_N$, for the sequence $(S_N( \vx), M_N( \vx) )$ is equal to $\cC_{\bar \beta}$.  
\end{lem}
\begin{proof}[Proof sketch]
    The detailed proof of this argument is deferred to Appendix ~\ref{sec:charac}. It essentially follows a similar line of reasoning as Lemma~\ref{lem:charmaxintro} and relies on the characterization of pseudo maximum likelihood in \eqref{eq:chargroundstateconstrained} and the exponential concentration of $H_N^{\bar \beta}$ to achieve an exponential rate of concentration of the event. 
\end{proof}
The two preceding Lemmas provide the characterization of the cosine simlarity in Theorems ~\ref{thm:main_regular} and ~\ref{th:corrected}.

	\subsection{Coarse Equivalence of Estimators}
	
	In Appendix~\ref{sec:modelequiv}, we  prove a sufficient condition for when two likelihoods are coarsely equivalent. Coarse equivalence will follow as a consequence of the  universality result in Lemma~\ref{lem:lemUniv}. 

    \begin{proof}[Proof of Theorem~\ref{th:robust} and Theorem~\ref{cor:model-equiv-corrected }]
    We first provide a proof of Theorem~\ref{cor:model-equiv-corrected }.
        We start with the first condition in Theorem~\ref{cor:model-equiv-corrected }. Given likelihoods $g^1,g^2$ which satisfy :
	\begin{equation}
		\frac{\sqrt{\beta_1(g^1)}}{\sqrt{\beta_1(g^2)}} = \frac{\beta_2(g^1)}{\beta_2(g^2)} =  \frac{\beta_3(g^1)}{\beta_3(g^2)} = \frac{\beta_4(g^1)}{\beta_4(g^2)} = C,
	\end{equation}
    the corresponding gaussian equivalents will be a scalar multiple of each-other. Consequently the collection of near-maximizers will be the same for both problems, and the result will then follow from Theorem~\ref{thm:main_regular}. 
    
    To see the second statement, notice that if our parameter space satisfies $|x| = C$, then the maximizer of the gaussian equivalent $H_N^{\bar \beta}$ given in \eqref{eq:modelequivmain} is independent of $\beta_3$ since it is constant, which proves the second statement in Theorem~\ref{cor:model-equiv-corrected }. 

    The proof of Theorem~\ref{th:robust} will follow from a a similar argument, where we appeal to Theorem~\ref{thm:main_regular}. 
    \end{proof}

    We further prove Theorem ~\ref{thm:master-problem} in this section, it will be an immediate consequence of the universality established in Theorems ~\ref{thm:main_regular} and ~\ref{th:corrected}. 

    \begin{proof}[Proof of Theorem~\ref{thm:master-problem}]
    It suffices to show the information parameters of $(g_{U,0}^{\bar \beta},g_{U,1}^{\bar \beta} )$ are equal to $\bar \beta$.  We have
    \[
    \partial_w g_{U,0}^{\bar \beta}(y,0) = \frac{\beta_2}{\beta_1} (y - \beta_4), \quad  \partial_w g_{U,1}^{\bar \beta}(y,0) = y, \quad \partial_w^2 g_{U,1}^{\bar \beta}(y,0) = \beta_3.
    \]
    Furthermore, under the null-model we have that $Y$ is Gaussian with mean $\beta_4$ and variance $\beta_1$. A direct computation implies that the information parameters of $(g_{U,0}^{\bar \beta},g_{U,1}^{\bar \beta} )$ are $\beta_1, \beta_2, \beta_3, \beta_4$. 
    \end{proof}

{
\subsection{Generalizations to Full Space}
	
In Appendix~\ref{sec:fullspace}, we prove results related to the weakening of the compactness assumptions on $\Omega$ and $\Omega_0$. These extensions to unbounded signals can be reduced to the bounded signal case under some assumptions on the tails of the signal using a truncation argument. This is done in Appendix~\ref{sec:subgauss_signal}. For the least squares problem on the entire space, the models are in fact simpler to analyze because the performance can be fully understood by the spectral properties of spiked matrices. In fact, the limiting performance in Theorem~\ref{thm:least-squares-all-space} is given explicitly in terms of the eigenvalues of a spiked Gaussian matrix. This connection with random matrix theory is explained and made precise in Appendix~\ref{sec:LSFullSpaceProof}.
}
\section{Universality with Non-Zero Score}\label{sec:univ}

	In this section, we prove universality of pseudo maximum likelihood estimation with possibly non-zero score parameters. This extends the universality result in \cite{nonbayes} to the case of the pseudo maximum likelihood and removes the zero score assumption in \cite[Hypothesis~2.3]{nonbayes}.
    
    Given the information parameters, $\barbeta$, we recall the gaussian equivalent $H^{\barbeta}_N$ from \eqref{eq:modelequivmain} 
    Likewise, given $S,M,v$ recall that $ \cL_N^{g,\epsilon}(S,M,v) $ and  $\cL_N^{\bar \beta,\epsilon}(S,M,v) $ 
	denotes the normalized pseudo maximum likelihood and the maximum of the gaussian equivalent respectively. 
	The goal of this section is to show that these quantities converge to the same value as stated in Lemma~\ref{lem:lemUniv},
    \[
		\lim_{N \to \infty} \frac{1}{N} | \cL_N^{g,\epsilon}(S,M,v)  - \cL_N^{\bar \beta,\epsilon}(S,M,v)   | = 0 \, .
		\]
	
	The proof of the maximum likelihood formulas will follow from an extension of the universality for Bayesian models proven in \cite{nonbayes}. In contrast to the Bayesian inference setting, we fix $\vect{x}^0$ and define the log-likelihood ratios
	\begin{align}
		F_N(g;S,M,v) \notag
        = \frac{1}{N} \bigg( \E_Y  \Big(\log \int_{\Omega_{\epsilon}(S,M,v)} e^{ \sum_{i \leq j} g\Big(Y_{ij}, \frac{x_i x_j}{\sqrt{N}}\Big) } \, d\pP_X^{\otimes N}(\vect{x}) - \sum_{i \leq j} g(Y_{ij},0) \Big) \bigg) \label{eq:FEgrowingrank}
	\end{align}
	where $\E_Y$ is with the average with respect to the conditional data distribution \eqref{eq:conditional_data_dist}.

	We have standard bounds relating $L_N$ and $F_N$ given by:
    \begin{equation}\label{eq:boundsgroundstate}
		\cL_N^g(S,M,v) \leq   \frac{ F_N(L g;S,M,v) }{L} \leq \cL_N^g(S,M,v) + o_L(1) \, .
	\end{equation}
	We also define the Gaussian log-likelihood ratio for $\bar \beta = (\beta_1, \beta_2, \beta_3, \beta_4) \in \R_+^3 \times \R$ by
	\[
	F_N(\bar\beta;S,M,v) = \frac{1}{N} \E  \log \int_{\Omega_{\epsilon}(S,M,v)} e^{ H_N^{\bar \beta}( \vect{x}) } \, d\pP_X^{\otimes N}(\vect{x})  \, , 
	\]
	where $	H_N^{\bar \beta}( \vect{x})$ was defined in \eqref{eq:modelequivmain}. In the case that $\Omega$ is discrete we let $\pP_X$ denote counting measure, and in the case that $\Omega$ is an interval, we let $\pP_X$ denote normalized Lebesgue measure. We start by proving universality for log-likelihood.
	\begin{prop}[Universality of Bayesian Models] \label{prop:universality1}
		Let $g, g^0 \in \cF_0$ and let $\barbeta$ be their corresponding information parameters with respect to $\pP_0$. For any $(S,M,v)$ there exists a constant $K>0$  depending only on $g,g^0$ such that
		\[
		\big| F_N(g;S,M,v) - F_N(\bar\beta;S,M,v) \big| \leq \frac{K}{\sqrt{N}}\,.
		\]
	\end{prop}

	\begin{proof}
		The proof is in Section~3 from \cite{nonbayes}. We highlight the key steps. Throughout we let $K$ denote a universal constant that only depends on the supports $\Omega$ and $\Omega_0$, but not on the dimension $N$. 
		\hfill
        
        \textit{Step 1 - Approximation by Third Order Terms:} We first show that to leading order in $N$, it suffices to consider only a third order expansion of the loglikelihood around $w=0$, define a proxy $F^{(1)}$ by:
		\[
		F^{(1)}_N(g;S,M,v) =  \frac{1}{N} \bigg( \E_Y  \Big(\log \int_{\Omega_{\epsilon}(S,M,v)} e^{ \sum_{i \leq j} \partial_w g( Y_{ij},0) w_{ij} + \frac{1}{2} \partial_w^{2}g(  Y_{ij},0)  w_{ij}^2 } \, d\pP_X^{\otimes N}(\vect{x}) \Big) \, .
		\]

		By our regularity assumptions on $g$ we may Taylor expand the log-likelihood. In particular, Taylor's theorem implies there is $\theta_{ij} \in [0,1]$ such that
		\[
		(g(  Y_{ij},w_{ij})-g( Y_{ij},0))=  \partial_w g( Y_{ij},0) w_{ij} + \frac{1}{2} \partial_w^{2}g(  Y_{ij},0)  w_{ij}^2  +\frac{w_{ij}^3}{3!} \partial_w^{3}g(Y_{ij},\theta_{ij}w_{ij}) \, ,
		\]
	and	since $\partial_w^{3}g(Y_{ij},\theta_{ij}w_{ij})$ is uniformly bounded and $|w_{ij} | \leq \frac{C^2 \lambda}{\sqrt{N}}$, we have
		\[
		\bigg| F_N(g;S,M,v) - F^{(1)}_N(g;S,M,v) \bigg| \leq \frac{\| \partial_w^{3}g(Y_{ij},\theta_{ij}w_{ij})\|_\infty K}{\sqrt{N}} \, ,
		\]
  and thus, it suffices to compute the limit for $F^{(1)}_N$.
\hfill

\textit{Step 2 - Control of the Second Order Terms:} 
	We now show that we can replace $\partial_w^{2}g(  Y_{ij},0)  w_{ij}^2 $ in the exponent with its average. Define
	\begin{equation}\label{eq:FNtwo}
	F^{(2)}_N(g;S,M,v) =  \frac{1}{N} \bigg( \E_Y  \Big(\log \int_{\Omega_{\epsilon}(S,M,v)} e^{ \sum_{i \leq j} \partial_w g( Y_{ij},0) w_{ij} + \frac{1}{2} \E_{Y}[ \partial_w^{2} g(  Y,0) ]  w_{ij}^2 } \, d\pP_X^{\otimes N}(\vect{x}) \Big) \, , 
	\end{equation}
	then we may express the difference of $F^{(1)}_N$ and $F^{(2)}_N$ as follows: 
	$$
	F^{(1)}_N(g;S,M,v)-F^{(2)}_N(g;S,M,v)=\E_Y\frac{1}{N}\log \Big\langle \exp \Big( \frac{1}{\sqrt{N}} \text{Tr}( Z(\bx^T\bx)^2)\Big)  \Big\rangle \, ,
	$$
	where for a function $f: \Omega^N \to \R$, the average $\langle f \rangle$ is defined as:
	\[
	\langle f \rangle := \frac{\int_{\Omega_{\epsilon}(S,M,v)}  f (\bx) e^{\sum_{i \leq j}  \partial_w g( Y_{ij},0) w_{ij} + \frac{1}{2} \E_{0}[ \partial_w^{2} g(  Y,0) ]  w_{ij}^2 } d\pP_X^{\otimes N}(\bx)}{ \int_{\Omega_{\epsilon}(S,M,v)}   e^{ \sum_{i \leq j} \partial_w g( Y_{ij},0) w_{ij} + \frac{1}{2} \E_{Y}[ \partial_w^{2} g(  Y,0) ]  w_{ij}^2 } d\pP_X^{\otimes N}(\bx)}\, ,
	\]
	and $Z$ is the matrix with entries:
	\[
	Z_{ij} := \frac{1}{2\sqrt{N}}(\partial_w^{2}g(Y_{ij},0)-\E_0[ \partial_w^{2}g(Y_{ij},0)]) \, .
	\]
	Now $Z$ is a centered random matrix whose entries have covariance bounded by $C/N$, and as $\partial_w^2 g(Y_{ij},0)$ is uniformly bounded by \eqref{eq:regularityg0}, the entries of $Z$ are bounded. Standard concentration inequalities for random matrices (see \cite[Corollary~2.3.5]{TaoRMTBook}) imply that the operator norm of $Z$ is order $1$ with exponentially decaying tails, and so by bounding the trace on this event we deduce that:
	
	$$
	\abs{ F^{(1)}_N(g:S,M,v)-F^{(2)}_N(g:S,M,v) } \leq  \frac{K}{\sqrt{N}} \, ,
	$$
	for some $K>0$. 
	
	\textit{Step 3 - Expansion of The First Order Term:} 
	We now show that the first order term can be approximated by a Gaussian random variable with non-zero mean. Define
	\begin{equation}\label{eq:FNthree}
		F^{(3)}_N(g;S,M,v) :=  \frac{1}{N} \E_Y  \Big(\log \int_{\Omega_{\epsilon}(S,M,v)}  e^{ \sum_{i \leq j} \beta_1 g_{ij} w_{ij} + \beta_2 w^0_{ij} w_{ij} - \frac{1}{2} \beta_3  w_{ij}^2 + \beta_4 w_{ij} } \, d\pP_X^{\otimes N}(\vect{x}) \Big) \, ,
	\end{equation}
	and consider the following moments of the information parameter under the data distribution,
	\begin{enumerate}
		\item $\mu_{ij} =\E_Y [\partial_w g(Y_{ij}, 0) \given \vect{x}^0]$
		\item $\sigma^2_{ij} = \E_Y[ (\partial_w g(Y_{ij}, 0)  -\mu_{ij})^2 \given \vect{x}^0]$
		\item $\gamma_{ij}= \E_Y[ \partial_w^{2} g(Y_{ij},0)\given \vect{x}^0]$.
	\end{enumerate}
	Using Taylor's theorem, these parameters may be expressed in terms of the information parameters under the null model,
	\begin{enumerate}
		\item With $w^{0}_{ij}= x^0_{i}x^0_{j}/\sqrt{N}$, and recalling the fact that $\|\partial^2_w g \|_\infty, \| \partial^2_w g^0 \|_\infty < \infty$ as well as the fourth moment bound of $\partial_w g$ and $\partial_w g^0$ for $g, g^0 \in \cF_0$, combined with a Taylor expansion to second order gives:
		\begin{align*}
			&\mu_{ij} = \E_{Y}[ \partial_w g(Y_{ij},0) |\bx^0 ] = \int \partial_w g(y,0) e^{g^0(y, w^0_{ij}) } \, dy  = 
			\\
            &\int \partial_w g(y,0) \Big(1 + \partial_w g^0(y,0) w^0_{ij}  + \big( (\partial_w g^0(y,0))^2 + \partial_w^{2}g^0(y,0)   \big) \frac{(w^0_{ij})^2}{2} \Big) e^{g^0(Y,0)}\, dy
			\\
            &= \E_{0} \partial_w g(Y,0) + \frac{x_i^0 x_j^0}{\sqrt{N}} \E_{0} \partial_w g(Y,0)g^0_w(Y,0) + \frac{\| \partial_w^{2}g(Y,0)\| K}{N}
			\\
            &= \beta_4 + \frac{x_i^0 x_j^0}{\sqrt{N}} \beta_2 +  O\bigg( \frac{ \sqrt{\E_0[ g(Y,0)^2  ] \E_0[ g_0(Y,0)^2  ] }  }{N} \bigg)  + O \bigg( \frac{\| \partial_w^{2}g^0(Y,0)\| K}{N} \bigg)
		\end{align*}
		\item Similarly, expanding the density we see that
		\begin{align*}
			\sigma^2_{ij} &= \E_{Y} [(\partial_w g(Y,0))^2 - \mu_{ij}^2 \given \vect{x}^0]
			= \int (\partial_w g(Y,0))^2 e^{g^0(Y,0)} (1 + O(N^{-1/2}) \, dy - \mu_{ij}^2
			\\&= \E_{0} (\partial_w g(Y,0))^2 - (\E_{0} \partial_w g(Y_{ij},0) )^2 + \E_{0} [ \partial_w g(Y,0)^2 \partial_w g^0(Y,0) ]  w_{ij}^0 
			\\&=  \beta_1  + O\bigg( \frac{L \sqrt{\E_0[ g(Y,0)^4  ] \E_0[ g_0(Y,0)^2  ] }  }{N^{1/2}} \bigg) \, , 
		\end{align*}
		
		\item Lastly, we have
		\begin{align*}
			\gamma_{ij} &= \E_{Y}[ \partial_w^{2}g(Y,0) \given \vect{x}^0] 
			= \int \partial_w^{2}g(Y,0) e^{g^0(Y,0)}  + O(N^{-1/2}) \, dy
			\\
           &= -\beta_3  + O\bigg( \frac{\| \partial_w^{2}g(Y,0)\|  K}{N^{1/2}} \bigg)
		\end{align*}
	\end{enumerate}
	Heuristically, one can expect that in the large $N$ limit, the first disorder term in $F_N^{(2)}$ can be approximated with a Gaussian with matching mean and variance, we have that informally
	\[
	g( Y_{ij},0)  \approx \sigma_{ij} g_{ij} + \mu_{ij} \approx \sqrt{\beta_1} g_{ij} + \beta_4 + \frac{x_i^0 x_j^0}{\sqrt{N}} \beta_2.
	\]
	Since we have assumed that $\E_{\pP_0} [ (\partial_w f(Y ,0) )^3 ]$ is finite, the substitution can be made precise using a standard approximate Gaussian integration by parts  argument---as was applied, e.g., to prove universality for the SK model in \cite{CarmonaHu}
	to conclude that 
	\begin{equation}\label{eq:boundf2minusf3}
		|F^{(2)}_N(g;S,M,v)  - F^{(3)}_N(g;S,M,v) | \leq \frac{K( g,g^0 ) }{\sqrt{N}} \,
	\end{equation}
	where the constant $K( g,g^0 )$ only depends on the quantities appearing in $\cF_0$, all of which are uniformly bounded. 
	This proof is verbatim as the one that appears in \cite[Lemma~3.4]{nonbayes}, and we give a quick sketch. Begin by forming the interpolating Hamiltonian
	\begin{align*}
		H_{N}(\bx,t) &= \Bigg[\frac{1}{\sqrt{N} } \sum_{i \leq j} \Big( \sqrt{t} ( \partial_{w}g(Y_{ij},0) - \mu_{ij}) + \sqrt{1 - t} \sigma_{ij} W_{ij} \Big) x_i x_j + \frac{1}{\sqrt{N}} \sum_{i \leq j} \mu_{ij}  x_i x_j 
        \\
        &\qquad + \frac{1}{2N} \sum_{i \leq j} \gamma_{ij}  (x_i x_j)^2 \Bigg]
        \\
	&= \frac{1}{\sqrt{N} } \sum_{i \leq j} \sigma_{ij} \Big( \sqrt{t} \tilde W_{ij}+ \sqrt{1 - t} W_{ij} \Big)  x_i x_j + \frac{1}{\sqrt{N}} \sum_{i \leq j} \mu_{ij}  x_i x_j  + \frac{1}{2N} \sum_{i \leq j} \gamma_{ij}  (x_i x_j)^2
	\end{align*}
	where we defined $\tilde W_{ij} = \sigma_{ij}^{-1} ( \partial_w g( Y_{ij},0) - \mu_{ij})$ to simplify notation and $W_{ij}$ are independent standard gaussians. Note that conditionally on $\vect{x}^0$, $\tilde{W}_{ij}$ has mean zero and variance $1$. 
    
    We define the interpolating free energy as,
	\[
	\phi(t) = \frac{1}{N} \E_{Y}\big[ \log \big[ \E_{X} \1(\vect{x}\in A) \exp(  H_N(\vect{x},t) )\big] \big|\vect{x}^0\big] \, ,
	\]
    from here one takes the derivative in $t$, and using gaussian integration by parts and approximate integration by parts  (see for example, \cite[Lemma~3.7]{PBook}) one shows that the derivative is $O(1/\sqrt{N})$. Integrating $\phi$ we conclude that $F_N^{(2)}$ and $F_N^{(3)}$ have the same limit.

    We emphasize that in contrast to the proof in ~\cite{nonbayes} we do not require that $\mu_{ij} = O(1/\sqrt{N})$. This term is of higher order, but it has no effect on this computation, since it does not appear with a $t$ coefficient and hence has no effect on the derivative.

		\textit{Step 4 - Summary:} We can use the triangle inequality and the estimates in steps 1 to steps 3, combined with the fact that $F^{(3)}_N(g;S,M,v)  = F_N(\bar\beta;S,M,v)$ to conclude the statement of the result.
	\end{proof}
	
	As a consequence of Proposition ~\ref{prop:universality1}, we obtain the following universality result for pseudo maximum likelihood estimation: 
	\begin{prop}[Universality of the Ground State] \label{prop:groundstateuniversality1}
		If our model is well-scored, then 
		\[
		\big| \cL_N^g(S,M,v) - \cL_N^{\bar \beta}(S,M,v) \big| =  O_L( N^{-1/2} ) + o_L(1)
		\]
		where $O_L( N^{-1/2} )$ is a term that goes to $0$ at rate $N^{-\frac{1}{2}}$ for every fixed $L$, $o_L(1) \to 0$ uniformly over $N$ and $\bar\beta=(\beta_1,\beta_2,\beta_3,\beta_4)$ are the information parameters defined in \eqref{eq:fisherscore1}, \eqref{eq:fisherscore2}, \eqref{eq:fisherscore3},  \eqref{eq:fisherscore4}.
	\end{prop}
	
	\begin{proof}
		This follows from a direct application of the universality at finite temperatures and a careful analysis of the dependencies of the error terms on the norms of $g$. By Proposition~\ref{prop:universality1} we have

		\[
		\big|  F_N(L g;S,M,v) - F_N(\bar\beta_L;S,M,v) \big| \leq \frac{K(Lg,g^0)}{\sqrt{N} } \, ,
		\]
		where $K$ is a univeral constant that only depends on $g$ and $g^0$ and $\bar\beta_L =  (L^2\beta_1, \beta_2,\beta_3,\beta_4)$. Then the bounds in ~\eqref{eq:boundsgroundstate} implies that
		\[
		| 	\cL_N^g(S,M,v) - 	\cL_N^{\bar \beta}(S,M,v) | \leq 	\big| \frac{1}{L} F_N(L g) - \frac{1}{L} F_N(L \bar\beta) \big| + o_L(1) \leq O_L( N^{-1/2} )  + o_L(1) \, .
		\]
		Taking $N$ to infinity, followed by $L$ to infinity then gives the desired result.  
	\end{proof}

	\section{Variational Formula with Constrained Sample Mean}\label{sec:devvarI}

  In this section we prove Theorem ~\ref{th:finitetempcorrected}. This amounts to extending the earlier result with constrained overlaps in \cite[Theorem~2.6]{nonbayes} with an additional sample mean constraint.  However, in the maximum likelihood setting the signal $\vect{x}_0$ is fixed and not random, so the technical details in this proof are simplified despite the inclusion of an extra constraint. The case without a sample mean constraint, which will be required for regular models, is a direct consequence of \cite[Theorem~2.6]{nonbayes}  and will be stated at the end of this section.

	The proof will be split into two parts. We first begin with a proof of the upper bound of the constrained free entropy. The following proofs are stated in terms of a quantity called the Ruelle probability cascades \cite[Chapter~2]{PBook}. A quick summary of the notation is provided for convenience in Appendix~\ref{sec:RPC}.
	\begin{prop}[Large Deviation Upper Bound of the Free Energy] \label{prop:upbd}
		There exists a universal  finite constant $L$ such that for every $S,M,v\in \cC$, and every  real numbers $\mu,\lambda,\rho$, we have 
		\begin{align*}
			F_{N}(\bar\beta:  \Omega_\epsilon(S,M,v)) &\leq 
			- \lambda S - \mu M - \rho v + \frac{1}{N} \sum_{i = 1}^N \Phi_{\lambda,\mu,\zeta}(0,0;x_i^0) - \frac{\beta_1^2}{2} \int_0^S t\zeta(t) \, dt 
			\\&+ \frac{\beta_2}{2} M^{2} - \frac{\beta_3}{4} S^{2} + \frac{\alpha}{2} v^2  +   L\varepsilon(|\mu|+|\lambda|+|\rho|)+o_{N,\epsilon}(1) \,
		\end{align*}
		where $\Phi_{\lambda,\mu,\zeta}(0,0;x_i^0)$ solves \eqref{eq:parisipdefinite}. 	Moreover $o_{N,\epsilon}(1)=O(\varepsilon)+O(N^{-1})$ is independent of $\lambda,\mu,\rho$.
	\end{prop}
	\begin{proof}
    This proof follows from the classical Guerra interpolation argument and holds verbatim as the one appearing in \cite[Section~4]{nonbayes}. There is an extra constraint parameter, but this is dealt by introducing Lagrange multipliers for the sum $\sum_{i = 1}^N x_i$. One key difference is that the upper bound is written in terms of the Ruelle probability cascades with an extra Lagrange multiplier parameter $\sum_{i = 1}^N \rho x_i$, but this representation is equivalent to \eqref{eq:parisipdefinite} (see \cite[Chapter~4]{PBook}). 
	\end{proof}
	
	We now claim that the upper bound is sharp in the sense that after one minimizes over the parameters $\mu, \lambda, \rho$, the upper bound is equal to the constrained integral. 

    For $(\lambda,\mu,\rho)\in\mathbb R^{3}$, consider the annealed log Laplace transform
	$$\Lambda(\lambda,\mu,\rho):=\int \log\left( \int  e^{ \lambda x^2 + \mu x x^0 + \rho x } \, d \pP_X(x)\right)d\pQ(x^{0})$$
	and  consider the rate function on $\mathbb R^{3}$ given by 
	\begin{align*}
		\cI(S,M,v) &=\sup_{(\lambda,\mu,\rho)\in\mathbb R^{2}}\{ I_{S,M,v}(\lambda,\mu,\rho)
		\} \, ,
        \\
        &\text{with} \  I_{S,M,v}(\lambda,\mu,\rho)
        =\lambda S + \mu M + \rho v-\Lambda(\lambda,\mu,\rho)\,.
	\end{align*}
	This quantity computes the entropy of the set $\Omega_\epsilon(S,M,v)$ under $\pP_X$ by \cite[Proposition~5.3]{nonbayes}.

    In the lemma that follows, we will use the Ruelle probability cascades, we denote by $v_{\alpha}$ the weights of the RPC corresponding to a measure $\gamma$ (see Appendix ~\ref{sec:RPC}). 

	\begin{lem}[Sharp Lower Bound] \label{lem:sharpupbd} 
		For $(S ,M,v) \in \cC$ and any $\epsilon,\delta > 0$ small enough, 
		\begin{align}
			&\liminf_{N \to \infty} \frac{1}{N} \E_{Z}  \log \sum_\alpha v_\alpha \int_{\Omega_\epsilon(S,M,v)} e^{\sum_{i \leq N} \beta_1 Z_i(\alpha) x_i } \, d \pP_X^{\otimes N} (\bx)
			\nonumber \\&\geq \inf_{\mu,\lambda,\rho} \bigg( -\lambda S - \mu M - \rho v + \E_{Z,\pQ} \log \sum_\alpha v_\alpha  \int e^{\beta_1 Z(\alpha) x +\lambda x^{2} +\mu x x^{0} + \rho x
			} \, d \pP_X (\bx)  \bigg).\label{lbge}
		\end{align} 
        Moreover, the right hand side is equal to $-\infty$ if $\cI(S,M,v)=\infty$.		Furthermore, if $S,M,v$ belong to the interior of $\cC$, then the minimizer is attained at a unique $\mu$ and $\lambda$, such that $|\mu| + |\lambda| + |\rho| \leq C(S,M,v)$ where the constant $C$ only depends on the distance from $(S,M,v)$ to the boundary. 
	\end{lem}

	\begin{proof}A similar result is proved in \cite[Section 7]{PVS} and \cite[Lemma~5.4]{nonbayes}. We adapt the Gartner-Ellis argument \cite[Section 2.3]{DZ}, taking into account the random density depending on the $Z_{i}$'s.

		We first show that we can restrict ourselves to $(S,M,v)$ with finite entropy because the lower bound in \eqref{lbge} is infinite otherwise. 
		Indeed,
		\begin{align*}
        &\E_{Z} \log \sum_\alpha v_\alpha  \int e^{\beta_1 Z(\alpha) x +\lambda x^{2} +\mu x x^{0} + \rho x} \, d \pP_X (\bx) 
        \\
        &\le
        \E_{Z} \log \sum_\alpha v_\alpha  \int e^{\beta_1 |Z(\alpha)|C} +\E_{\pQ}\log \int e^{\lambda x^{2} +\mu x x^{0} + \rho x
		} \, d \pP_X (\bx) \, ,
        \end{align*}
        and 
		$
		\E\log \sum_\alpha v_\alpha e^{ \beta_1 |Z(\alpha)| C}
		$
		is bounded uniformly, by the recursive properties of the Ruelle probability cascades \cite[Equation 2.51]{PBook},
		\[
		\E e^{ | \sum_{k = 1}^r (  Q^2_k - Q^2_{k - 1}  )^{1/2} z_{ k} |  C   } < \infty.
		\]
		Therefore there exists a finite constant $L$ such that
		\begin{align*}
         &\inf_{\mu,\lambda} \bigg( -\lambda S - \mu M - \rho v + \E_{Z,\pQ} \log \sum_\alpha v_\alpha  \int e^{\beta_1 Z(\alpha) x +\lambda x^{2} +\mu x x^{0}
			+ \rho x} \, d \pP_X (\bx)  \bigg)
            \\
            &\le -\cI(S,M,v)+ L.
          \end{align*}  
		Thus we may restrict to values of $(S,M,v)$ with finite entropy. 
		
		We now adapt the Gartner-Ellis argument to our setting. It is based on a large deviation upper bound for certain tilted measures. Namely let $\lambda,\mu, \rho \in\mathbb R^{3}$. We will show for every $(S,M,v)\in [0,C^{2}]\times [-C^{2},C^{2}] \times [-C,C]$,
		
		\begin{eqnarray}\label{ldubt0}
			&&\limsup_{N \to \infty} \frac{1}{N} \E_{Z,x^{0}}  \log\frac{ \sum_\alpha v_\alpha \int_{\Omega_\epsilon(S,M,v)} e^{\sum_{i \leq N} (\beta_1 Z_i(\alpha) x_i +\lambda x_{i}^{2}+\mu x_{i}x_{i}^{0} + \rho x_i ) } \, d \pP_X^{\otimes N} (\bx)}{ \sum_\alpha v_\alpha \int  e^{\sum_{i \leq N} (\beta_1 Z_i(\alpha) x_i +\lambda x_{i}^{2}+\mu x_{i}x_{i}^{0}+ \rho x_i ) } \, d \pP_X^{\otimes N} (\bx)}\nonumber\\
			&&\qquad
			\le -\Lambda^{*}_{\lambda,\mu}(S,M,v)+O(\epsilon) \, ,
            \end{eqnarray}
		with 
		$$\Lambda^{*}_{\lambda,\mu,v}(S,M,\rho)=-\lambda S-\mu M - \rho v +\Lambda(\mu,\lambda,v)+\sup_{\lambda',\mu', \rho'}\{\lambda'S+\mu'  M + \rho' v -\Lambda(\lambda',\mu', \rho')\} \, , 
        $$
		where
		$$
        \Lambda(\lambda,\mu,\rho)= \E_{Z,\pQ} \log \sum_\alpha v_\alpha  \int e^{\beta_1 Z(\alpha) x +\lambda x^{2} +\mu x x^{0} + \rho x
		} \, d \pP_X (\bx) \, . 
        $$
		We denote in short $\Lambda^{*}=\Lambda^{*}_{0,0,0}$.
		To see this claim \eqref{ldubt0}, start by observing that as a direct consequence of the fact that the $v_{\alpha}$ are non-negative, and  almost surely we have
		\begin{align*}
        &\int_{\Omega_\epsilon(S,M)} e^{\sum_{i \leq N} (\beta_1 Z_i(\alpha) x_i +\lambda x_{i}^{2}+\mu x_{i}x_{i}^{0} + \rho x_i) } \, d \pP_X^{\otimes N} (\bx)\qquad\qquad\qquad
       \\
          &\le e^{N(\lambda-\lambda') S+N(\mu-\mu') M + N (\rho - \rho') v+NO(\epsilon)}\int  e^{\sum_{i \leq N} (\beta_1 Z_i(\alpha) x_i +\lambda' x_{i}^{2}+\mu' x_{i}x_{i}^{0}+ \rho' x_i ) } \, d \pP_X^{\otimes N} (\bx) \, .
        \end{align*}
		
        We pause to  introduce the notion of exposed points: $(S,M,v)$ is said to be \emph{exposed} if there exists $(\lambda,\mu,\rho)$ such that for every $(S',M', v')\neq (S,M,v)$ we have 
		\begin{align*}\label{defcrit} 
			\lambda S+\mu M + \rho v -\Lambda^{*}(S,M,v)&>\lambda S'+\mu M' + \rho v' -\Lambda^{*}(S',M',v')
            \\
            &=-\Lambda^{*}_{\lambda,\mu,\rho}(S',M',v')+\Lambda(0,0,0) \,.
            \end{align*}
		The set $(\lambda,\mu,\rho)$ is called an \emph{exposing hyperplane}.
		We first prove \eqref{ldubt0} for an exposed point $(S,M,v)$ with exposing hyperplane $(\lambda,\mu,\rho)$ by showing that the associated tilted measure puts some mass on a neighborhood of $(S,M,v)$, see \eqref{tot}. To see this, 
		we first observe that for every $(S',M',v')\neq (S,M,v)$, 
		\begin{align*}
			\Lambda^{*}_{\lambda,\mu,\rho}(S',M',v')&= \Lambda^{*}(S',M',v')-(\lambda S'+\mu M' + \rho v' -\Lambda(\mu,\lambda,\rho)+\Lambda(0,0,0) )\\
			&\ge \Lambda^{*}(S',M',v')-(\lambda (S'-S)+\mu (M'-M) + \rho (v - v') +\Lambda^{*}(S,M,v)) 
            \\
            &>0\end{align*}
		Now, it is easy to see that $\Lambda^{*}_{\lambda,\mu,\rho}$ is a good rate function so that it achieves its minimum value on the closure $\bar\Omega_\epsilon(S,M,v	)^{c}$  of $\Omega_\epsilon(S,M,v)^{c}$. Hence 
		$\inf_{\bar\Omega_\epsilon(S,M,v)^{c}}\Lambda^{*}_{\lambda,\mu,\rho}\ge \kappa>0$. Moreover, 
		we can cover  $\bar\Omega_\epsilon(S,M,v)^{c}$   by a  union of  finitely many balls  $(B_{j})_{j\le K}$  so that for each $j\le K$
		
		\begin{eqnarray}\label{ldubt}
			&&\limsup_{N \to \infty} \frac{1}{N} \E_{Z}  \log\frac{ \sum_\alpha v_\alpha \int_{(S_N(x),M_N(x),\bar x)\in B_{j}} e^{\sum_{i \leq N} (\beta_1 Z_i(\alpha) x_i +\lambda x_{i}^{2}+\mu x_{i}x_{i}^{0} + \rho x_i) } \, d \pP_X^{\otimes N} (\bx)}{ \sum_\alpha v_\alpha \int  e^{\sum_{i \leq N} (\beta_1 Z_i(\alpha) x_i +\lambda x_{i}^{2}+\mu x_{i}x_{i}^{0}+ \rho x_i) } \, d \pP_X^{\otimes N} (\bx)}\nonumber\\
			&&\qquad
			\le -\kappa+O(\delta) \, .
            \end{eqnarray}
		Therefore, there exists $\kappa=\kappa_{\epsilon}>0$ such that
		\begin{eqnarray}\label{ldubt2}
			&&\limsup_{N \to \infty} \frac{1}{N} \E_{Z}  \log\frac{ \sum_\alpha v_\alpha \int_{\bar \Omega_\epsilon(S,M)^{c}} e^{\sum_{i \leq N} (\beta_1 Z_i(\alpha) x_i +\lambda x_{i}^{2}+\mu x_{i}x_{i}^{0} + \rho x_i) } \, d \pP_X^{\otimes N} (\bx)}{ \sum_\alpha v_\alpha \int  e^{\sum_{i \leq N} (\beta_1 Z_i(\alpha) x_i +\lambda x_{i}^{2}+\mu x_{i}x_{i}^{0} + \rho x_i) } \, d \pP_X^{\otimes N} (\bx)}\nonumber\\
			&&\le \limsup_{N \to \infty} \frac{1}{N} \E_{Z}  \log \sum_{j\le K} \frac{ \sum_\alpha v_\alpha \int_{B_{j}} e^{\sum_{i \leq N} (\beta_1 Z_i(\alpha) x_i +\lambda x_{i}^{2}+\mu x_{i}x_{i}^{0} + \rho x_i ) } \, d \pP_X^{\otimes N} (\bx)}{ \sum_\alpha v_\alpha \int  e^{\sum_{i \leq N} (\beta_1 Z_i(\alpha) x_i +\lambda x_{i}^{2}+\mu x_{i}x_{i}^{0} + \rho x_i ) } \, d \pP_X^{\otimes N} (\bx)}\nonumber\\
			&&\le  -\kappa \, , 
            \end{eqnarray}
		where in the last step we use Lemma \ref{lem:upbdRPC} to pull the sum outside of the logarithm.  Applying Lemma \ref{lem:upbdRPC} again, we conclude that
		\begin{align*}
			0&=\liminf_{N \to \infty} \frac{1}{N} \E_{Z} \log\frac{ \sum_\alpha v_\alpha \int e^{\sum_{i \leq N} (\beta_1 Z_i(\alpha) x_i +\lambda x_{i}^{2}+\mu x_{i}x_{i}^{0} + \rho x_i) } \, d \pP_X^{\otimes N} (\bx)}{ \sum_\alpha v_\alpha \int  e^{\sum_{i \leq N} (\beta_1 Z_i(\alpha) x_i +\lambda x_{i}^{2}+\mu x_{i}x_{i}^{0} + \rho x_i ) } \, d \pP_X^{\otimes N} (\bx)} 
            \\
			&\leq \max\left\{\liminf_{N \to \infty} \frac{\E_Z}{N} \log\frac{ \sum_\alpha v_\alpha \int_{\Omega_\epsilon(S,M)} e^{\sum_{i \leq N} (\beta_1 Z_i(\alpha) x_i +\lambda x_{i}^{2}+\mu x_{i}x_{i}^{0}+ \rho x_i) } \, d \pP_X^{\otimes N} (\bx)}{ \sum_\alpha v_\alpha \int  e^{\sum_{i \leq N} (\beta_1 Z_i(\alpha) x_i +\lambda x_{i}^{2}+\mu x_{i}x_{i}^{0}+ \rho x_i) } \, d \pP_X^{\otimes N} (\bx)}, -\kappa+\delta\right\}
		\end{align*} 
		and therefore for $\delta$ small enough (depending on $\epsilon$)
		\begin{equation}\label{tot}
			\liminf_{N \to \infty} \frac{1}{N} \E_{Z} \log\frac{ \sum_\alpha v_\alpha \int_{ \Omega_\epsilon(S,M,v)} e^{\sum_{i \leq N} (\beta_1 Z_i(\alpha) x_i +\lambda x_{i}^{2}+\mu x_{i}x_{i}^{0} + \rho x_i) } \, d \pP_X^{\otimes N} (\bx)}{ \sum_\alpha v_\alpha \int  e^{\sum_{i \leq N} (\beta_1 Z_i(\alpha) x_i +\lambda x_{i}^{2}+\mu x_{i}x_{i}^{0} + \rho x_i) } \, d \pP_X^{\otimes N}} \ge 0.\end{equation}
		 \eqref{lbge} then follows. Indeed, by H\"older's inequality
		\begin{align}
			&\frac{1}{N} \E_{Z} \log \sum_\alpha v_\alpha \int_{\Omega_\epsilon(S,M,v)} e^{\sum_{i \leq N} \beta_1 Z_i(\alpha) x_i } \, d \pP_X^{\otimes N} (\bx)
			\nonumber\\&\geq  -\lambda S - \mu M - \rho v + \frac{1}{N} \E_{Z} \log \sum_\alpha v_\alpha  \int e^{\sum_{i \leq N} (\beta_1 Z_i(\alpha) x_i +\lambda x_{i}^{2}+\mu x_{i}x_{i}^{0}  + \rho x_i 
			} \, d \pP^{\otimes N}_X (\bx)  \\
			& +\frac{1}{N} \E_{Z} \log\frac{ \sum_\alpha v_\alpha \int_{\bar \Omega_\epsilon(S,M)} e^{\sum_{i \leq N} (\beta_1 Z_i(\alpha) x_i +\lambda x_{i}^{2}+\mu x_{i}x_{i}^{0}  + \rho x_i ) } \, d \pP_X^{\otimes N} (\bx)}{ \sum_\alpha v_\alpha \int  e^{\sum_{i \leq N} (\beta_1 Z_i(\alpha) x_i +\lambda x_{i}^{2}+\mu x_{i}x_{i}^{0} + \rho x_i) } \, d \pP_X^{\otimes N}} +O(\epsilon).\label{lbge2}
		\end{align}	
		We also have that 
        \begin{align*}
            \frac{1}{N} \E_{Z} \log \sum_\alpha v_\alpha  \int e^{\sum_{i \leq N} (\beta_1 Z_i(\alpha) x_i +\lambda x_{i}^{2}+\mu x_{i}x_{i}^{0}  + \rho x_i 
			} \, d \pP^{\otimes N}_X (\bx)
\to\Lambda(\lambda,\mu,\rho)
        \end{align*}
	 as $N \to \infty$. Hence, letting $N$ go to infinity, $\delta$ to zero and then $\epsilon$ to zero we arrive at the desired statement.
		
		To conclude that the lower bound holds not only for exposed points we appeal to Rockafellar's lemma, see \cite[Lemma 2.3.12]{DZ}, which shows that it suffices to prove that $\Lambda$ is essentially smooth, lower semi-continuous and convex. This follows as $\pP_{X}$ and $\pQ$ are compactly supported. Consequently, the relative interior of the set of points where $\Lambda^{*}$ is finite is included in the set of exposed points, and so by our earlier reduction to points with finite entropy, the Lemma is proven. 
	\end{proof}
	
	The rest of the proof of the lower bound can be adapted from \cite{nonbayes}. The main difference is that in our setting $\vect{x}_0$ is non-random, while in the Bayesian setting, there is a prior on $\vect{x_0}$. The current setting with non-random  $\vect{x}_0$ is actually much simpler, and the lower bound can be proved using the classical perturbations without localizing $\vect{x}_0$ around typical values. We sketch the key steps below.
	
	\begin{prop}[Lower Bound of the Free Energy]\label{prop:lowerboundFE} For any real numbers $\beta_1,\beta_2,\beta_3$,
		for any $(S, M,v) \in \cC$,  for any $\epsilon>0$, we have
		$$ \liminf_{N \to \infty}  F_{N}(\bar\beta,\epsilon;S,M,v)  \ge \varphi_{\bar\beta}(S,M,v)+O(\epsilon)$$	
	\end{prop}
	
	\begin{proof} The key ideas of the proof is similar to the ones used to derive the lower bound of the Sherrington--Kirkpatrick model. The approximation techniques used to deal with the random constraint set in \cite{nonbayes} is also not needed in this setting, since $\vect{x}_0$ is fixed and non-random. We summarize the key steps.
		\\
		\textit{Step 1:} We first introduce a perturbation of the likelihood function that will allow us to characterize its limiting distribution. To introduce the perturbed Hamiltonian let us first fix the self-overlap by setting
		
		\begin{equation}\label{eq:modifiedcoords}
			\hat \bx = \frac{\sqrt{SN}}{ \| \bx \|_2 } \bx
		\end{equation}
		The entries of $\hat \bx$ are still uniformly bounded for $\bx$ so that $S_N(x)=\frac{1}{N}\|x\|_{2}^{2}$ is at $\epsilon$ distance of $S$, provided $\epsilon<S/2$. Throughout $D$ will denote such a uniform bound (which depends on $S$ and $C$). 
		For $p \geq 1$, consider
		\[
		g_{p}(\hat\bx) = \frac{1}{N^{p/2}} \sum_{\iii} g_{\iii}  \hat x_{i_1} \cdots  \hat x_{i_p} \, ,
		\]
		and the Gaussian process
		\begin{equation}\label{eq:pertg}
			g(\hat \bx) = \sum_{p \geq 1} 2^{-p} D^{- p } t_p  g_p(\hat \bx) \, ,
		\end{equation}
		where the $g_{\iii}$ are independent standard Gaussians and $(t_p)_{p \geq 1}$ is a sequence of parameters such that $t_p \in [0,3]$ for all $p \geq 1$.  Notice that the covariance is bounded
		\begin{equation}\label{boundcov}
			\E g(\hat\bx^1) g(\hat\bx^2) = \sum_{p \geq 1} 4^{-p} D^{-2p} t_p^2 \bigg(\frac{1}{N}\sum_{i=1}^{N}\hat x_{i}^{1}\hat x_{i}^{2}\bigg)^{p}\leq \sum_{p \geq 1} 4^{-p} D^{-2p} t_p^2 D^{2p}\le 3 \, ,
		\end{equation}
		where the first inequality uses  $R_{1,2}= \frac{1}{N}\sum \hat x^{1}_{i}\hat x^{2}_{i} \leq C^2$. For $s>0$, we define the interpolating Hamiltonian as
		\begin{equation}\label{eq:parthamil}
			H^\pert_{N}(\bx) = H_{N}( \bx) + s g(\hat \bx).
		\end{equation}
		A key consequence is that under the perturbed likelihood function, samples from the posterior will satisfy a concentration inequality called the Ghirlanda--Guerra identities (see \cite[Section~3]{PBook}). 
			\begin{theo}[Ghirlanda--Guerra Identities]\label{thm:GGI}Let $\hat R_{k,\ell}=\frac{1}{N}\sum_{i=1}^{N}\hat x_{i}^{k}\hat x^{\ell}_{i}$. 
			If $s = N^\gamma$ for $1/4 < \gamma < 1/2$, then
			\[
			\lim_{N \to \infty} \E_u \bigg| \E \langle f  \hat R_{1,n+1}^p \rangle - \frac{1}{n} \E  \langle f \rangle \E  \langle \hat R_{1,2}^p \rangle - \frac{1}{n} \sum_{\ell = 2}^n \E \langle f\hat R^p_{1,\ell} \rangle \bigg| = 0 \, ,
			\]
			for any $p \geq 1$, $n \geq 2$ and bounded measurable function $f$ of the $n \times n$ sub array of the overlaps. 
		\end{theo}
		\noindent\textit{Step 2:} We now compute the limit by showing that the limit can be expressed as functions of samples from the posterior, which we have a limiting characterization of. This is commonly known as the Aizenman--Sims--Starr scheme or cavity computations in statistical physics. We have for every $n \geq 1$,
		\[
		\liminf_{N \to \infty}  F_{N}(\bar\beta,\epsilon;S,M,v) \geq \liminf_{N \to \infty} \bigg( \frac{1}{n} \E \log Z_{N + n} - \frac{1}{N} \E \log Z_{N} \bigg) \, ,
 		\]
 		where
 		\[
 		Z_{N} =  \int \1(	\Omega_\epsilon(S,M,v) ) e^{ H_N^{\bar \beta,\alpha}( \vect{x}) } \, d\pP_X^{\otimes N}(\vect{x})  \, . 
 		\]
 		Let $(\vect{x},\by) \in \R^{N + n}$. We decompose $H_{N + 1}$ into terms that depend on the cavity coordinate $x$ and its bulk terms. 
 		Consider the following cavity fields defined with respect to the modified coordinates $\hat x_{i}=\sqrt{(N+n)S} x_{i}/\|x\|_{2}$ (see \eqref{eq:modifiedcoords}):
 		\begin{equation}\label{eq:cav1}
 			H_{N,n}^\pert( \bx)  := \sum_{1\le i<j\le N} \beta \frac{W_{ij}}{ \sqrt{ (N + n)}}  x_i x_j + s g_N(\hat \bx) ,
 		\end{equation}
 		\begin{equation}\label{eq:cav2}
 			z_i(\hat \bx) =\frac{\beta}{ \sqrt{ N}}   \sum_{j = 1}^N  W_{j,N + i} \hat x_j ,
 		\end{equation}
 		\begin{equation}\label{eq:cav3}
 			y(\hat \bx) =\frac{ \sqrt{n} \beta }{N}  \sum_{1\le i<j\le N} W_{ij} \hat x_i \hat x_j \, .
 		\end{equation}
 		Using an interpolation argument \cite[Section~3.5]{PBook} implies that we can replace $H_{N + n}(\bx, \by)$ with $\sum_{i = 1}^N	z_i(\hat \bx) \hat y_i + 	H_{N,n}^\pert( \bx) $ and $H_{N}(\bx, \by)$ with $y(\hat \bx)   + 	H_{N,n}^\pert( \bx)$,  which gives us the following lower bound
 		\begin{align*}
 			&\frac{1}{n} \E \log Z_{N + n} - \frac{1}{N} \E \log Z_{N}
 			\\&\geq \frac{1}{n} \bigg( \E \log \bigg\langle \int_{\Omega_{\epsilon}(S,M,v)}  e^{\sum_{i = 1}^n z_i(\hat \bx) y_i}   \, d \pP^{\otimes n}_X(\by) \bigg\rangle_{N,n}^\pert 
 			- \E  \log \bigg\langle e^{y(\hat \bx)} \bigg\rangle_{N,n}^\pert \bigg) + o(1)	 \, ,
 		\end{align*}
 		where $\langle \cdot \rangle_{N,n}^\pert$ denotes the average with respect to 	$H_{N,n}^\pert( \bx)$. This lower bound can be approximated by a continuous function of finitely many samples from $\langle \cdot \rangle_{N,n}^\pert$ (see Lemma~\ref{lem:continuityarray}).
        \hfill
        
		\noindent\textit{Step 3:} We now identify the limit of this lower bound. Since $R^\infty$ satisfies the Ghirlanda--Guerra identities, the distribution of the entire array is determined by $\zeta(t) = \pP(R^\infty_{1,2} \leq t)$ \cite[Theorem~2.13 and Theorem~2.17]{PBook}. Approximate $\zeta$ by an $r$-atomic measure such that their CDFs satisfy 
		\[
		\int | \zeta(t) - \mu(t) | \, dt < \epsilon.
		\]
		The density function $\mu$ of a measure can be encoded by the parameters
		\begin{align}\label{eq:zetaseqlwbd}
			\zeta_{-1} &= 0 < \zeta_0< \dots < \zeta_{r-1} \, ,\\
			0 &= Q_0 \leq Q_1 \leq \dots \leq Q_{r-1} \leq Q_r =  S \, .
		\end{align}
		That is, these sequences define the density function
		\[
		\mu(Q) = \zeta_k \qquad \text{for} \qquad Q_k \leq Q < Q_{k + 1} \, .
		\]
		Let $v_\alpha$ denote the weights of the Ruelle probability cascades corresponding to the sequence \eqref{eq:zetaseqlwbd}. If $(\alpha^\ell)_{\ell \geq 1}$ are samples from the Ruelle probability cascades, then $\pP( \alpha^1 \wedge \alpha^2 \leq t) = \mu(t)$ by construction. This gives us an explicit way to construct the off-diagonal entries of the overlap array in the limit.  We define Gaussian processes $Z(\alpha)$ and $Y(\alpha)$ with covariance
		\[
		\E Z(\alpha^1) Z(\alpha^2) = Q_{\alpha^1 \wedge \alpha^2} \quad \E Y(\alpha^1)  Y(\alpha^2)  = \frac{1}{2} Q^2_{\alpha^1 \wedge \alpha^2}
		\]
		and let $Z_i$ for $1 \leq i \leq n$ denote independent copies of $Z$. The functionals
		\begin{align*}
		f_n^Z(\mu) &= \frac{1}{n} \E \log \sum_\alpha v_\alpha \int_{\Omega_{\epsilon}(S,M,v)}  e^{\sum_{i \leq n} \beta Z_i(\alpha) y_i   } \, d \pP_X^{\otimes n} (\by)\\
		f_n^Y(\mu) &= \frac{1}{n} \E \log \sum_\alpha v_\alpha  e^{ \sqrt{n} \beta Y(\alpha)   } \, ,
		\end{align*}
		are of the same form as the functionals in Lemma~\ref{lem:continuityarray} because they depend on the overlap array in exactly the same way. Furthermore, one can show that they are Lipschitz continuous in $\mu$ with respect to the $L^1$ distance (see the proof of \cite[Lemma~4.1]{PBook} or \cite[Proposition~6.1]{HJBook}).
	\hfill
    
    \textit{Step 4:} To remove the constraint and identify the limit with its matching upper bound, we can apply Lemma~\ref{lem:sharpupbd} to finish the proof.
	\end{proof}
	
	\subsection{Simplification when $\alpha = 0$}
	
	When $\alpha = 0$, which corresponds to regular models, then the variational formula can be expressed in a simpler form. When $\alpha = 0$, the constraint on $v$ is unnecessary and we instead define
	\begin{equation}\label{eq:parisipdefinite2}
		\begin{cases}
			\partial_t \Phi_{\zeta,\lambda,\mu,0} = - \frac{\beta_1^2}{4} ( \partial_{y}^2 \Phi_\zeta + \zeta( [0,t] ) ( \partial_y \Phi_\zeta )^2) & (t,y) \in (0,S) \times \R \\
			\Phi_{\zeta,\lambda,\mu,0}(S,y;x^0) = \log \int e^{yx + \lambda x x^0 + \mu x^2 } \, d \pP_X (x)
		\end{cases}.
	\end{equation}
	Define the Parisi functional
	\begin{align}
		\varphi_{\bar\beta}(S,M) &=\inf_{\mu,\lambda, \zeta} \bigg( \E_{\pQ}[\Phi_{\zeta,\lambda,\mu,0}(0,0;x^0)] - \frac{\beta_1^2}{2} \int_0^S t\zeta(t) \, dt - \mu S - \lambda M 
		+ \frac{\beta_2 M^2}{2}  - \frac{\beta_3 S^2}{4}  \bigg) \label{eq:SMparisifinite}
	\end{align}
	With a slight abuse of notation, we notice that $	\varphi_{\bar\beta}(S,M)$ and $	\varphi_{\bar\beta}(S,M,v)$ are almost identical, but the former no longer depends on $\alpha, v, \rho$. The next theorem shows that the maximum for regular models converges to $\sup \varphi_{\beta}(S,M)$.
	
	\begin{theo}\label{th:finitetemplimit}
		For any $\beta_1, \beta_2, \beta_3$ and $\alpha$ and constraints $(S,M,v)$, we have 
		\[
		\lim_{\epsilon \to 0} \lim_{N \to \infty} F_{N}(\bar\beta,\alpha,\epsilon;S,M,v) = \varphi_{\mathbf{\beta}}(S,M,v) .
		\]
		If $\alpha = 0$, then for any constraints $(S,M) \in \cC$
		\[
		\lim_{\epsilon \to 0} \lim_{N \to \infty} F_{N}(\bar\beta,0,\epsilon;S,M) = \varphi_{\mathbf{\beta}}(S,M)
		\]
	\end{theo}
	\begin{proof}
		This is a direct consequence of \cite[Theorem~2.6]{nonbayes}. One slight difference is that in the setting of MLE, $\vect{x}_0$ is taken to be non-random while in the Bayesian setting $\vect{x}_0$ is drawn from some prior $\pP_*^{\otimes N}(\vect{x}_0)$. However, this is not an issue because the proof of \cite{nonbayes} holds conditionally on a realization of $\vect{x}_0$, and we can simply view $\vect{x}_0$ as a realization of a sample from the limiting measure $\pQ$. 
	\end{proof}
	
	\section{ Gamma Convergence of Local Free Energies}\label{sec:devvarII}

In this section we show that the local quantities computed by taking the limit as $N$ tends to infinity of $\frac{1}{NL} F_N^L(L \bar \beta,\epsilon; S,M,v)$, are $\Gamma$ convergent as $L$ tends to infinity to $\psi_{\bar \beta}(S,M,v)$. We prove this result in the case when $\beta_4=0$ to simplify notation, but note that in the case where $\beta_4 \neq 0$, the modification is simple. We point out where the modifications are necessary as we go along.

	We recall the following result in \cite{nonbayes}. Let $\pP_{X}$ denote either normalized Lebesgue measure on counting measure depending on if $\Omega$ is an interval or discrete. We consider the finite temperature free energy given by:
	\begin{align*}
	F_N(\overline{\beta}) := 
        \lim_{N \to \infty} \frac{1}{N} \E \log \int_{\Omega^N} \exp\left( \frac{\beta_1}{\sqrt{N}} \sum_{i \leq j} g_{ij} x_i x_j + \frac{\beta_2}{N} \sum_{i \leq j} x_i^0 x_j^0 x_i x_j - \frac{\beta_3}{2N} \sum_{i \leq j} x_i^2 x_j^2  \right) d\pP_X^{\otimes N}  \, ,
	\end{align*}
	where $x^0_i$ are the entries of our rank one signal.  Then $F(\overline{\beta})$ can be computed by solving the variational problem in Theorem~\ref{th:finitetemplimit}. defined by 
	\begin{align*}
		\lim_{N \to \infty} F_N(\bar \beta) = F(\overline{\beta}) &= \sup_{(S,M) \in \mathcal{C}}	\varphi_{\bar\beta}(S,M)
	\end{align*}
    where $\varphi$ is defined in \eqref{eq:SMparisifinite}. In order to compute the limit of the pseudo maximum likelihood we must compute the quantity: 
	\[
	\lim_{L \to \infty} \E_{x^0 \sim \mathbb{Q} } \frac{1}{L} F( L \overline{\beta} ) \, ,
	\]
	and we shall do so by means of $\Gamma$ convergence. For fixed $0 \leq t \leq S$, $h,y \in \R$ we define functionals $F_L(\zeta,\lambda,\mu)$ by: 
	\[
	\mathcal{F}_{L,S}(\zeta,\mu,\lambda; t,y,h) = 
	\begin{cases}
		\Phi_{\zeta,\lambda,\mu}^L(t,y) &\text{if} \ \zeta= L \rho(t) dt \\
		+ \infty &\text{otherwise} 
	\end{cases} \, ,
	\]
	where $\Phi_{\zeta,\lambda,\mu}^L$ is the weak solution to the Parisi PDE:
	\begin{equation}\label{eq:parisiPDEIC}
	\begin{cases}
		\partial_t \Phi + \frac{\beta_1}{4} ( \Delta \Phi + L \rho(s) (\partial_y \Phi)^2) =0 \\
		\Phi(S,y) = \frac{1}{L} \log \int e^{L(yx + \lambda x h + \mu x^2 )} d\pP_X(x) =: f_{L}(y,\lambda,\mu)
	\end{cases} \, .
	\end{equation}
	In \cite{JagSen}, the authors showed the following theorem:
	\begin{theo} \label{Thm:as-Gamma-convergence}
		Fix $t,y,h$, then the sequence $F_L$ is $\Gamma$-convergent to the functional $F$ as $L \to \infty$. In particular the following hold:
		\begin{enumerate}
			\item $(\Gamma-\liminf)$ For any sequence $(\zeta_L,\lambda_L,\mu_L) \to (\zeta,\lambda,\mu)$ we have the inequality:
			\[
			\liminf_{L \to \infty} F_L( \zeta_L,\lambda_L,\mu_L ; t,y,h ) \geq F (\zeta,\lambda,\mu ; t,y,h) \, .
			\]
			\item $(\Gamma-\limsup)$ For any $(\zeta,\lambda,\mu)$ there is a recovery sequence, i.e, there is $(\zeta_L,\lambda_L,\mu_L) \to (\zeta,\lambda,\mu) $ such that: 
			\[
			\lim_{L \to \infty} F_L(\zeta_L,\lambda_L,\mu_L ; t,y,h) = F(\zeta,\lambda,\mu ; t,y,h) \, ,
			\]
			and furthermore the recovery sequence $(\zeta_L,\lambda_L,\mu_L)$ can be taken as $(\zeta_L,\lambda,\mu)$ with $\zeta_L$ independent of the choice of $t,y,h$. 
		\end{enumerate}
	\end{theo}

\begin{rem}
    Theorem ~\ref{Thm:as-Gamma-convergence} remains true if the additional Lagrange multiplier corresponding to fixed magnetization $v$ is added. Additionally the recovery sequence in the $\Gamma$-limsup condition can still be taken to be $(\zeta_L,\lambda,\mu,\eta)$ with $\zeta_L$ independent of the choice of $h$. 
\end{rem}
    With this theorem in hand we may complete the proof of $\Gamma$-convergence to show $\E_{h \sim \mathbb{Q}} \mathcal{F}_L \to \E_{h \sim \mathbb{Q}} \mathcal{F}$.
	
	\begin{lem}\label{lem:gammaconvergence_psi}
		The functionals defined by $\E \mathcal{F}_L$ are $ \Gamma$ convergent to $\E \mathcal{F}$. 
	\end{lem}
	
	\begin{proof}
		To prove the $\Gamma$ $\liminf$ inequality note that Theorem ~\ref{Thm:as-Gamma-convergence} implies for every sequence $(\zeta_L, \lambda_L,\mu_L) \to (\zeta,\lambda,\mu)$ one has $\mathbb{Q}$	almost surely that 
		\begin{align*}
			\liminf_L \Phi_{\zeta_L,\lambda_L,\mu_L}(t,y) \geq \Phi_{\zeta,\lambda,\mu}(t,y) \, ,
		\end{align*}
		and hence by Fatou's lemma one has:
		\begin{align*}
			\E \mathcal{F}(\zeta,\lambda,\mu,t,y)=	\E_{\mathbb{Q}} \Phi_{\zeta,\lambda,\mu} &\leq
			\E_{\mathbb{Q}} \liminf_{L} \Phi_{\zeta_L,\lambda_L,\mu_L} 
			\\
            &\leq \liminf_{L} \E \Phi^L_{\zeta_L,\lambda_L,\mu_L}(t,y) 
			\\&= \liminf_L \E_{h \sim \mathbb{Q}}  \  \mathcal{F}_L(\zeta_L,\lambda_L,\mu_L,t,y,h)  \, .
		\end{align*}
		To prove the $\Gamma$ $\limsup$ inequality we note that part (b) of Theorem ~\ref{Thm:as-Gamma-convergence} implies the recovery sequence can be taken to be $(\zeta_L,\lambda,\mu)$ with $\zeta_L$ independent of the realization of $\mathbb{Q}$. Lastly, we note that initial condition of \eqref{eq:parisiPDEIC} at finite $L$ satisfies a uniform upper bound
		\[
        f_L(y,\lambda,\mu,c) \leq  Cy + D \, ,
            \]
        for some constants $C,D>0$. Consequently if we define $\hat{\Phi}^L_{\zeta_L,\lambda_L,\mu_L}(t,y)$ to be the solution to the Parisi PDE with initial condition given by: 
		\[
		\hat{f}(y) = \max_{y \in \Omega} (Cy+D) \, ,
		\]
		then $\mathbb{Q}$ almost surely we have the pointwise bound: 
		\[
		\Phi^L_{\zeta_L,\lambda,\mu} (t,y) \leq \hat{\Phi}^L_{\zeta_L,\lambda,\mu} (t,y) \, .
		\]
		{ Indeed, the Parisi PDE is monotonic in the initial condition. 
        To see this let $u,v$ denote two solutions having $u(x,1)<v(x,1)$, and set $h=v-u$. It is quick to check that $h$ satisfies the PDE
        \[
 \partial_t h + \frac{\beta_1}{4} (\Delta h + L \rho(s) (v+u) \partial_x h) = 0 \, ,
        \]
        and so by an application of the Feynman-Kac formula, we see that $h$ is non-negative for all time. }
         From here it is immediate to check the convergence 
		\[
		\hat{\Phi}^L_{\zeta_L,\lambda,\mu}(t,y) \to \hat{\Phi}_{\zeta,\lambda,\mu} \, ,
		\]
		where $\hat{\Phi}_{\zeta,\lambda,\mu}$ solves the zero temperature Parisi PDE with initial condition given by $\hat{f}(y)$. Consequently the generalized dominated convergence theorem implies that 
		\[
		\lim_{L \to \infty} \E_{\mathbb{Q}} \Phi^L_{\zeta_L,\lambda,\mu}(t,y) = \Phi_{\zeta,\lambda,\mu}(t,y) \, ,
		\]
		which completes the proof. 
	\end{proof}
	
As a consequence of the $\Gamma$-convergence in  Lemma ~\ref{lem:gammaconvergence_psi}, we have that  $\psi^{L}$ converges to $\psi$ point wise on $\mathcal{C}$. In the next section we establish the limiting formula for the maximum of the Gaussian equivalent.

\section{Limiting Formula for the Gaussian Equivalent} \label{AP:discrete-convergent}
In this section we prove the limiting variational formula for the gaussian equivalent. We begin by proving a lemma for discrete and continuous parameter spaces.

\begin{lem} \label{lem:Lipschitz-ground-state} 
    Suppose that $\Omega$ is finite, and let $0<\epsilon < \delta$. Suppose that the quantities $\abs{S-S'}$, $\abs{M-M'}$,  $\abs{v - v'} $ are all at most $\delta$, and that $\Omega_{\epsilon}(S,M,v), \Omega_{\epsilon}(S',M',v')$ are non-empty.     
    Then there is $C>0$ (depending only on $\Omega$) such that with probability $1-e^{-cN}$ one has:
    \[
\abs{ \max_{\Omega_{\epsilon}(S,M,v) } \frac{1}{N} H_N - \max_{\Omega_{\epsilon}(S',M',v') } \frac{1}{N} H_N } \leq C \abs{ 2\delta - \epsilon } \, ,
    \]
  
\end{lem}

\begin{proof}
    It suffices to reduce to shells around either value $(S,M,v)$ or $(S',M',v')$. Indeed by the conditions on $(S,M,v), (S',M',v')$ and the inequality $\epsilon < \delta$ we have the inclusions: 
    \begin{align*}
\Omega_{\epsilon} (S,M,v) \subset \Omega_{2\delta}(S',M',v') \qquad \Omega_{\epsilon} (S',M',v') \subset \Omega_{2\delta} (S,M,v) \, ,
    \end{align*}
    and hence it will suffice to bound the quantity 
    \[
\E \abs{ \max_{ \Omega_{2\delta}(S,M,v)} \frac{1}{N} H_N - \max_{\Omega_{\epsilon}(S,M,v) } \frac{1}{N} H_N } \, .
    \]
Let $\pi : \Omega_{2\delta}(S,M,v) \to \Omega_\epsilon(S,M,v)$  be the map that takes $\vect{x}$ to $\pi(\vect{x}) \in \Omega_\epsilon(S,M,v) $ such
 			that the Euclidean distance, $d(\pi(x), x)$, is minimized. As $\Omega$ is finite, this map is well-defined.
 			Furthermore, we can choose $\pi(\vect{x})$ so that $d(\pi^{\vect{x_0}}(x), x) \leq C \sqrt{N}(2\delta - \epsilon)$.  By standard bounds on operator norms of GOE matrices (see ~\cite{AGZ}  Theorem 2.3.5) there is a constant $C'>0$ such that with probability $1-e^{-cN}$, $H_N/N$ is $C'$ Lipschitz, and consequently we have:
           \begin{align*}
        \abs{ \max_{ \Omega_{2\delta}(S,M,v)} \frac{1}{N} H_N - \max_{\Omega_{\epsilon}(S,M,v) } \frac{1}{N} H_N } &\leq \E \max_{x \in \Omega_{2\delta}(S,M,v)} \frac{1}{N}\abs{H_N(x) - H_N(\pi(x) ) }     
        \\
       &\leq C'' (2\delta - \epsilon) \, , 
        \end{align*}
        proving the result. 
\end{proof}

\begin{lem}\label{lem:holdergroundstate} Suppose that $\Omega= [a,b]$, and that $(S,M,v)$ and $(S',M',v')$ belong to $\mathcal{C}$, then there is a $\frac{1}{2}$-H\"older function $f$ with $f(0)=0$, such that with probability $1-e^{-cN}$ one has: 
    \[
 \abs{ \max_{x \in \Omega(S,M,v) } \frac{1}{N} H_N(x) -  \max_{x \in \Omega(S',M',v') } \frac{1}{N} H_N(x) } \leq  f( S-S', M-M', v-v')  \, .
    \]
\end{lem}

\begin{proof}
    We write $y$ for the latent vector throughout to simplify notation. We assume without loss of generality that $y$ is not a constant vector, and that the empirical distribution of $y$ converges to a non-trivial distribution.   

We proceed as follows, given the constraints on the vector $x$, let us define vectors $e_1, e_2 \in \R^N$ as follows:
\[
e_1= \frac{y}{\norm{y} } \qquad \text{and} \qquad e_2 = \frac{ \mathbf{1} - e_1 \langle e_1, \mathbf{1} \rangle  }{\sqrt{N - N^2 \frac{ (\overline{y})^2 }{\norm{y}^2} }}
\, .
\]
(Note that in the case $y$ is a vector with constant entries, one has $e_2=0$, and the calculations that follow simplify greatly.)
Given $x \in \Omega(S,M,v)$ we may then write: 
\[
x= \alpha(x) e_1 + \beta(x) e_2 + w(x) \, ,
\]
with $w(x)$ orthogonal to $e_1$ and $e_2$ where
\begin{align*}
\alpha(x) = \frac{NM}{\norm{y}} \qquad &\qquad 
\beta(x) =  \frac{Nv - \frac{N^2 \overline{y} M}{\norm{y}^2} }{\sqrt{N- \frac{N^2 (\overline{y})^2 }{\norm{y}^2}} } \\
\norm{w(x)}^2 =  NS &- \frac{N^2 M^2}{\norm{y}^2} -  \frac{(Nv - \frac{N^2 \overline{y} M}{\norm{y}^2})^2 }{{N- \frac{N^2 (\overline{y})^2 }{\norm{y}^2}} } \, .
\end{align*}
Note that by definition of $\Omega(S,M,v)$ the functions $\alpha(x),\beta(x)$ are constant when fixing overlaps and sample means. Consequently, we may define $x' \in \Omega(S',M',v')$ as follows: 
\[
x' =  \frac{NM'}{\norm{y} } e_1 + \frac{ N v' - \frac{N^2 \overline{y}}{\norm{y}^2} M'}{\sqrt{N- N^2 \frac{\overline{y}^2}{\norm{y}^2} }} + w(x') \, ,
\]
where $w(x')$ is chosen to be collinear with $w(x)$, and such that the norm squared of $x'$ is $NS'$.  By collinearity of $w(x)$ and $w(x')$ one may then compute: 
\begin{align*}
\frac{1}{N} \norm{x-x'}^2 &=  \frac{N}{\norm{y}^2} (M-M')^2 + \frac{ ((v-v') +  \overline{y} \frac{N}{\norm{y}^2} (M'-M))^2 }{1- \frac{ N (\overline{y})^2}{\norm{y}^2}} \\
&+ \Bigg[ \sqrt{S- \frac{N M^2}{\norm{y}^2} - \frac{(v- \frac{N \overline{y}}{\norm{y}^2} M )^2}{1- \frac{N \overline{y}}{\norm{y}^2} }  } 
- \sqrt{S'- \frac{N (M')^2}{\norm{y}^2} - \frac{(v'- \frac{N \overline{y}}{\norm{y}^2} M' )^2}{1- \frac{N \overline{y}}{\norm{y}^2} }  } \Bigg]^2 \, .
\end{align*} 
By our assumptions on the empirical distribution of $y$, the terms $ \frac{N}{\norm{y}^2}$  converge to a non-zero constant as $N \to \infty$, as does $\overline{y}$. By the non-triviality assumption of the limit, and the Cauchy-Schwarz inequality, one also has some universal constant $c_1 >0$ so that for all $N$ sufficiently large 
\[
 1- N \frac{ \overline{y}^2}{\norm{y}^2} > c_1 \, .
\]
With this in mind let us then note that for some $\frac{1}{2}$-H\"older  function $f$, we have a bound for all large $N$ given by: 
\[
\norm{x-x'} \leq \sqrt{N} f( S-S',M-M',v-v') \, ,
\]
which satisfies $f(0,0,0)=0$.  

To finish the proof, note for $C >0$ sufficiently large we have $c>0$ so that with probability $1-e^{-cN}$ one has:
\[
\abs{ H_N(x_1) - H_N(x_2) } \leq C \sqrt{N} \norm{x_1-x_2} \, . 
\]
Indeed this follows by standard bounds on the operator norm of a GOE matrix exceeding $2+\epsilon$ (see \cite{AGZ} Theorem 2.3.5), and by noting the non-random terms in $H_N$ are $C \sqrt{N}$ Lipschitz, for some $C>0$ depending on $\Omega$ and $\overline{\beta}$. 

 Now given $x \in \Omega(S,M,v)$, one may always pair it with the constructed $x'$ above to obtain
\[
 \max_{ x \in \Omega(S,M,v)}  \frac{1}{N} H_N(x) \leq  C f( S-S',M-M',v-v') + \max_{ x \in \Omega(S',M',v')}  \frac{1}{N} H_N(x) \, ,
\]
with the reverse inequality following via symmetric argument. 
Absorbing $C$ into $f$,  we conclude that with probability $1-e^{-cN}$ that:
\[
\abs{ \max_{x \in \Omega(S,M,v) } \frac{1}{N} H_N(x) -  \max_{x \in \Omega(S',M',v') } \frac{1}{N} H_N(x) } \leq  f( S-S', M-M', v-v')  \, .
\]
Completing the proof. 
\end{proof}

With Lemma ~\ref{lem:Lipschitz-ground-state} and Lemma ~\ref{lem:holdergroundstate} in hand, we may now prove the variational characterization of the ground state for the gaussian equivalent. 
\begin{lem}\label{lem:GSlimit}
    Suppose that $\Omega$ is an interval or finite set, then for any $\bar \beta = (\beta_1, \beta_2, \beta_3, \beta_4)$,
    \[
    \lim_{N \to \infty} \E \max_{x \in \Omega^N} \frac{H_N^{\bar \beta}}{N} = \sup_{(s,m,v) \in \cC} \psi_{\bar \beta}(s,m,v)
    \]
\end{lem}
\begin{proof}

For a lower bound note for any $\epsilon >0$ and any $(S,M,v)$ one has: 
\[ 
\E \max_{x \in \Omega^{N} } \frac{1}{N} H_N(x) \geq \E \max_{x \in \Omega_{\epsilon} (S,M,v) } \frac{1}{N} H_N(x) \, ,
\]
and taking $N \to \infty$ and then $\epsilon \to 0$ we obtain:
\[
\lim_{N \to \infty} \E \max_{x \in \Omega^{N} } \frac{1}{N} H_N(x)  \geq \psi_{\bar \beta} (S,M,v) \, ,
\]
for any $(S,M,v) \in \cC$. 

In the case $\Omega$ is an interval, to prove an upper bound, let us define $\psi_{N,\bar \beta}$ by 
        \[
        \psi_{N,\bar \beta}(s,m,v) := \max_{\vect{x} \in \Omega(s,m,v)} \frac{H_N^{\bar \beta}(\vect{x})}{N} \, .
        \]
        By Lemma ~\ref{lem:gammaconvergence_psi}, we have that $\psi_{N,\bar \beta} \to \psi_{\bar \beta}$ pointwise as $N$ tends to infinity. To conclude the result it will suffice to show that $\psi_{N,\bar \beta}$ converges uniformly to $\psi_{\bar \beta}$ on $\mathcal{C}$, as uniform convergence implies convergence of the maximum. 
        
        Fix $\epsilon > 0$, then if $f$ is as in Lemma ~\ref{lem:holdergroundstate}, we may pick $\delta>0$ so that
        \[
         f( S-S', M-M', v-v') \leq \epsilon 
        \]
        whenever $d( (s,m,v), (s',m',v')) \leq \delta$.

Fix a $\delta$ net of $\Omega_N$, in terms of the parameters $(S,M,v)$, i.e a collection of points $\Delta \subset \mathcal{C}$ such that $\Omega_N \subset \cup_{(S,M,v) \in \Delta} \Omega_{\delta} (S,M,v) $ for $N$ sufficiently large.  Such a set can be chosen to be finite by the definition of $\cC$ and compactness of $\Omega$. Then by Lemma ~\ref{lem:holdergroundstate} with probability $1-e^{-cn}$ one has:
\[
 \max_{x \in \Omega_N} \frac{1}{N} H_N \leq \max_{(S,M,v) \in \Delta} \psi_N(S,M,v) + \epsilon + e^{-cN} \ ,
\]
and hence by Lemma ~\ref{lem:conccentrationrestrictedMLE}  we conclude:
\[
\lim_{N \to \infty} \E \max_{x \in \Omega^N} \frac{1}{N} H_N \leq \sup_{(S,M,v) \in \cC} \psi_{\bar \beta} (S,M,v) +\epsilon \, ,
\]
this holds for any $\epsilon>0$, and concludes the proof for intervals. 

In the case that $\Omega$ is discrete we proceed in the same way, fix $0<\epsilon < \delta$, and $\Delta$ be a $\delta$ net of $\Omega_N$. By Lemma ~\ref{lem:Lipschitz-ground-state}, one then has: 
\begin{align*}
\max_{x \in \Omega_N} \frac{1}{N} H_N &\leq \max_{ (S,M,v) \in \Delta} \left[ \max_{\Omega_{\delta}(S,M,v)}  \frac{1}{N} H_N \right] + \delta \\
&\leq  \max_{ (S,M,v) \in \Delta} \left[ \max_{\Omega_{\epsilon}(S,M,v)}  \frac{1}{N} H_N \right] + C(2\delta - \epsilon) + \delta \, ,
\end{align*}
and so by concentration of the maximum, taking $N \to \infty$ and then $\epsilon \to 0$ one has: 
\[
\lim_{N \to \infty} \max_{x \in \Omega_N} \frac{1}{N} H_N  \leq \max_{(S,M,v) \in \Delta} \psi_{\bar \beta} (S,M,v) + 2C \delta + \delta \leq \sup_{(S,M,v \in \cC} \psi_{\bar \beta} + C' \delta \, ,
\]
and as this holds for every $\delta>0$, the result follows.
\end{proof}

\section{Proofs of Variational Formulas}
    
	In this section, we will prove the variational formulas for the zero score, score biased, and score corrected models. These correspond to the first parts of Theorems \ref{thm:main_regular}, \ref{th:negscore}, and \ref{th:corrected}. By universality in Proposition~\ref{prop:groundstateuniversality1}, all of these variational formulas are direct consequences of the master theorem for the gaussian equivalent which is summarized below.
    \begin{theo}\label{theo:MLEmaster}
    Suppose that $\Omega$ is an interval or finite collection of points. For any $\beta_1, \beta_2, \beta_3, \beta_4$ and $\alpha$ and constraints $(S,M,v)$, we have 
    \begin{equation}\label{eq:master1}
        \lim_{\epsilon \to 0} \lim_{N \to \infty} \frac{1}{N} \E \max_{x \in \Omega_N} H_N^{\bar \beta,\beta_4}(x)  = \sup_{S,M,v}\psi_{\mathbf{\beta}}(S,M,v) 
    \end{equation}
        and 
        \begin{equation}\label{eq:master2}
            \lim_{\epsilon \to 0} \lim_{N \to \infty} \frac{1}{N} \E \max_{x \in \Omega_{\epsilon}(S,M,v)} H_N^{\bar \beta,\beta_4}(x)  = \psi_{\mathbf{\beta}}(S,M,v) .
        \end{equation}
		If $\alpha = 0$, then for any constraints $(S,M) \in \cC$
        \begin{equation}\label{eq:master3}
           \lim_{\epsilon \to 0} \lim_{N \to \infty} \frac{1}{N} \E \max_{x \in \Omega_N} H_N^{\bar \beta,0}(x)  = \sup_{S,M}\psi_{\mathbf{\beta}}(S,M)
        \end{equation}
        and 
        \begin{equation}\label{eq:master4}
          \lim_{\epsilon \to 0} \lim_{N \to \infty} \frac{1}{N} \E \max_{x \in \Omega_{\epsilon}(S,M)} H_N^{\bar \beta,0}(x)  = \psi_{\mathbf{\beta}}(S,M) .  
        \end{equation}
    \end{theo}
    \begin{proof}
        We provide the proof for the cases when $\alpha \neq 0$, because the case when $\alpha = 0$ follows from an identical argument.  The limit for the unconstrained maxima is given in Lemma ~\ref{lem:GSlimit}. The limit for the constrained model is given in Lemma~\ref{lem:characterizationconstrained}.
    \end{proof}
    
	\subsection{Proof of the Variational Formula for the Score Biased Models}
	
	By universality, we start by showing that the maximum liklihood estimate associated with the score corrected likelihood
	\[
	\mathcal{L}^{g}_{N,\alpha}(Y,x) =   \sum_{i \leq j } g \Big(Y_{ij}, \frac{\lambda x_i x_j}{\sqrt{N}} \Big).
	\]
	is equivalent to $H_N^{\bar \beta,\alpha}(x)$ where $\beta_1, \beta_2, \beta_3$ are the information parameters and $\alpha = N^{1/2} \beta_4$ defined in  \eqref{eq:modelequiv2}.
	\begin{lem}
		For $g,g_0 \in \cF_0$, we have
		\[
		\lim_{N \to \infty} \bigg| \frac{1}{N} \E \max_{x \in \Omega_N} \mathcal{L}^{g}_{N,\alpha}(Y,x) - \frac{1}{N} \E \max_{x \in \Omega_N} H_N^{\bar \beta,N^{\frac{1}{2}}\beta_4}(x) \bigg| = 0.
		\]
	\end{lem}
	\begin{proof}
		This is simply a restatement of Lemma~\ref{lem:lemUniv}.        
	\end{proof}
	
	In particular, it suffices to study the maximizers of the function
	\begin{align}\label{eq:non-negscore}
		H_N^{\bar \beta,N^{\frac{1}{2}}\beta_4}(x) &=  \frac{\beta_1}{\sqrt{N}} \sum_{i \leq j} g_{ij} x_i x_j + \frac{N \beta_2}{2} M_N(x)^2 - \frac{N \beta_3}{4}S_N(x)^2 + \frac{N^{3/2} \beta_4}{2} \bar{x}^2 + o_N(1).
	\end{align}
	Notice that the last term $ \frac{N^{3/2} \beta_4}{2} \bar{x}^2$ is the leading order term. This leading order term does not depend on the unknown variable, but dictates the performance of the MLE. If $\beta_4 > 0$, then the estimator must maximize this term, which is the statement of Theorem~\ref{th:posscore}.
	
	\begin{proof}[Proof of Part 1 of Theorem~\ref{th:posscore}]
		If $\beta_4 > 0$, notice that $	H_N^{\bar \beta,N^{\frac{1}{2}}\beta_4}(x)$ is maximized when $(\bar{x})^2$ is maximized. In particular, we have that $x = x_+ \vect{1}$, where $\vect{1}$ is the all $1$'s vector and $x_+$ was the largest point in $\Omega$.  
	\end{proof}
	
	If $\beta_4 < 0$, then the estimator must minimize the leading order term in \eqref{eq:non-negscore}, which is the conclusion of Theorem~\ref{th:negscore}. In the following, we write $a\approx_\eps b$ if $\abs{a-b}\leq \epsilon$.
	
	\begin{proof}[Proof of Part 1 of Theorem~\ref{th:negscore}]
		If $\beta_4 < 0$, notice that $	H_N^{\bar \beta,N^{\frac{1}{2}}\beta_4}(x)$ is maximized when $(\bar{x})^2$ is minimized. In particular, we have that $x = x_-$, where $x_-$ is the smallest point in the convex hull of $\Omega$. 
		Note however that $x_{-}$ may not lie within $\Omega^N$, but up to introducing a term of order $\frac{C}{N}$ for some $C>0$, we may assume it does. Taking $\epsilon_N = \frac{C}{N}$ for a large constant $C$, we then have
		with probability at least $1-e^{-cN}$ that:
		\[
			| \max_{x \in \Omega_N} H^{\bar \beta,N^{\frac{1}{2}}\beta_4}_N(x) -  \max_{x \in \Omega : \bar x \approx_{\epsilon_N} x_- } H^{\bar \beta,N^{\frac{1}{2}}\beta_4}_N(x) | \leq \frac{C'}{\sqrt{N}} \, ,
		\]
 where the bound above comes from tail bounds on the operator norm of a GOE matrix, see ~\cite{AGZ}. 
 
		We now maximize the constrained maximization problems
		\[
		\max_{x \in \Omega : \bar x \approx_{\epsilon} x_- } H^{\bar \beta,N^{\frac{1}{2}}\beta_4}_N(x) = \sup_{(S,M) \in \cC} \max_{ \Omega_{\epsilon}(S,M, x_{-} )}   H^{\bar \beta,N^{\frac{1}{2}}\beta_4}_N(x) \, .
		\]
		Subtracting the leading order term and taking limits implies that
		\begin{align*}
			&\lim_{\epsilon \to 0} \lim_{N \to \infty}  \max_{x \in \Omega_N} \frac{1}{N}  \Big( H^{\bar \beta,N^{\frac{1}{2}}\beta_4}_N(x)  -  \frac{N^{3/2} \beta_4}{2} \bar{x}^2 \Big)
			\\&=\lim_{\epsilon \to 0} \lim_{N \to \infty} \sup_{(S,M) \in \cC} \max_{ x \in \Omega_{\epsilon}(S,M) : \bar x = x_- } \frac{1}{N} \big(   H^{\bar \beta,N^{\frac{1}{2}}\beta_4}_N(x)  -  \frac{N^{3/2} \beta_4}{2} \bar{x}^2 \Big)
			\\&=\lim_{\epsilon \to 0} \lim_{N \to \infty} \sup_{(S,M) \in \cC} \max_{ x \in \Omega_{\epsilon}(S,M) : \bar x = x_- } \frac{1}{N}  H^{\bar \beta,0}_N(x)  
			\\&= \sup_{(S,M) \in \cC} \psi_-(S,M,x_-) \, .
		\end{align*}
        where the last equality follows from \eqref{eq:master2} and concentration of the ground state Lemma~\ref{lem:conccentrationrestrictedMLE}. 
	\end{proof}

	\subsection{Proof of the Variational Formula for the Score Corrected Model}
	
	We prove Theorem~\ref{th:corrected}. We start by showing that the maximum likelihood estimate associated with the score corrected likelihood
	\[
	\mathcal{L}^{g}_{N,\alpha}(Y,x) =   \sum_{i \leq j } g \Big(Y_{ij}, \frac{\lambda x_i x_j}{\sqrt{N}} \Big) - N^{\frac{3}{2}} \hat \beta_4 \bar x^2 + N \alpha \bar x^2.
	\]
	is equivalent to $H_N^{\bar \beta,\alpha}(x)$ where $\beta_1, \beta_2, \beta_3$ are the information parameters and we let  $\alpha$  be equal to $\beta_2 [\E_{\pQ} x_0 ]^2  + \gamma $ defined in  \eqref{eq:modelequiv2}.
	\begin{lem}\label{lem:approxcorrected}
		For $g,g_0 \in \cF_0$, we have
		\[
		\lim_{N \to \infty} \bigg| \frac{1}{N} \E \max_{x \in \Omega_N} \mathcal{L}^{g}_{N,\alpha}(Y,x) - \frac{1}{N} \E \max_{x \in \Omega_N} H_N^{\bar \beta,\alpha}(x) \bigg| = 0 \, .
		\]
	\end{lem}
	\begin{proof} 
		We define 
		\[
		\hat \beta_4  = \frac{2}{N^2} \sum_{i \leq j} \partial_w g(Y_{ij},0) \, ,
		\]
		which approximates
		\[
		\E \hat \beta_4 = \frac{2}{N^2} \sum_{i \leq j} \E \partial_w g(Y_{ij},0) = \E [ \partial_w g(Y,0)] = \beta_4 +  \frac{\beta_2}{\sqrt{N}}  \bigg( \frac{1}{N} \sum_{i} x_i^0\bigg)^2  + O(N^{-1}) \, .
		\]
		This is not an immediate consequence of universality (Proposition~\ref{prop:groundstateuniversality1}) because $\hat \beta_4$ depends on all entries of $Y$. We will show that we can replace $\hat \beta_4$
		with its expected value.
		We define the likelihood
		\[
		\mathcal{\bar L}^{g}_{N,\alpha}(Y,x) =   \sum_{i \leq j } g \Big(Y_{ij}, \frac{\lambda x_i x_j}{\sqrt{N}} \Big) - N^{\frac{3}{2}}  \E [\hat \beta_4]  \bar x^2 + N \alpha \bar x^2 \, ,
		\]
		which replaces $\hat \beta_4$ in $\mathcal{\bar L}^{g}_{N,\alpha}(Y,x) $ with its expected value. We will prove that
		\[
		\bigg| \frac{1}{N} \E  \max_{x \in \Omega_N} 	\mathcal{ L}^{g}_{N,\alpha}(Y,x;) - \frac{1}{N} \E  \max_{x \in \Omega_N} 	\mathcal{\bar L}^{g}_{N,\alpha}(Y,x) \bigg| \leq O(N^{-\frac{1}{2}}) \, .
		\]
		
		By Jensen's inequality, it suffices to show that 
		\[
		\E \bigg| \frac{1}{N}   \max_{x \in \Omega_N} 	\mathcal{ L}^{g}_{N,\alpha}(Y,x) - \frac{1}{N}   \max_{x \in \Omega_N} 	\mathcal{\bar L}^{g}_{N,\alpha}(Y,x)  \bigg|\leq O(N^{-\frac{1}{2}})\, .
		\]
        Since $\abs{ \max f(x) - \max g(x)} \leq \max \abs{f(x)-g(x)} $ we have
		
		\[
		\bigg| \frac{1}{N} \E  \max_{x \in \Omega_N} 	\mathcal{ L}^{g}_{N,\alpha}(Y,x) - \frac{1}{N} \E  \max_{x \in \Omega_N} 	\mathcal{\bar L}^{g}_{N,\alpha}(Y,x) \bigg| \leq \sqrt{N} \E   \max_{x \in \Omega_N} 	 |\hat \beta_4 - \E \hat \beta_4 | (\bar x)^2  \, ,
		\]
		and as $\Omega$ is bounded by $C$, this further implies
		\begin{align} \label{eq:max-mle-bound}
		\E \bigg| \frac{1}{N}   \max_{x \in \Omega_N} 	\mathcal{ L}^{g}_{N,\alpha}(Y,x) - \frac{1}{N}   \max_{x \in \Omega_N} 	\mathcal{\bar L}^{g}_{N,\alpha}(Y,x)  \bigg| &\leq C \sqrt{N} \E |\hat \beta_4 - \E \hat \beta_4 | 
        \\
        &\leq C \sqrt{N} ( \E ( \hat \beta_4 - \E \hat \beta_4 )^2 )^{1/2} \, .
		\end{align}
		Note that
		\[
		\Var( \hat \beta_4  ) = \frac{4}{N^4}\sum_{i \leq j} \E [ \partial_w g( Y_{ij}, 0 ) - \E \partial_w g( Y_{ij}, 0 ) ]^2  = O(N^{-2}) \, ,
		\]
		as $ \Var(   \partial_w g( Y_{ij}, 0 )  )$ is uniformly bounded for all $i,j$ independent of $N$ since $g \in \mathcal{F}_0$. Consequently the right hand side of \eqref{eq:max-mle-bound} is $O(N^{-1/2} )$.
		
		By Proposition~\ref{prop:groundstateuniversality1} applied to $	\mathcal{\bar L}^{g}_{N,\alpha}(Y,x) $ the result then follows. \qedhere 
	\end{proof}
	
	Theorem~\ref{th:corrected} now follows by applying the variational formulas in Theorem~\ref{theo:MLEmaster} to $H_N^{\bar \beta,\alpha}(x) $. 
	\begin{proof}[Proof of Theorem~\ref{th:corrected}]
		This follows immediately by combining Lemma~\ref{lem:approxcorrected} and \eqref{eq:master1}. 
	\end{proof}

    \subsection{Proof of the Variational Formula for the Zero Score Model}
    For completeness, we also provide the proof for zero score models, which follows from a simple modification of the previous arguments for score biased score models. 
    
    \begin{proof}[Proof of part 1 of Theorem~\ref{thm:main_regular}]
		By Lemma~\ref{prop:groundstateuniversality1}, it suffices to compute the limit of
		$
		\frac{1}{N} \E \max_{x \in \Omega_N} H_N^{\bar \beta,\alpha}(x) 
		$.
		The maximum of such functions are precisely the one computed in \eqref{eq:master3}, so 
		\[
		 \E \frac{1}{N}   \max_{x \in \Omega_N} 	\mathcal{\bar L}^{g}_{N,\alpha}(Y,x)   \to \sup_{(S,M,v) } \psi_{\bar{\beta},\alpha}(S,M,v) \, . \qedhere
		\]
	\end{proof}

	\section{Characterizations of the Overlaps for MLE}\label{sec:charac}

	In this section, we will prove the second part of Theorem~\ref{thm:main_regular},   Theorem~\ref{th:negscore}, and Theorem~\ref{th:corrected}. For fixed $S$ and $M$, to simplify notation we define
	\[
	\psi_{\bar \beta}(S,M) = \begin{cases}
		\sup_{v} \psi_{\bar \beta, \alpha}(S,M,v) & \text{ if } \alpha \neq 0\\
		\psi_{\bar \beta,0}(S,M) & \text{ if } \alpha = 0.
	\end{cases} 
	\]
	
	We show that the limiting overlaps of the ground state variational formula are given by the corresponding maximizers of the ground state free energy provided the maximizers are unique as in Hypothesis \ref{hyp:existintro}. 
    In the case the maximizers are not unique, the limits of the overlaps converge to one of the maximizers of $\psi$, and will be dealt with separately at the end of this section.

	Hypothesis \ref{hyp:existintro} implies the MLE is well defined in the following sense. Suppose that there exists a unique $(s_*,m_*)$ such that all maximizers of $H_N^{\bar \beta}(\vect{x})$ are attained on the set $\Omega_\epsilon(s_* , m_*)$
	\[
	\max_{\vect{x}} H_N^{\bar \beta}(\vect{x}) = 	\max_{s,m} \max_{\Omega_\epsilon(s,m)} H_N^{\bar \beta}(\vect{x}) = \max_{\Omega_\epsilon(s^*,m^*)} H_N^{\bar \beta}(\vect{x}).
	\]
	If this holds, then
	\[
	\vect{ \hat x}_{\mis} = \arg \max_{\vect{x}} H_N^\beta(\vect{x})
	\]
	satisfies $M_N(\vect{\hat x}_{\MLE}) \approx_\epsilon m_*$, $S(\vect{\hat x}_{\MLE}) \approx_\epsilon s_*$. In particular, any maximizer of $H_N^\beta(\vect{x})$ maximizes the overlap with the underlying signal. The rest of this section will be devoted to show that the maximizing $m_*$ is given by the largest maximizer of $\psi_{\bar \beta}$, for models such that the information parameters $\bar \beta$ satisfy \eqref{hyp:existintro}.
	
	We begin by showing a concentration result that implies that we can consider the average overlaps. 
    \begin{lem}\label{lem:conccentrationrestrictedMLE}
    Let $\overline{\beta}$ be fixed. There exists a universal constant $C$ that depends only on $\bar \beta$ such that
    \[
    \pP\Big( \Big| \frac{1}{N} \max_{x \in \Omega_\epsilon(S,M)} H^{\bar \beta}_N(x) - \E \frac{1}{N} \max_{x \in \Omega_\epsilon(S,M)} H^{\bar \beta}_N(x) \Big| \geq t \Big) \leq e^{- C t^2 N} \, ,
    \]
    for any $(s,m) \in \cC$. Furthermore,
        \[
		 \E \bigg| \frac{1}{N} \max_{x \in \Omega_\epsilon(s,m)} H_N^{\bar \beta}(x) -  \frac{1}{N} \E  \max_{x \in \Omega_\epsilon(s,m)} H_N^{\bar \beta}(x) \bigg| \leq \frac{C}{\sqrt{N}}  \, .
		\]
\end{lem}
\begin{proof}
    This follows immediately from standard concentration inequalities for Gaussian processes such as the one in \cite[Section~2.1]{adler_taylor_book}. 
\end{proof}
    
	We now show that the limit of the constrained maximization problem for the gaussian equivalent is given by $\psi_{\bar \beta}(s,m)$. 
	
	\begin{lem}\label{lem:characterizationconstrained}
		For any $(s,m) \in \cC$ 
		\[
		\lim_{\epsilon \to 0} \lim_{N \to \infty} \frac{1}{N} \E \max_{\Omega_\epsilon(s,m)} H_N^{\bar \beta}(\vect{x}) =  \psi_{\bar \beta}(s,m).
		\]
	\end{lem}
	
	\begin{proof}
		Let $\epsilon > 0$. We consider the constrained free energy
		\[
		\frac{1}{L} F_N^{L\bar \beta}(s,m) = \frac{1}{NL} \E \log \int_{\Omega_{\epsilon}(S,M)} e^{L H^{\bar \beta}_N(\vect{x})} \, d \pP^{\otimes N}_X(\vect{x}).
		\]
		Without loss of generality, we may assume that $\pP_X$ is uniform over $\Omega_N$. Using the bounds of the ground state with the free energy \eqref{eq:boundsgroundstate}, we have 
		\[
		\frac{1}{N} \E \max_{\Omega_\epsilon(s,m)} H_N^{\bar \beta}(\vect{x}) + o_N(L) \leq	\frac{1}{L} F_N^{L\bar \beta}(s,m)\leq  \frac{1}{N} \E \max_{\Omega_\epsilon(s,m)} H_N^{\bar \beta}(\vect{x}) \, .
		\]
		Therefore, it suffices to compute a limit of $\frac{1}{L} F_N^{L\bar \beta}(s,m) $ for large $N$ and fixed $L$. 
		
		For every $L$, we have as $N \to \infty$ by the finite temperature case that
		\[
		\lim_{\epsilon \to 0} \lim_{N \to \infty}	\frac{1}{L} F_N^{L\bar \beta}(s,m) = \varphi_{L \bar \beta}(s,m) \, . 
		\]
		By applying Lemma ~\ref{lem:gammaconvergence_psi}, we conclude that
		\[
		\lim_{L \to \infty} \lim_{\epsilon \to 0} \lim_{N \to \infty}	\frac{1}{L} F_N^{L\bar \beta}(s,m) = \lim_{L \to \infty} 	\varphi_{L \bar \beta}(s,m) = \psi_{\bar \beta}(s,m).
		\]
		Therefore, using the ground state bounds \eqref{eq:boundsgroundstate} we have 
		\[
		\lim_{\epsilon \to 0} \lim_{N \to \infty} \frac{1}{N} \E \max_{\Omega_\epsilon(s,m) } H_N^{\bar \beta}(\vect{x}) =  \psi_{\bar \beta}(s,m) \, ,
		\]
		which is what we needed to show. 
	\end{proof}
	
	It now remains to show that the gaussian equivalent characterizes the maximum likelihood estimator of the original inference problem. To this end, given a model $g$ and the corresponding Fisher score parameters $\bar \beta$, we define
	\[
	\vect{ \hat x}^g_{\PMLE} = \arg \max_{x \in \Omega^N}  \sum_{i \leq j} g \Big(Y_{ij}, \frac{ x_i x_j}{\sqrt{N}} \Big)  \, ,
	\]
        as was defined in \eqref{eq:def-of-mismatch-mle}. The following Lemma is a universality statement for the overlaps of the ground state.
	\begin{lem}\label{lem:characterization}
		If the Fisher score parameters $\bar \beta$ corresponding to $g \in \cF_0$  satisfies Hypothesis~\ref{hyp:existintro},then for any choice of maximizer $\hbx^g_{\PMLE}$ one has almost surely
		\begin{align*}
		\lim_{N \to \infty} M_N(\vect{\hat x}^g_{\PMLE})^2 &= 	\lim_{N \to \infty}  M_N(\vect{\hat x}^{\bar \beta}_{\PMLE})^2 = (m_*)^2 \, , \\
	\lim_{N \to \infty} S_N(\vect{\hat x}^g_{\PMLE}) &=  	\lim_{N \to \infty} S_N(\vect{\hat x}^{\bar \beta}_{\PMLE}) = s_* \, ,
		\end{align*}
		where $(s_*,(m_*)^2)$ is the maximizing pair of $\psi_{\bar \beta}$ given in Hypothesis~\ref{hyp:existintro}.
	\end{lem}
	\begin{proof}
		By Proposition~\ref{prop:groundstateuniversality1} we know that uniformly for $s$ and $m$,
		\[
		\bigg| \frac{1}{L} F_N^{L \bar \beta} (s,m) - 	\frac{1}{L} F_N^{L g} (s,m) \bigg| \leq o_N(L).
		\]
		Furthermore, we have using the finite temperature formulas in Theorem~\ref{th:finitetemplimit} that
		\begin{equation}\label{eq:univcons1}
			\lim_{N \to \infty} \frac{1}{L} F_N^{L g} (s,m) =\lim_{N \to \infty}  \frac{1}{L} F_N^{L \bar \beta} (s,m) = \lim_{N \to \infty}  \frac{1}{L} \varphi_{\bar \beta}(s,m) \, ,
		\end{equation}
		and so by \eqref{eq:master2} we obtain 
		\begin{align}
		\lim_{L \to \infty}\lim_{\epsilon \to 0} \frac{1}{L} F_N^{L g} (s,m)  &= 	\lim_{L \to \infty}\lim_{\epsilon \to 0}   \lim_{N \to \infty} \frac{1}{L} F_N^{L \bar \beta} (s,m) \nonumber
        \\
        &= \lim_{\epsilon \to 0} \lim_{N \to \infty} \frac{1}{N} \E \max_{\Omega_\epsilon(s,m)} H_N^{\bar \beta}(\vect{x})  = \psi_{\bar \beta}(s,m) \, .\label{eq:univcons3}
		\end{align}
		An identical argument using the unconstrained limit of Theorem~\ref{th:finitetemplimit} and \eqref{eq:master1} implies that
		\[
		\lim_{L \to \infty}\lim_{\epsilon \to 0}   \lim_{N \to \infty} \frac{1}{L} F_N^{L g}  = \lim_{\epsilon \to 0} \lim_{N \to \infty} \frac{1}{N} \E \max_{\vect{x}} H_N^{\bar \beta}(\vect{x})  = \sup_{s,m} \psi_{\bar \beta}(s,m).
		\]

		The following is where Assumption~\ref{hyp:existintro} plays its most crucial role. Since the square of the maximizers of $\psi_{\bar \beta}(s,m)$ are unique, we have for all $m^2 \neq m_*^2$,
		\[
		\psi_{\bar \beta}(s,m) < \psi_{\bar \beta}(s^*,m^*).
		\]
		Furthermore, since $\psi_{\bar \beta}(s,m)$ only depends on $m$ through $m^2$, we have for the case that $m^2 = m_*^2$ that
		\[
		\psi_{\bar \beta}(s,m) = \psi_{\bar \beta}(s^*,m^*).
		\]
		Using the characterization in \eqref{eq:master2}, this implies that for all $s,m$ such that  $m^2 \neq m_*^2$
		\[
		\lim_{\epsilon \to 0}\lim_{N \to \infty}  \frac{1}{N} \E \max_{\Omega_\epsilon(s,m)} H_N^{\bar \beta}(\vect{x}) < 	\lim_{\epsilon \to 0} \lim_{N \to \infty}  \frac{1}{N}  \E \max_{\Omega_\epsilon(s_*,m_*) } H_N^{\bar \beta}(\vect{x}) \, ,
		\]
		where equality is attained when $m^2 = m_*^2$. 
		By partitioning our state space, we have that
		\[
		\frac{1}{N} \E \max_{\vect{x}} H_N^{\bar \beta}(\vect{x}) = \frac{1}{N} \E \max_{s,m} \max_{\Omega_\epsilon(s,m)} H_N^{\bar \beta}(\vect{x}) \, ,
		\]
		and concentration in Lemma~\ref{lem:conccentrationrestrictedMLE} implies that for every $\epsilon > 0$, 
		\[
		\lim_{N \to \infty} \frac{1}{N} \E \max_{s,m} \max_{\Omega_\epsilon(s,m)} H_N^{\bar \beta}(\vect{x}) = \lim_{N \to \infty} \frac{1}{N}\max_{s,m}  \E  \max_{\Omega_\epsilon(s,m)} H_N^{\bar \beta}(\vect{x}) \, .
		\]
		Taking limits and applying \eqref{eq:master2} and \eqref{eq:univcons1}, \eqref{eq:univcons3} implies
		\begin{align*}
		\lim_{\epsilon \to 0}\lim_{N \to \infty} \frac{1}{N} \E \max_{\vect{x}} H_N^{g}(\vect{x})  &= \lim_{\epsilon \to 0}\lim_{N \to \infty} \frac{1}{N} \E \max_{\vect{x}} H_N^{\bar \beta}(\vect{x})
        \\
        &=	\lim_{\epsilon \to 0}\lim_{N \to \infty}  \frac{1}{N} \E  \max_{\Omega_\epsilon(s_*,m_*)} H_N^{\bar \beta}(\vect{x})
        \\
        &= \psi_{\bar \beta}(s_*,m_*) \, .
		\end{align*}
		This implies that the maximum is attained on  $ \Omega_\epsilon(s,m)$, so 
		$
		M(\vect{\hat x}^g_{\PMLE})^2 = M(\vect{\hat x}^{\bar \beta}_{\PMLE})^2 = (m_*)^2,
		$
		where $m_*$ is the largest maximizer of $\psi_{\bar \beta}(s,m)$ and
		$
		S(\vect{\hat x}^g_{\PMLE}) = S(\vect{\hat x}^{\bar \beta}_{\PMLE}) = s_* \, .\qedhere
		$
	\end{proof}
	
	Combining all of the above implies the following lemma which characterizes the cosine similarity and mean squared error in the high dimensional limit.
	
	\begin{lem}
		If our model and associated Fisher parameters  $\bar \beta_g$ satisfy Hypothesis~\ref{hyp:existintro}, we have
		\[
		|\cs(\vect{\hat x}_{\MLE}, \vect{x}_0)|   \to \frac{|m_{\beta_1, \beta_2, \beta_3}|}{ \sqrt{s_{\beta_1, \beta_2, \beta_3} } \sqrt{\E_{\mathbb{Q} } (x^0)^2 } } \qquad \text{a.s.} 
		\]
		where $m_{\beta_1, \beta_2, \beta_3}$ is the largest maximizer of $\sup_{(s,m) \in \cC} \psi_{\bar\beta}(s,m)$. 
	\end{lem}
	\begin{proof}
		This follows immediately from the characterization of the normalized inner products in Lemma~\ref{lem:characterization} and the fact that the mean squared error and cosine similarity are determined by the normalized inner products. 
		
		Indeed, notice that
		\[
		\cs(\vect{\hat x}_{\MLE}, \vect{x}^0) = \frac{ R( \vect{\hat x}_{\MLE},\vect{x}^{0} ) }{\sqrt{ R(\vect{\hat x}_{\MLE},\vect{\hat x}^g_{\MLE})  R(\vect{x}^0 , \vect{x}^0 ) } } \, ,
		\]
		and furthermore, Lemma~\ref{lem:characterization}  implies that
		\[
		|R( \vect{\hat x}_{\MLE},\vect{x}_0 )| \to |m_*| \text{ and } R(\vect{\hat x}_{\MLE},\vect{\hat x}^g_{\MLE}) \to s_* \, ,
		\]
		which finishes the proof. 
	\end{proof}
	
	We close this section by arguing that the technical assumption Hypothesis~\ref{hyp:existintro} is equivalent to a regularity condition on $\psi_{\bar \beta}$ with respect to its information parameters. In particular, we will show that the overlaps for the gaussian equivalent problem are uniquely  characterized by the maximizers of $\psi_{\bar \beta}$ on the sets where $\psi_{\bar \beta}$ is differentiable. Consider the function
    \begin{equation}\label{eq:limit_temperature}
    f(\beta_1, \beta_2, \beta_3) = \sup_{s,m} \psi_{\bar \beta}(S,M) \, ,  
    \end{equation}
	and define the sets
	\[
	\mathcal{D}_{\beta_2} = \{ (\beta_1, \beta_2, \beta_3) \mmm \partial_{\beta_2} f \text{ exists} \} \text{ and } \mathcal{D}_{\beta_3} = \{ (\beta_1, \beta_2, \beta_3) \mmm \partial_{\beta_3} f \text{ exists} \} \, .
	\]
	We show that the characterization of the overlap for the gaussian equivalent is valid at points where $	f(\beta_1, \beta_2, \beta_3) $ is differentiable. 
	
	\begin{lem} \label{lem:condition_unique}
		We have for all $\beta \in \mathcal{D}_{\beta_2}$ that
		\[
		R(\vect{\hat x}^{\bar \beta}_{\MLE}, \vect{x}_0)^2   \to m^2_{\beta_1, \beta_2, \beta_3} \, ,
		\]
		and for all  $\beta \in \mathcal{D}_{\beta_3}$ that
		\[
		R(\vect{\hat x}^{\bar \beta}_{\MLE}, \vect{\hat x}^{\bar \beta}_{\MLE})   \to s_{\beta_1, \beta_2, \beta_3} \, ,
		\]
		where $(s_{\beta_1, \beta_2, \beta_3}, m_{\beta_1, \beta_2, \beta_3}) $ are maximizers $\psi_{\bar \beta}$. Furthermore, the optimizers of $\psi_{\bar \beta}(s,m)$ are unique up to a sign for $\beta \in \mathcal{D}_{\beta_2} \cap \mathcal{D}_{\beta_3}$.
	\end{lem}
	\begin{proof} 
		This proof follows from an application of the envelope theorem and the fact that the ground state variational formula is the limit of the finite dimensional ground state. We start by characterizing the overlaps $R(\vect{\hat x}_{\MLE}, \vect{x}_0)$. By universality, we know that the square of $R(\vect{\hat x}_{\MLE}, \vect{x}_0)$ and $R(\vect{\hat x^{\bar \beta}}_{\MLE}, \vect{x}_0)$ converge to the same value. 
		\hfill

        \textit{Step 1:} We fix the parameters $\beta_1$ and $\beta_3$ and consider 
		\[
		f_N(\beta_2) = \frac{1}{N} \E \max H_N^{\bar \beta}(x) = \frac{1}{N} \E \max \frac{\beta_1}{\sqrt{N}} \sum_{i \leq j} g_{ij} x_i x_j + \frac{\beta_2}{N} \sum_{i \leq j} x_i^0 x_j^0 x_i x_j - \frac{\beta_3}{2N} \sum_{i \leq j} x_i^2 x_j^2
		\]
		as a function of $\beta_2$ only. By the envelope theorem in ~\cite{Milgrom-enveloppe}  we have that
		\[
		f_N'(\beta_2) = \frac{1}{2} \E R(\vect{\hat x}_{\MLE}, \vect{x}_0)^2 \, ,
		\]
		where  $\hat x_{\MLE}$ denotes a maximizer of $H_N^{\bar \beta}(x)$.\\
        
        \textit{Step 2:} We now need to relate $f_N'$ with the derivative of $\psi$ with respect to $\beta$. Notice that $f_N(\beta_2)$ is convex in $\beta_2$ because it is the pointwise limit in $L$ of convex functions in $\beta_2$ by \eqref{eq:boundsgroundstate},
		\[ \lim_{L \to \infty} \frac{1}{L} F_N(L \beta_2) := \lim_{L \to \infty} \frac{1}{L} F_N(L \beta_1, L \beta_2, L \beta_3 )  = f_N(\beta_2).\]
		Since the derivatives of convex functions converge to the derivative of the limit on all points where the limit is differentiable (sometimes called Griffith's Lemma) we have
		\[
		f_N'(\beta) \to \frac{d}{d\beta_2} \sup_{s,m} \psi_{\bar \beta}(s,m) \, .
		\]
		\textit{Step 3:} It remains to see that the limiting object is characterized by a maximizer of $\psi_{\bar \beta}$. By another application of the envelope theorem we see that
		\[
		\frac{d}{d\beta_2} \sup_{s,m} \psi_{\bar \beta}(s,m) = \frac{ m_*^2 }{2} \, ,
		\]
		where $m_*$ is a maximizer of $\psi(S,M)$. In particular, there can only be at most two maximizers of $\psi(S,M)$ and they are unique up to a sign. We have 
		\[
		\E R(\vect{\hat x}_{\PMLE}, \vect{x}_0)^2 \to m_*^2 \, ,
		\]
		where $m_*$ is a maximizer of $\psi(S,M)$. By Lemma~\ref{lem:conccentrationrestrictedMLE},  we conclude
		\[
		R(\vect{\hat x}_{\PMLE}, \vect{x}_0)^2 \to m^2_* \, ,
		\]
		almost surely. 
		
		The proof for the characterization of the limit of $R(\vect{\hat x}_{\MLE}, \vect{\hat x}_{\MLE})$ is identical, differentiating in $\beta_3$ instead. It uses again that the only dependence of $\psi$ in $\beta_3$ is in the linear last term of the formula. 
	\end{proof}

        The next result shows that the differentiability condition is necessary and sufficient for $\psi_{\bar \beta}$ to have a unique maximizer. 

    \begin{lem}
        We have
		\[
		\{ (\beta_1, \beta_2, \beta_3) \mmm \partial_{\beta_2} f \text{ exists}, \partial_{\beta_3} f \text{ exists} \} =  \{ (\beta_1, \beta_2, \beta_3) \mmm \psi_{\bar \beta} \text{ satisfies Assumption~\ref{hyp:existintro} } \} \, .
		\]
    \end{lem}
    \begin{proof}
    We begin by showing that uniqueness implies differentiability. We first claim that 
    \[
    \tilde f(\beta_1,\beta_2,\beta_3) := f(\beta_1^2, \beta_2, \beta_3)
    \]
    where $f$ was defined in \eqref{eq:limit_temperature} is convex. This follows from the fact that the free energy function $\frac{1}{L}F_{N,\alpha}( (L\beta_1^2, L\beta_2, L\beta_3 ), L\alpha,\epsilon;S,M,v) $ defined in \eqref{eq:constrainedFE} is convex in $(\beta_1^2,\beta_2,\beta_3)$ for every fixed $S,M,v$. Thus, $\tilde f(\beta_1,\beta_2,\beta_3)$ is convex, since it is the pointwise limit of convex functions by Theorem~\ref{th:finitetempcorrected}. Since $\psi_{\beta_1,\beta_2,\beta_3}$ only depends on $M$ through $M^2$, Danskin's Envelope Theorem (see \cite[Proposition~A.22]{Bertsekas_Danskins}) implies that whenever $\psi_{\beta_1,\beta_2,\beta_3}$ satisfies Assumption~\ref{hyp:existintro}, $\tilde f$ is differentiable at $(\sqrt{\beta_1},\beta_2,\beta_3)$. Differentiability of $f$ at $(\beta_1,\beta_2,\beta_3)$ follows immediately from the differentiability of $\tilde f$ at $(\sqrt{\beta_1},\beta_2,\beta_3)$ by the chain rule.

The converse is immediate, because if $\psi$ is differentiable in $\beta_2$ and $\beta_3$, then one immediately obtains uniqueness for maximizing pairs by Danskin's envelope theorem as was shown in part 3 of the proof of Lemma~\ref{lem:condition_unique}.
    \end{proof}
	
We end this section by showing that although in some cases the maximizers of $\psi_{\overline{\beta}}$ may not be unique, the performance of the MLE is still characterized by the maximizers of $\psi$, as stated in Lemma ~\ref{lem:char_limitpoints}.

\begin{proof}[Proof of Lemma ~\ref{lem:char_limitpoints}] 
    We proceed as follows, since $\psi(S,M) < \sup \psi$, there are constants $ \epsilon, \delta >0$ such that if $A_{\delta} = [ S-\delta, S+\delta] \times [M -\delta, M +\delta] \cap \mathcal{C}$, then
    \[
 \sup_{ (s,m) \in A_{\delta} }  \psi + \epsilon <  \sup_{ (s,m) \in \mathcal{C}} \psi \, .
    \]
    By Lemma~\ref{lem:conccentrationrestrictedMLE} one has that: 
    \[
\pP\Big( \frac{1}{N} \abs{ \E \sup_{\Omega_{\delta}(s,m) } H_n(x) - \sup_{\Omega_{\delta}(s,m) } H_n(x) } \geq \frac{\epsilon}{4} \Big) \leq  e^{ -\frac{CN \epsilon^2}{16} } \, , 
    \]
    and so by Theorem ~\ref{thm:main_regular}, we have with  with probability at least $1-2e^{-\frac{1}{16} CN \epsilon^2}$ for any $x \in GS_N$ that: 
    \begin{align*}
 \frac{1}{N} \max_{x \in \Omega_{\delta}(S,M)} H_N &\leq \sup_{(S,M) \in A_{\delta} } \psi(S,M) + \frac{\epsilon}{4} + o(1)
 \\
 &<  \frac{1}{N} H_N(x)  -\frac{\epsilon}{2} + \abs{ \E \frac{1}{N} \max_{x \in \Omega^N} H_N - \sup_{ (s,m) \in \mathcal{C} } \psi(s,m)  } + o(1) \, .
    \end{align*}
    Hence, for $N$ sufficiently large one has that $\text{GS}_N \cap \Omega_{\delta}(S,M) \neq \emptyset$, with probability at most $4e^{-\frac{1}{16} C \epsilon^2 N} $. Consequently: 
    \[
\limsup_{N \to \infty} \frac{1}{N} \log \pP ( \text{GS}_N \cap \Omega_{\epsilon}(S,M) \neq \emptyset ) \leq -  \frac{C \epsilon^2}{16} \, ,
    \]
    which proves the first claim.  The second claim follows from the portmanteau lemma. Suppose that $(S_N(\vx_N),M_N(\vx_N))$ converges weakly to $(\mathscr{S}, \mathscr{M} )$, then for any sufficiently small $\delta$ neighbourhood $U$ of $(S,M)$, such that $\cC_{\beta} \cap \overline{U} $ is empty, the large deviation upper bound from before implies that: 
    \begin{align*}
\pP( (\mathscr{S},\mathscr{M}) \in U) &\leq  \liminf_{N \to \infty} \pP( (S_N(\vx_N),M_N(\vx_N) ) \in U ) 
\\
&\leq \limsup_{N \to \infty} \pP(\text{GS}_N \cap \Omega_{\delta}(S,M) 
\neq \emptyset )  = 0 \, ,
    \end{align*}
    and hence the support of $(\mathscr{S},\mathscr{M})$ is contained in $\cC_{\bar \beta}$.

 Now let us fix a point $(S,M)$ in $\cC_{\bar \beta} $, then by definition of a near maximizer and concentration of $H_N/N$, one has for any fixed $\epsilon >0$, and for $N$ sufficiently large that:   
  \[
  \max_{x \in \Omega_{\epsilon}(S,M) } \frac{H_N}{N} > \psi(S,M) - \epsilon \, , 
  \]
  with probability $1-o(1)$. Consequently there is $ \hbx_N \in \Omega_{\epsilon} (S,M)$ achieving this bound. Let us consider the sequence of overlaps $ (R( \hbx_N,\hbx_N) , R( \hbx_N, \hbx^{0,N} ) )$, we note that this sequence is tight, and so by passing to a subsequence we may assume that $( S_N( \hbx_{N,\epsilon}) , M_N( \hbx_N ) )$ converges almost surely as $N \to \infty$ to some random variables $( \mathscr{S}_{\epsilon},\mathscr{M}_{\epsilon})$ By definition of $\Omega_{\epsilon} $ we have that 
  \[
 \abs{ \mathscr{S}_{\epsilon} - S} \leq \epsilon \qquad \text{and} \qquad \abs{\mathscr{M}_{\epsilon} - M } \leq \epsilon \, ,
  \]
  hence taking $\epsilon \to 0 $ implies that $\mathcal{M}_{\epsilon} \to M$ and $\mathcal{S}_{\epsilon} \to S$ almost surely, finishing the proof.   \qedhere

\end{proof}

 We now state an impossibility result for models when $\beta_2 = 0$ and $\vect{x}^0$ is balanced, i.e. its sample mean converges to $0$. This follows from the following fact about exchangeable independent sums.

        \begin{lem}\label{lem:symmetry}
            Let $\vect{y}^N$ be a triangular array with  bounded entries such that $ \frac{1}{N} \sum_{i = 1}^N y_i^N$ $\to 0$ and let $\bx_N = \{x_{i,N}\}_{i=1}^{N}$ be a triangular array of uniformly bounded exchangable vectors independent of $y$. Then
            \[
            \frac{1}{N} \sum_{i = 1}^N x_{i,N} y_i \to 0 \, ,
            \]
            in probability.
        \end{lem}
        \begin{proof}
            By Markov's inequality, it suffices to show that 
            \[
            \E \bigg( \frac{1}{N} \sum_{i = 1}^N x_i y_i \bigg)^2 \to 0 \, .
            \]
            We have
            \begin{align*}
            \E \bigg( \frac{1}{N} \sum_{i = 1}^N x_i y_i \bigg)^2 = \frac{1}{N^2} \bigg[ \sum_{i,j} \E [x_i x_j] y_i y_j \bigg]
            &= \frac{1}{N^2} \bigg[ a_N \sum_{i = 1}^N y_i^2 + b_N \sum_{1 \leq i \neq j \leq N} y_i y_j \bigg]
            \\&\leq C \bigg( \frac{1}{N}  \sum_{ i = 1 }^N y_i\bigg)^2 + O(N^{-1}) \, ,
            \end{align*}
            where $a_N = \E[x_1]$ and $b_N = \E[x_1 x_2]$.
            The last inequality follows since $\mathbf{y}^N$ and $\bx_N$ are both bounded. By assumption on $\mathbf{y}^N$ the upper bound tends to zero, and we are done. \qedhere 
        \end{proof}

        We now have the following characterization of performance when $\beta_2 = 0$.
        \begin{lem}
            If $\beta_2 = 0$ and $\E_{\pQ} x_0 = 0$, then $\cs(\vect{x}^0, \vect{\hat x}_{\PMLE}) = 0$.
        \end{lem}
        \begin{proof}
            By Theorem~\ref{lem:char_limitpoints} we have that $(S_N(x),M_N(x))$ have limit points in the set $\cC_{\bar \beta}$. It suffices to  show that $\cC_{\bar \beta}$ only contains points of the form $(s,0)$. 

            When $\beta_2 = 0$, we have $H_N^{\bar \beta}(\vect{x})$ defined in \eqref{eq:modelequivmain} does not depend on $\vect{x}_0$ nor $\pQ$. By symmetry, any near maximizer $\hat \vx$ of $H_N^{\bar \beta}(\vect{x})$ has exchangeable bounded entries and is independent of $\vect{x}_0$. By Lemma~\ref{lem:symmetry} it follows that any near maximizer satisfies
            $
            M(\hat{x}) = \frac{1}{N} \sum_{i = 1}^N \hat x_i x_i^0 \to 0 \, ,
            $
            in probability. Since the possible limit points of the overlaps of near maximizers determine $\cC_{\bar \beta}$ by Theorem~\ref{lem:char_limitpoints}, the set $\cC_{\bar \beta}$ only contains points of the form $(s,0)$.
        \end{proof}

	\section{Coarse Equivalence of Pseudo Estimators}\label{sec:modelequiv}

	In this section, we prove Theorem ~\ref{th:robust}, using results proved in Sections ~\ref{sec:univ} and ~\ref{sec:charac}. 
	\begin{proof}[Proof of Theorem~\ref{th:robust}]
		We first consider the case of well-scored models. Given two well-scored loglikelihood functions $g_1$ and $g_2$, we let $\bar \beta(g_1)$ and $\bar \beta(g_2)$ to be the Fisher score parameter vectors corresponding to $g_1$ and $g_2$. Note that if the ratios satisfy:
		\begin{align*}
			\frac{\sqrt{\beta_1(g^1)}}{\sqrt{\beta_1(g^2)}} = \frac{\beta_2(g^1)}{\beta_2(g^2)} =  \frac{\beta_3(g^1)}{\beta_3(g^2)} \, , 
		\end{align*}
		then $\bar \beta(g_1) = C \bar \beta(g_2)$ for some constant $C$. By Lemma~\ref{lem:lemUniv} (or by definition of the gaussian equivalent) this implies that 
		\[
		H_N^{\bar \beta(g^2)}(x) =  C H_N^{\bar \beta(g^1)}(x) \, ,
		\]
		and hence both functions have the same maximizers. Therefore, 
		\[
		M(\vect{\hat x}^{\bar \beta(g_1)}_{\MLE}) = M(\vect{\hat x}^{\bar \beta(g_2)}_{\MLE}) \text{ and }  S(\vect{\hat x}^{\bar \beta(g_1)}_{\MLE}) = S(\vect{\hat x}^{\bar \beta(g_2)}_{\MLE}) \, .
		\]
		On the other hand, by Theorem~\ref{thm:main_regular} and Lemma~\ref{lem:char_limitpoints}, the maximizers satisfy
		\begin{align*}
		M(\vect{\hat x}^{g_1}_{\MLE}) &= M(\vect{\hat x}^{\bar \beta(g_1)}_{\MLE}) \to m(g_1) \quad \text{ and } \quad  S(\vect{\hat x}^{g_1}_{\MLE}) = S(\vect{\hat x}^{\bar \beta(g_1)}_{\MLE}) \to s(g_1) 
		\\
		M(\vect{\hat x}^{g_2}_{\MLE}) &= M(\vect{\hat x}^{\bar \beta(g_2)}_{\MLE}) \to m(g_2) \quad \text{ and } \quad  S(\vect{\hat x}^{g_2}_{\MLE}) = S(\vect{\hat x}^{\bar \beta(g_2)}_{\MLE}) \to s(g_2) \, , 
		\end{align*}
		where $s(g_i), m(g_i)$ are maximizers of $\psi_{\bar \beta(g_1)}$ and $\psi_{\bar \beta(g_2)}$ respectively. We conclude that the set of all limit points of the overlaps coincide so appealing again to Lemma~\ref{lem:char_limitpoints},  we can conclude that $\cC_{\bar \beta(g_1)} = \cC_{\bar \beta(g_2)}$, which is the definition of coarsely equivalent inference tasks.

		To prove part (b) of  Theorem~\ref{th:robust}, note that if $\Omega$ satisfies $\abs{\bx} = C$ for every $\bx \in \Omega$, then the third term in $ H_N^{\bar \beta(g^1)}(x)$ and $ H_N^{\bar \beta(g^2)}(x)$ are constant. Consequently one has for some constants $C,D$ that: 
        \[
C H_N^{\bar \beta(g^1)}(x) + D = H_N^{\bar \beta(g^2)}(x) \, , 
        \]
		and hence $	H_N^{\bar \beta(g^1)}(x)$ and $	H_N^{\bar \beta(g^2)}(x)$ have the same maximizers. The result then follows from the same argument as above.
	\end{proof}
        
        \begin{proof}[Proof of Theorem~\ref{cor:model-equiv-corrected }]
		Apply Lemma~\ref{lem:lemUniv} as above, and apply  Theorem ~\ref{th:corrected}.
        \end{proof}

 {

 \section{Generalizations to Domains on $\R$}\label{sec:fullspace}

 \subsection{Sub Gaussian Signal Vectors}\label{sec:subgauss_signal}
 In this section, we weaken the assumptions on the data model \eqref{eq:conditional_data_dist} and  now explain how to extend the results in the paper to the case that $\vect{x}^0$ is a vector of i.i.d. random variables with subgaussian tails. We assume that
     \begin{equation}
        Y_{ij} \sim \pP_Y(\cdot \given[]  \frac{1}{\sqrt{N}} x^{0,N}_i x^{0,N}_j) \quad \text{for } i\leq j \, , 
    \end{equation}
where $\vect{x}^0$ has subgaussian teails. 

We first show that Gaussian universality holds under this generalized setting. 
\begin{lem}\label{lem:lemUniv2}
		If $g,g_0 \in \cF_0$, then for any $(S,M,v)\in \cC_c$
		\[
		\lim_{N \to \infty} \frac{1}{N} | \cL_N^{g,\epsilon}(S,M,v)  - \cL_N^{\bar \beta,\epsilon}(S,M,v)   | = 0 \, ,
		\]
    with probability $1 - o(1)$ in $\vect{x}^0$. 
\end{lem} 
\begin{proof}
    The proof is largely identical to the proof in Section~\ref{sec:univ}. The only part where the boundedness assumption of $\vect{x}^0$ was in the approximation of the conditional averages with the average with respect to the null model in Step~3 of the proof of Proposition~\ref{prop:universality1}. 

    We emphasize the main difference here. Using the same notation as in the proof of Proposition~\ref{prop:universality1}, we will arrive at
    \begin{align*}
			\mu_{ij} &= \E_{Y}[ \partial_w g(Y_{ij},0) |\bx^0 ] = \int \partial_w g(y,0) e^{g^0(y, w^0_{ij}) } \, dy 
			\\
            &= \int  \partial_w g(y,0) e^{g^0(Y,0)} \Big[1 + \partial_w g^0(y,0) w^0_{ij}  +  \big( (\partial_w g^0(y,0))^2 + \partial_w^{2}g^0(y,0) \big) \frac{(w^0_{ij})^2}{2}
            + O(N^{-1})\Big] \, dy
			\\
            &= \E_{0} \partial_w g(Y,0) + \frac{x_i^0 x_j^0}{\sqrt{N}} \E_{0} \partial_w g(Y,0)g^0_w(Y,0) + \frac{1}{2}\E_0 \partial_wg(Y,0) (\partial_w g^0(y,0))^2 w_{ij}^0 \, ,
	\end{align*}
    where the last equality holds up to a term of order $O(N^{-1})$.
    To bound the second order term, we can apply the Cauchy--Schwarz inequality to see that
    \begin{align*}
    |\E_0 \partial_wg(Y,0) (\partial_w g^0(y,0))^2 w_{ij}^0| &\leq \bigg[ \E_0 (\partial_wg(Y,0))^2 \E_0 (\partial_w g^0(y,0))^4 \bigg]^{1/2}  w_{ij}^0
    \end{align*}
    We conclude that
    \[
    \mu_{ij} =  \beta_4 + \frac{x_i^0 x_j^0}{\sqrt{N}} \beta_2 + O ( (w_{ij}^0)^2 ) + O(N^{-1}).
    \]
    A similar computation will show that
    \[
    \sigma^2_{ij} = \beta_4 + \frac{x_i^0 x_j^0}{\sqrt{N}} \beta_2 + O ( w_{ij}^0 ) + O(N^{-1}) \text{ and } \gamma_{ij} = -\beta_3  O ( w_{ij}^0 ) + O(N^{-1}) .
    \]
    Using these error estimates in the rest of the proof will result in a bound of the form
    \[
		\big| F_N(g;S,M,v) - F_N(\bar\beta;S,M,v) \big| \leq \frac{K}{\sqrt{N}} + \frac{K}{N^{5/2}} \sum_{i,j} ( (x^0_i)^2 (x^0_j)^2 + x^0_i x^0_j )\,
	\]
    where $K$ is some absolute consntant that does not depend on the signal $\vect{x}^0$. Since the entries of $\vect{x}^0$ are subgaussian, an application of Markov's inequality will imply that
    \[
    \pP \bigg(  \sum_{i,j} | (x^0_i)^2 (x^0_j)^2 + x^0_i x^0_j |  > N^{9/4} \bigg) \leq \frac{ (\E (x_i^0)^2 )^2 + (\E|x_i^0|)^2  }{ N^{1/4} } = O(N^{-1/4}).
    \]
    So with probability $1 -  O(N^{-1/4})$, we have that
    \[
    \big| F_N(g;S,M,v) - F_N(\bar\beta;S,M,v) \big| \leq \frac{K}{N^{1/4}} 
    \]
    for some constant $K$. 
\end{proof}

 Next, we show that the signal in the proxy model can be truncated. We start with the following truncation lemma for the constrained set in the optimization:
 \begin{lem}\label{lem:truncation_set} Suppose that $ \epsilon_1 < \epsilon_2 < \epsilon_3$. Let $\vect{x}^{0,K}$ denote the truncation of each entry at level $K$, i.e.
 \[
x_i^{0,K} = \begin{cases}
    x_i^{0} & |x_i^0| \leq K\\
    K & |x_i^0| > K.
\end{cases}
 \]
 There is sufficiently large $K$ such that with probability $1-o(1)$, for all sufficiently large $N$, one has for any $(S,M,v)$ that: 
 \[
 \Omega_{\epsilon_1}(S,M,v) \subset \Omega_{\epsilon_2}^K (S,M,v) \subset \Omega_{\epsilon_3}(S,M,v) \,
 \]
 where $\Omega_{\epsilon_2}^K (S,M,v)$ is the restricted domain defied in \eqref{eq:defnOmegaepsilon} where the $M_N$ overlap is with respect to the truncated signal. 
 \end{lem} 
 \begin{proof}
     It suffices to show that there is a truncation level $K$ so that if $\vect{x}$ is a point in $\Omega_{\epsilon_1}(S,M,v)$, then $x$ is also in $\Omega_{\epsilon_2}^K(S,M,v) $. Note that the conditions on the norm of $\vect{x}$ and the magnetization of $\vect{x}$ are automatically satisfied, so it suffices to control the value of the inner product of $\vect{x}$ with $\vect{x}^{0,K}$ and $\vect{x}^0$.  Write $\vect{x}^{0,K}= \vect{x} + (\vect{x}^{0,K}-\vect{x}) $, then one has:
     \[
 \frac{1}{N} \langle \vect{x}, \vect{x}^{0,K} \rangle = \frac{1}{N} \langle \vect{x}, \vect{x}^0 \rangle + \frac{1}{N} \langle \vect{x},(\vect{x}^{0,K}-\vect{x}^0) \rangle \, ,
     \]
     the first term lies in an $\epsilon_1$ neighbourhood of $M$, so it suffices to show if $K$ is chosen correctly then the whole sum is in an $\epsilon_2$ neighbourhood of $M$ as well. 
     
We can use the Cauchy--Schwarz and Markov's inequality to control the inner product. We have
\begin{equation}\label{eq:truncationtail1}
 \frac{1}{N} \abs{ \langle \vect{x}, \vect{x}^{0,K}-\vect{x} \rangle}  \leq \sqrt{\frac{\|\vect{x}\|^2}{N}} \sqrt{ \frac{\| \vect{x}^{0,K}- \vect{x} \|^2}{N} } 
\end{equation}
 Then by Markov's inequality, we have
 \begin{equation}\label{eq:truncationtail2}
\pP\Big(  \frac{\| \vect{x}^{0,K}-\vect{x} \|^2}{N} > t \Big) \leq \frac{\E \| \vect{x}^{0,K}-\vect{x} \|^2}{Nt} = \frac{\E ( x_1^{K,0}-x^{0}_1 )^2}{t} = \frac{\E ( x^{0}_1- K )^2 \1(x^{0}_1 \geq K)}{t}   
 \end{equation}
 The expected value is sub-Gaussian in $K$, so  we can  take $t = K^{-1}$, which shows our bound with high probability.       

\end{proof}
The main consequence of Lemma~\ref{lem:truncation_set} is that the maximum of $H_N^{\barbeta}$ and the truncated version $H_N^{\barbeta, K}$ are close when optimizing over fixed constraints $(S,M,v)$. Notice that that information parameters $\barbeta$ do not depend on the $\vect{x}^0$ and hence are not affected by the truncation. Furthermore, we have that
\[
\sup_{\vect{x} \in \Omega^N} \bigg| \frac{1}{N} H_N^{\barbeta} -  \frac{1}{N} H_N^{\barbeta, K} \bigg| \leq \frac{1}{N^2} \sum_{i \leq j} ( x_i^0 x_j^0  - x_i^{0,K} x_j^{0,K} ) x_i x_j = \frac{\abs{ \langle \vect{x}, \vect{x}^{0,K}-\vect{x} \rangle} ^2}{N^2}.
\]
Therefore, the estimates in \eqref{eq:truncationtail1} and \eqref{eq:truncationtail2} implies that with probability tending to $1$ as $N \to \infty$, 
\[
\sup_{\vect{x} \in \Omega^N} \bigg| \frac{1}{N} H_N^{\barbeta} -  \frac{1}{N} H_N^{\barbeta, K} \bigg| \leq O( K^{-1} ).
\]

From here, we can compute the limiting variational formula for $H_N^{\barbeta,K}$
with fixed constraints $(S,M,v)$ using Theorem~\ref{th:corrected} since $H^K$ has bounded signal. Then we can send the truncation $K \to \infty$ in the very last step to recover the untruncated value.

 }

{

\subsection{Least squares over $\R^N$}\label{sec:LSFullSpaceProof}

In this section, we prove the Theorem ~\ref{thm:least-squares-all-space}. We will see that these results are simpler than in the compact parameter space setting because the problem reduces to optimization over balls.

\begin{proof}[Proof of Theorem ~\ref{thm:least-squares-all-space}]
\hfill\\\\
\textit{Step 1:} We begin by reducing the model to a Gaussian one. 
    First note that by directly expanding the least squares estimator we have that: 
    \[
\mathcal{L}_N(Y,x)-\mathcal{L}_N(Y,0)=  \frac{\lambda}{\sqrt{N}} \sum_{i\leq j}^{N} Y_{ij} x_i x_j - \frac{\lambda^2}{4} N R_{1,1}^2 \, ,
    \]
so if one can replace $Y_{ij}$ with gaussian random variables of the same mean and variance, then standard results about the BBP transition will give the full characterization. 

By the assumption on $Y$, and on the signal $\vect{x}$ there is a uniform bound on the operator norm of $Y$, of the form $\norm{Y}_{op} \leq C \sqrt{N}$ for some constant $C$ with probability $1-o(1)$. Consequently one has some value of $R>0$ such that the maximum of $\mathcal{L}_{N}(Y,x)$ is achieved on the ball of radius $R \sqrt{N}$. 

A quick calculation like in step 3 of the proof of Proposition~\ref{prop:universality1} of  shows that conditionally on the value of $\vect{x}^{0,N}$ one has that $\E Y_{ij} = \frac{\beta_2}{\sqrt{N}} x_i^0 x_j^0$, and the variance satisfies $\Var(Y_{ij}) = \E_{\pP_0} Y_{ij}^2 + o(N^{-1} )$. Thus we may rearrange and obtain:
\[
\mathcal{L}_N(Y,x) =  \frac{\sqrt{\beta}_1}{\sqrt{N}} \sum_{i \leq j } \frac{1}{(\E_{\pP_{0}}Y^2)^{1/2}} ( Y_{ij} - \E Y_{ij}) x_i x_j  +  \frac{\beta_2}{2} NR_{1,0}^2 - \frac{\beta_3}{4} N R_{1,1}^2  + O(1) \, ,
\]
    and from here one follows the same interpolation strategy as in Lemma~\ref{lem:univfinitetemp} to show the equivalence of the maximum value over the ball of radius $R \sqrt{N}$. 
\\\\
\noindent\textit{Step 2:} Now let us determine the maximum under the case where the disorder is given by centred standard normal random variables. In this case one can write out the model as: 
\[
H^{\barbeta}_N (x) =  \frac{1}{2} \langle x , Z x \rangle - \frac{\beta_3}{4N} \norm{x}^4 \, ,
\]
where $Z = \sqrt{\beta_1} G + \frac{\beta_2}{N} \vect{x}^{0,N} (\vect{x}^{0,N})^T $, and $G$ is a GOE$(N)$ matrix. The maximum is determined by taking the top eigenvalue of $Z$, and maximizing over the norm.   

The behavior of the top eigenvalue and correlation with the signal is well known for a large class of Wigner matrices, see for example \cite{renfrew_soshnikov_spiked}. Applying these results to our setting, we have
\[
\max_{ \|x\| \leq R \sqrt{N} } \langle x , Z x \rangle = \lambda_{max}(Z) R^2 N \, ,
\]
where
\[
\lambda_{max}(Z) =
\begin{cases}
     \beta_2 \E_{\pQ}[ x_0^2 ] + \frac{\beta_1}{ \beta_2 \E_{\pQ}[ x_0^2 ] }    & \frac{\beta_2 \E_{\pQ}[ x_0^2 ]}{\sqrt{\beta_1}} > 1\\
    2& \frac{\beta_2 \E_{\pQ}[ x_0^2 ]}{\sqrt{\beta_1}} < 1
\end{cases} \, ,
\]
and then optimizing over $R \geq 0$ one has the optimal $R = \frac{2 \lambda_{max}}{\beta_3}$, with maximum value $N \lambda^2 \beta_3^{-1}$. The formula for the cosine similarity then follows from the fact that the overlap between the top eigenvalue and the signal.

\end{proof}
}

	\section{Properties of the Ruelle Probability Cascades }\label{sec:RPC}

	For a textbook introduction to the Ruelle probability cascades, we refer to \cite[Chapter~2]{PBook}. In this section, we recall only the essentials to understand the notation in the Appendices, and remind readers of its connection with the Parisi PDE. The Ruelle probability cascades are a measure on a Hilbert space indexed by $\N^r$ parameterized by sequences
	\begin{align}\label{eq:zetaseq}
	\zeta_{-1} &= 0 < \zeta_0< \dots < \zeta_{r-1} < 1\\
	\label{eq:Qseq1}
		0 &= Q_0 \leq Q_1 \leq \dots \leq Q_{r-1} \leq Q_r =  S.
	\end{align}
	
	The weights of the Ruelle probability cascades is indexed by $\N^r$, the leaves of the infinite rooted tree with depth $r$ encoded by the sequence of parameters $\zeta$. Every leaf of the tree $\alpha = (n_1, \dots, n_r) \in \N^r$ can be encoded by a path along the vertices,
	\[
	\alpha_{|1} = (n_1), ~\alpha_{|2} = (n_1,n_2),~ \dots,~ \alpha_{|r-1} = (n_1,n_2, \dots, n_{r-1}),~ \alpha = \alpha_{|r} =  (n_1, \dots, n_r)
	\]
	with the convention that $\alpha_{|0} = \emptyset$ is the root of the tree, and $k \leq r$ denotes the distance from the vertex $\alpha_{|k} \in \N^k$ to the root. Each vertex $\beta_{|k} = (n_1, \dots, n_{k-1}, n_k )$ of the tree will be associated with a random variable $u_{\beta_{|k}}$ defined as follows: Let $\beta_{|k - 1} = (n_1, \dots, n_{k-1})$ denote the parent of $\beta_{|k}$ and let
	\[
	u_{(\beta_{|k - 1} , 1)} > u_{(\beta_{|k - 1},2)} > \dots > u_{(\beta_{|k - 1},n_k)} > \dots.
	\]
	be the points from a Poisson process with mean measure $\zeta_{k-1} x^{-1 - \zeta_{k-1}}$ arranged in decreasing order, and define
	\[
	u_{\beta_{|k}} = u_{ (n_1, \dots, n_{k-1}, n_k ) } = u_{(\beta_{|k - 1},n_k)}.
	\]
	We further assume that these points are generated independently for different parent vertices. For each leaf $\alpha \in \N^r$, the weights of the Ruelle probability cascades $v_\alpha$ is the product of these points along the path from the root to the leaf:
	\[
	v_\alpha = \frac{u_{\alpha_{|1}} \cdots u_{\alpha_{|r}}}{ \sum_{\beta\in\mathbb N^{r}} u_{\beta_{|1}} \cdots u_{\beta_{|r}}  }.
	\]
	
	We consider the Gaussian processes $Z(\alpha)$ and $Y(\alpha)$ indexed by points on the infinite tree $\N^r$ with covariances
	\[
	\E Z(\alpha^1) Z(\alpha^2) = Q_{\alpha^1 \wedge \alpha^2} \quad \E Y(\alpha^1)  Y(\alpha^2)  = \frac{1}{2} Q^2_{\alpha^1 \wedge \alpha^2}.
	\]
	The notation $\alpha^1 \wedge \alpha^2$ denotes the least common ancestor of the paths leaves $\alpha^1$ and $\alpha^2$ of the infinite tree indexed by $\N^r$,
	\[
	\alpha^1 \wedge \alpha^2 = \min\Big\{0 \leq j \leq r \mmm \alpha_{|1}^1 = \alpha_{|1}^2, \dots, \alpha_{|j}^1 = \alpha_{|j}^2, \alpha_{|j + 1}^1 \neq \alpha_{|j +1}^2  \Big\}
	\]
	
	These averages with respect to the Ruelle probability cascades variable $\alpha$ can be computed using the following recursive formulation from, for example \cite[Theorem 2.9]{PBook}.
	\begin{lem}[Averages with Respect to the Ruelle Probability Cascades ] \label{lem:RPCavg}
		Let $C: \R \to \R$ be an increasing non-negative function. Suppose that there exists a Gaussian process $g(\alpha)$ by $\alpha \in \N^r$ with covariance
		\[
		\E g(\alpha^1) g(\alpha^2) = C( Q_{\alpha^1 \wedge \alpha^2} )
		\]
		independent of $v_\alpha$. For a function $f: \R \to \R$ we define
		\[
		X_r = f\Big( \sum_{k = 1}^r ( C( Q_k ) - C( Q_{k - 1} )  )^{1/2} z_{ k}  \Big) \qquad X_{p} = \frac{1}{\zeta_p} \log \E_{z_{k + 1}} e^{\zeta_p X_{p + 1}} \quad \text{for $0 \leq p \leq r-1$}
		\]
		where $z_k$ are iid standard Gaussians. If $\E e^{\zeta_{r-1}  X_r } < \infty$ then
		\[
		\E \log \sum_{\alpha} v_\alpha e^{ f ( g(\alpha) ) } = X_0.
		\]
		The average on the outside is over the randomness in the Gaussian processes and the random measure $v_\alpha$.  
	\end{lem}

	This is applied in  Section~\ref{sec:devvarI} in the following way. We start by defining recursively the random variables $X_r, X_{r - 1}, \dots, X_0$ that depend on $x^{0}$, the sequences \eqref{eq:zetaseq} and \eqref{eq:Qseq1}, and real parameters $\lambda,\mu$.  
	Let $X_r$ be the random variable
	\[
	X_r = \log  \int e^{\beta \sum_{j = 1}^r z_i x  + \lambda x^2 + \mu x x^0  } \, d \pP_X (x)
	\]
	where $z_j$ are Gaussian random variables with covariance
	\[
	\Var(z_j) = Q_j - Q_{j-1}
	\]
	and $x^0$ is an independent random variable with distribution $\pP_0$.
	We define recursively for $0 \leq p \leq r-1$ the random variables 
	\begin{equation}\label{recur}
		X_j = \frac{1}{\zeta_j} \log \E_{z_{j + 1}} e^{\zeta_j X_{j + 1}}.
	\end{equation}
	Then Lemma~\ref{lem:RPCavg} implies that
	\[
	\lim_{N\to \infty} \frac{1}{N} \E \log \sum_{\alpha} v_\alpha \int e^{\beta \sum_{i =1}^N Z_i(\alpha) x_i  + \lambda x_i^2 + \mu x_i x_i^0  } \, d \pP_X (x) = \E_{\pQ} X_0(x^0).
	\]
	
	By continuity, one can represent the averages with respect to the Ruelle probability cascades as the solution to the Parisi PDE, which is the form we are using in this work. 	The details of this reduction can be found in \cite[Section~4.1]{PBook}. Consider a distribution function $\zeta(t)$ such that
	\[
	\zeta(t) = \zeta_p \qquad Q_p \leq t < Q_{p + 1}
	\]
	for $p = 0, \dots, r$. The following Lemma from \cite[Theorem~6]{Arguin_RPC} shows that we can approximate the discrete distributions.

	\begin{lem}
	For every discrete distribution function $\zeta(t)$ encoded by the parameters \eqref{eq:zetaseq} and \eqref{eq:Qseq1} one has
	\[
	\lim_{N\to \infty} \frac{1}{N} \E \log \sum_{\alpha} v_\alpha \int e^{\beta \sum_{i =1}^N Z_i(\alpha) x_i  + \lambda x_i^2 + \mu x_i x_i^0  } \, d \pP_X (x) = \E_{\pQ} X_0(x^0) = \Phi_{\zeta}(0,0).
	\]
	\end{lem}
	Furthermore, by continuity \cite[Lemma~4.1]{PBook} we can extend this result to any distribution function $\zeta(t)$, which gives us the representation in \eqref{eq:parisipdefinite}.
	
We also use the following upper bound of the Ruelle probability cascades of a partition with respect to its maximum value on each partition. The proof can be found in \cite[Lemma~6]{PPotts}. 

\begin{lem}[Upper Bound of the Ruelle Probability Cascades]\label{lem:upbdRPC} 
	Let $g(\alpha)$ be a Gaussian process indexed by $\alpha \in \N^r$ with covariance	\[
	\E g(\alpha^1) g(\alpha^2) = C( Q_{\alpha^1 \wedge \alpha^2} )
	\]
	independent of $v_\alpha$. If $A_j: \R \to \R$ are positive functions of the same Gaussian process $g(\alpha)$ for $1 \leq j \leq n$ then
	\[
	\E \log \sum_{\alpha \in \N^r} v_\alpha \sum_{j \leq n} A_j(g(\alpha)) \leq \frac{\log n}{\zeta_0} + \max_{j\leq n} \E \log \sum_{\alpha \in \N^r} v_\alpha A_j(g(\alpha)),
	\]
	where $\zeta_0 > 0$ is the smallest point in the sequence \eqref{eq:zetaseq}.
\end{lem}

Our final result is a continuity statement that allows one to approximate the Gibbs measure with one associated with the Ruelle Probability Cascades. 

\begin{lem}[Continuity of the Lower Bound with Respect to the Overlaps]\label{lem:continuityarray}
	Let $\langle \cdot \rangle$ be the average with respect to some non-random Gibbs measure $\mathbb{G}$ on the sphere with radius $\sqrt{S}$ in some Hilbert space $H$. Consider the Gaussian processes $Z(\bs)$ and $Y(\bs)$ indexed by points $\bs$  in $H$ with covariances
	\[
	\E Z(\bs^1) Z(\bs^2) = \langle \bs^1 , \bs^2 \rangle \qquad \E Y(\bs^1) Y(\bs^2) = \frac{\langle \bs^1 , \bs^2 \rangle}{2} 
	\]
	Let $(S,M,v)$ be a point with finite entropy, i.e.  there exists a finite constant $c$ independent of $k$ and $\epsilon$ such that for $k$ large enough, $\pP_X^{\otimes k}(\Omega_\epsilon(S,M,v))\ge e^{-ck}$ uniformly for all $\by_0$ with limiting empirical distribution $\pQ$. Then for all sufficiently large $n$ the functionals
	\begin{align*}
	f_n^Z(S,M,v) = \frac{1}{n} \E_{Z} \log \bigg\langle \int_{\Omega_{\epsilon}(S,M,v)} e^{\sum_{i = 1}^n Z_i(\bs) y_{i} } \, d \pP_X^{\otimes n}(\by) \bigg\rangle
	\end{align*}
	where $Z_i$ are independent copies of $Z$ and
	\[
	f_n^Y = \frac{1}{n} \E_{Y}\log \Big\langle e^{ \sqrt{n} \beta Y(\bs)  } \Big\rangle \, ,
	\]
	are continuous functionals of the distribution of the overlap array $(\bx^\ell \cdot \bx^{\ell'})_{\ell,\ell' \geq 1}$ under $ \mathbb{G}^{\otimes \infty}$. In particular, for any $\eta>0$ there exists  a finite integer number $K(\eta)$ so that  these functionals can be approximated within $\eta$ by a continuous function of the finite array $(\bx^\ell \cdot \bx^{\ell'})_{1 \leq \ell,\ell' \leq K(\eta)}$ uniformly over all possible choices of Gibbs measures $\mathbb{G}$ and all $\by^0$ limiting empirical distribution $\pQ$.
\end{lem}

\begin{proof}
The proof of this argument is standard and can be found in several variants such as \cite[Theorem~1.4]{PBook}, \cite[Lemma~8]{PPotts}, \cite[Lemma~5.11]{nonbayes}. The key step in this proof is the observation that as a function of $x$, both functions $g_1: \R^n \to \R$ and $g_2: \R \to \R$ given by
\[
g_1(\bx) = \int \1( |S_N(\by)-S|\le \epsilon, |M_N(\by)-M|\le \epsilon, |\bar \by -v|\le \epsilon ) e^{\sum_{i = 1}^n x_i y_{i} } \, d \pP_X^{\otimes n}(\by)
\]
and
\[
g_2(x) = e^{ \sqrt{n} \beta x }
\]
is well-defined and continuous in $\bx$ and $x$ because $\by$ takes values in a compact set. By Gaussian concentration, we can truncate the logarithm and $g_1$ and $g_2$ to approximate $	f_n^Z(S,M,v)$ and $f_n^Y$ by polynomials of independent copies of the Gaussian processes $Z_i(\bs^\ell)$ and $Y(\bs^\ell)$ up to an arbitrary $\eta> 0$ error. The expected value of such polynomials are continuous function of finitely many entries of the matrix $(\bx^\ell \cdot \bx^{\ell'})_{1 \leq \ell,\ell' \leq K}$ for some constant $K$ that depends on $S,M,v,\eta,\epsilon$. 
\end{proof}

\end{appendix}
\end{document}